\documentclass{amsart}
  \usepackage{amsrefs}
  \usepackage{amsmath}
  \usepackage{amssymb}
  \usepackage{amsthm}
  \usepackage{booktabs}
  \newcommand{\ml}{\mathcal}
  \newcommand{\mb}{\mathbb}
  \DeclareMathOperator{\non}{non}
  \DeclareMathOperator{\lin}{lin}
  \DeclareMathOperator{\intt}{int}
  \DeclareMathOperator{\extt}{ext}
  \DeclareMathOperator{\bdd}{bdd}

  \newtheorem{theorem}{Theorem}[section]
  \newtheorem{prop}{Proposition}[section]
  \newtheorem{lemma}{Lemma}[section]
  \newtheorem{coro}{Corollary}[section]
  \newtheorem{remark}{Remark}[section]

\begin{document}                                                 
    \title[Asymptotic behaviors for Blackstock's model]{Asymptotic behaviors for Blackstock's model of thermoviscous flow}                                 
    \author[W. Chen]{Wenhui Chen}                                
    \address[W. Chen]{School of Mathematical Sciences, Shanghai Jiao Tong University, 200240 Shanghai, China}
     \email{wenhui.chen.math@gmail.com}       
        
            \author[R. Ikehata]{Ryo Ikehata}                                
        \address[R. Ikehata]{Department of Mathematics, Division of Educational Sciences, Graduate School of Humanities and Social Sciences, Hiroshima University, 739-8524 Higashi-Hiroshima, Japan}                                    
        \email{ikehatar@hiroshima-u.ac.jp}    
        
         \author[A. Palmieri]{Alessandro Palmieri}                                
        \address[A. Palmieri]{Department of Mathematics, University of Pisa, 56127 Pisa, Italy}                                    
        \email{alessandro.palmieri.math@gmail.com}                                 
    \date{July 27, 2021}                                       
    \thanks{The first author was supported by the China Postdoctoral Science Foundation (Grant No. 2021T140450 and No. 2021M692084). The second author was supported in part by Grant-in-Aid for scientific Research (C) 20K03682 of JSPS.}                                     
    \begin{abstract}
We study a fundamental model in nonlinear acoustics, precisely, the general Blackstock's model (that is, without Becker's assumption) in the whole space $\mathbb{R}^n$. This model describes nonlinear acoustics in perfect gases under the irrotational flow. By means of the Fourier analysis we will derive $L^2$ estimates for the solution of the linear homogeneous problem and its derivatives. Then, we will apply these estimates to study three different topics: the optimality of the decay estimates in the case $n\geqslant 5$ and the optimal growth rate for the $L^2$-norm of the solution for $n=3,4$; the singular limit problem in determining the first- and second-order profiles for the solution of the linear Blackstock's model with respect to the small thermal diffusivity; the proof of the  existence of global (in time) small data Sobolev solutions with suitable regularity for a nonlinear Blackstock's model.    
\end{abstract}                
    \maketitle  
    
    \section{Introduction} 
    \subsection{Background of Blackstock's model} \label{Subsect intro}
    It is well-known that in order to characterize the propagation of sound in thermoviscous fluids, mathematical models in the form of wave equations with viscoelastic damping (or the so-called \emph{Kuznetsov's equation}) arise in the study of nonlinear acoustics, which are widely applied in medical  and industrial uses of high-intensity ultra sound, for instance, lithotripsy, thermotherapy, sonochemistry or ultrasound cleaning. There are several mathematical models to describe nonlinear acoustics phenomena (see \cite{Kaltenbacher-Thalhammer-2018,Nikoli-Said-Houari} and references therein). In particular, one of the fundamental models in nonlinear acoustics that is the so-called \emph{Blackstock's model}, has been established in the pioneering work \cite{Blackstock-1963}, and the detailed introduction with its applications has been provided in \cite{Hamilton-Blackstock-1998}. 
    \medskip
    
    \noindent\textbf{Description of Blackstock's model:} First of all, let us introduce some notations for physical quantities in nonlinear acoustics and auxiliary abbreviations to be used throughout this paper by Table \ref{Table_1}.
    \begin{table}[http]
    	\centering	
    	\caption{Notations for physical quantities and auxiliary abbreviations}
    	\begin{tabular}{ll}
    		\toprule
    		Quantity & Notation   \\
    		\midrule
    		Bulk viscosity  & $\mu_{\mathrm{B}}\geqslant0$\\
    		Shear viscosity  & $\mu_{\mathrm{V}}\geqslant0$ \\
    		Thermal conductivity & $\mathrm{K}>0$\\
    		Specific heat & $c_{\mathrm{V}}>0$ \\
    		Specific gas constant & $R>0$\\
    		Speed of sound & $c_0>0$\\
    		Viscosity number & $b=4/3+\mu_{\mathrm{B}}/\mu_{\mathrm{V}}$\\
    		Kinematic viscosity & $\bar{\nu}>0$\\
    		Thermal diffusivity & $\bar{\kappa}>0$\\
    		Prandtl number & $\mathrm{Pr}=\bar{\nu}/\bar{\kappa}$\\
    		Ratio of specific heats & $\gamma\in(1,5/3]$\\
    		Diffusivity of sound & $\bar{\delta}=\bar{\nu}(b+(\gamma-1)/\mathrm{Pr})$\\
    		Modified kinematic viscosity & $\nu=\bar{\nu}/c_0^2$\\
    		Modified diffusivity of sound & $\delta=\bar{\delta}/c_0^2$\\
    		Modified thermal diffusivity & $\kappa=\bar{\kappa}/c_0^2$\\
    		\bottomrule
    		\multicolumn{2}{l}{\emph{$*$The notations with ``modified'' are related to speed of sound.}}
    	\end{tabular}
    	\label{Table_1}
    \end{table}
    
    \noindent In thermoviscous fluids, by combining \emph{the conservation of mass}
    \begin{align*}
    	\frac{\mathrm{D}\rho}{\mathrm{D}t}+\nabla\cdot(\rho\mathbf{u})=0,
    \end{align*}
    with the mass density $\rho=\rho(t,x)$ and the acoustic velocity $\mathbf{u}=\mathbf{u}(t,x)$, where $\frac{\mathrm{D}}{\mathrm{D}t}:=\partial_t+\mathbf{u}\cdot\nabla$ denotes the material derivative;  \emph{the conservation of momentum}
    \begin{align*}
    	\rho\frac{\mathrm{D}\mathbf{u}}{\mathrm{D}t}+\nabla p=\mu_{\mathrm{V}}\Delta\mathbf{u}+\left(\mu_{\mathrm{B}}+\frac{\mu_{\mathrm{V}}}{3}\right)\nabla(\nabla\cdot\mathbf{u})
    \end{align*}
    with the thermodynamic pressure $p=p(t,x)$;  \emph{the conservation of energy with Fourier's law of heat conduction}
    \begin{align*}
    	\rho c_{\mathrm{V}}\left(\frac{\mathrm{D}T}{\mathrm{D}t}+(\gamma-1)\nabla\cdot\mathbf{u}\right)=\mathrm{K}\Delta T+\mu_{\mathrm{B}}(\nabla\cdot\mathbf{u})^2+\frac{1}{2}\mu_{\mathrm{V}}\left(\frac{\partial u_i}{\partial x_j}\frac{\partial u_j}{\partial x_i}-\frac{2}{3}\delta_{ij}\frac{\partial u_k}{\partial x_k}\right)^2
    \end{align*}
    with the temperature $T=T(t,x)$, where $\delta_{ij}$ is the Kronecker delta; \emph{the state equation for a perfect gas}
    \begin{align*}
    	p=R\rho T;
    \end{align*}
    and using Lighthill scheme of approximation procedures (only the terms of first- and second-order can be retained with small perturbation around the equilibrium state), one can describe nonlinear acoustics in perfect gases under the irrotational flow
    \begin{align*}
    	\nabla\times\mathbf{u}=0\ \ \mbox{so that}\ \ \mathbf{u}=-\nabla \bar{\psi},
    \end{align*}
    where the scalar unknown $\bar{\psi}=\bar{\psi}(t,x)\in\mb{R}$ denotes the acoustic velocity potential. Applying the approximation scheme set down by D.T. Blackstock \cite{Blackstock-1963} in 1963, a general form of Blackstock's model (for example, equations (4) or (6) in the recent paper \cite{Brunnhuber-Jordan-2016}) can be described by
    \begin{align}\label{Non-linear_Blackstock}
    	&\left(\partial_t-\frac{\bar{\nu}}{\mathrm{Pr}}\Delta\right)(\bar{\psi}_{tt}-c_0^2\Delta\bar{\psi}-\bar{\delta}\Delta\bar{\psi}_t)-\frac{\bar{\nu}}{\mathrm{Pr}}(\bar{\delta}-b\gamma\bar{\nu})\Delta^2\bar{\psi}_t\notag\\
    	&\qquad=\partial_{t}\left(\partial_t|\nabla\bar{\psi}|^2+(\gamma-1)\bar{\psi}_t\Delta\bar{\psi}\right).
    \end{align} For more detailed derivation of the model \eqref{Non-linear_Blackstock}, we refer to Section 1 of \cite{Brunnhuber-2015}.
    
    For the sake of simplicity, we apply the transform $\psi(t,x):=\bar{\psi}(t,c_0x)$. It immediately follows the equivalent form of \eqref{Non-linear_Blackstock} that is
    \begin{align}\label{Non-linear_Blackstock_2}
    	&\left(\partial_t-\frac{\nu}{\mathrm{Pr}}\Delta\right)(\psi_{tt}-\Delta\psi-\delta\Delta\psi_t)-\frac{\nu}{\mathrm{Pr}}(\delta-b\gamma\nu)\Delta^2\psi_t\notag\\
    	&\qquad=c_0^{-2}\partial_{t}\left(\partial_t|\nabla\psi|^2+(\gamma-1)\psi_t\Delta\psi\right),
    \end{align}
    where the parameters were introduced  in Table \ref{Table_1}.
    \medskip
    
    \noindent\textbf{Nonlinear effects in acoustics:} Basing on the assumption that the shape of the acoustic field is ``close'' to that of a plane wave \cite{Coulouvrat-1992}, one encounters the approximation $\Delta\psi\approx\psi_{tt}$ from the so-called \emph{substitution corollary}, as allowed under the weakly nonlinear scheme. The term $\Delta\psi$ can be replaced by $\psi_{tt}$ so that the nonlinearity of \eqref{Non-linear_Blackstock_2} can be approximated by
    \begin{align}\label{Nonlinear_01}
    	\partial_{t}\left(\partial_t|\nabla\psi|^2+(\gamma-1)\psi_t\Delta\psi\right)\approx\partial_t^2\left(|\nabla\psi|^2+\frac{\gamma-1}{2}(\psi_t)^2\right),
    \end{align}
    which leads to the so-called \emph{Blackstock-Crighton equation}, i.e. 
    \begin{align}\label{Blackstock-Crighton}
    	&\left(\partial_t-\frac{\nu}{\mathrm{Pr}}\Delta\right)(\psi_{tt}-\Delta\psi-\delta\Delta\psi_t)-\frac{\nu}{\mathrm{Pr}}(\delta-b\gamma\nu)\Delta^2\psi_t\notag\\
    	&\qquad=c_0^{-2}\partial_t^2\left(|\nabla\psi|^2+\frac{\gamma-1}{2}(\psi_t)^2\right),
    \end{align}
    for example, \cite{Brunnhuber-2015,Brunnhuber-2016,Brunnhuber-Kaltenbacher-2014,Brunnhuber-Meyer-2016,Kaltenbacher-Thalhammer-2018}. Actually, a natural idea is reversing the procedure to replace $\psi_{tt}$ by $\Delta\psi$, and this procedure has been employed in Kuznetsov's equation \cite{Fritz-Nikolic-Wohlmuth-2018} already. Strongly motivated by those reasons, using $\psi_{tt}\approx\Delta\psi$ we reformulate the nonlinear term \eqref{Nonlinear_01} in the following way:
    \begin{align}\label{Nonlinear_02}
    	&\partial_{t}\left(\partial_t|\nabla\psi|^2+(\gamma-1)\psi_t\Delta\psi\right)\notag\\
    	&\qquad=2|\nabla\psi_t|^2+2\nabla\psi\cdot\nabla\psi_{tt}+(\gamma-1)\psi_{tt}\Delta\psi+(\gamma-1)\psi_t\Delta\psi_t\notag\\
    	&\qquad\approx2|\nabla\psi_t|^2+2\nabla\psi\cdot\nabla\Delta\psi+(\gamma-1)(\Delta\psi)^2+(\gamma-1)\psi_t\Delta\psi_t.
    \end{align}
    Then, combining \eqref{Non-linear_Blackstock_2} and \eqref{Nonlinear_02}, the nonlinear equation, whose investigation is an aim of this paper, becomes
    \begin{align}\label{Pre-Pre-Nonlinear}
    	&\left(\partial_t-\frac{\nu}{\mathrm{Pr}}\Delta\right)(\psi_{tt}-\Delta\psi-\delta\Delta\psi_t)-\frac{\nu}{\mathrm{Pr}}(\delta-b\gamma\nu)\Delta^2\psi_t\notag\\
    	&\qquad=c_0^{-2}\left(2|\nabla\psi_t|^2+2\nabla\psi\cdot\nabla\Delta\psi+(\gamma-1)(\Delta\psi)^2+(\gamma-1)\psi_t\Delta\psi_t\right),
    \end{align}
    that has never been considered in the literature in the initial value problems' framework (to the best of the authors' knowledge), even for the linearized Cauchy problem.  Concerning some studies on linear or nonlinear Blackstock's model under Becker’s assumption (cf. next subsection), we refer the interested readers to \cite{Brunnhuber-Kaltenbacher-2014,Brunnhuber-2015,Brunnhuber-2016,Brunnhuber-Meyer-2016,Kaltenbacher-Thalhammer-2018,Celik-Kyed-2019,Gambera-Lizama-Prokopczyk-2021} for Dirichlet or Neumann boundary value problems.
    \subsection{Becker's assumption}
    However, most of the above quoted recent papers consider a special case for Blackstock's model, i.e. the third-order (with respect to $t$) nonlinear partial differential equation \eqref{Blackstock-Crighton} under the so-called \emph{Becker's assumption} \cite{Becker-1922,Morduchow-Libby-1949,Hayes-1960}. Becker's assumption means that the fluid under the consideration is a monatomic perfect gas. From the mathematical viewpoint, it says $\mu_{\mathrm{B}}=0$ and
    \begin{align}\label{Becker_Assum}
    	b+\frac{\gamma-1}{\mathrm{Pr}}=b\gamma\ \ \mbox{if and only if}\ \ \delta-b\gamma\nu=0.
    \end{align}
    In other words, the condition \eqref{Becker_Assum} permits us to do further factorization of the linear part of Blackstock's model (the last term on the left-hand side of \eqref{Non-linear_Blackstock_2} is vanishing) such that the heat operator $\partial_t-\frac{\nu}{\mathrm{Pr}}\Delta$ acts on the viscoelastic damped wave operator $\partial_t^2-\Delta-\delta\Delta\partial_t$. With the help of Becker's assumption, the consideration of Blackstock's model becomes easier, especially, estimates of solutions to the linear problem (see the explanation in Subsection \ref{Subsec_Inhomo} later). But the special case of a monatomic perfect gas is still not satisfactory from a physicist's point of view since its bulk viscosity coefficient is zero \cite{Pierce}, and the monatomic perfect gas assumption is restricted for the consideration of nonlinear acoustic, which do not include polyatomic gases, etc. Therefore, a natural question is to understand Blackstock's model without Becker's assumption. We will partly answer this question in the whole space $\mb{R}^n$.
    %
    \subsection{Main purposes of the paper}
    Let us turn to the main aims of this paper. Recalling from Subsection \ref{Subsect intro} the definitions of the modified diffusivity of sound $\delta$ and of the Prandtl number $\mathrm{Pr}$, we can rewrite \eqref{Pre-Pre-Nonlinear}. So, in the present paper we will consider the higher-order model in nonlinear acoustics
    \begin{align}\label{Blackstock_Eq-New-Nonlinear}
    	\begin{cases}
    		(\partial_t-\kappa\Delta)(\psi_{tt}-\Delta\psi-\delta\Delta\psi_t)+\kappa(\gamma-1)(b\nu-\kappa)\Delta^2\psi_t=F(\psi;\partial_t,\nabla),\\
    		\psi(0,x)=\psi_0(x),\ \psi_t(0,x)=\psi_1(x),\ \psi_{tt}(0,x)=\psi_2(x),
    	\end{cases}
    \end{align}
    for $x\in\mb{R}^n$, $t>0$, where $F(\psi;\partial_t,\nabla)$ denotes the nonlinear term such that
    \begin{align*}
    	F(\psi;\partial_t,\nabla):=c_0^{-2}\left(2|\nabla\psi_t|^2+2\nabla\psi\cdot\nabla\Delta\psi+(\gamma-1)(\Delta\psi)^2+(\gamma-1)\psi_t\Delta\psi_t\right),
    \end{align*}
    and its corresponding linearized problem with vanishing right-hand side
    \begin{align}\label{Linear_Blackstock_Eq}
    	\begin{cases}
    		(\partial_t-\kappa\Delta)(\psi_{tt}-\Delta\psi-\delta\Delta\psi_t)+\kappa(\gamma-1)(b\nu-\kappa)\Delta^2\psi_t=0,\\
    		\psi(0,x)=\psi_0(x),\ \psi_t(0,x)=\psi_1(x),\ \psi_{tt}(0,x)=\psi_2(x),
    	\end{cases}
    \end{align}
    for $x\in\mb{R}^n$, $t>0$. For the sake of readability, we recall the relation between these parameters
    \begin{align*}
    	\delta=\nu\left(b+\frac{\gamma-1}{\mathrm{Pr}}\right)=b\nu+(\gamma-1)\kappa
    \end{align*}
    and their ranges $0<\kappa\ll 1$, $\gamma\in(1,5/3]$, $b>0$ as well as $\nu>0$. In reality, the modified thermal diffusivity $\kappa=\bar{\kappa}/c_0^2$ is always a small number in perfect gases.
    %
    As a consequence, it motivates our research into the behavior of solutions as the modified thermal diffusivity is small, i.e. $0<\kappa\ll 1$, throughout this paper. Moreover, we remark that the thermal diffusivity $\bar{\kappa}\downarrow0$ is equivalent to the modified thermal diffusivity $\kappa\downarrow0$ since the speed of sound $c_0$ is finite. Furthermore, we assume nontrivial initial data $\psi_2$, since the case with $\psi_2\equiv0$ is easy to be treated (see, for example, Remark \ref{Rem_Trivial_psi_2}).  We will not repeat these settings again in the subsequent statements. 
    
    	\begin{remark}
    	Differently from the classical works in the full compressible Navier-Stokes equations for perfect gases and their related models (see \cite{Hoff-Zumbrun, Kobayashi-Shibata,  Danchin, Kobayashi-Shibata-02, Xu-Kawashima, Dachin-Xu} and references therein), the general Blackstock's model (8) is an approximated scalar equation for some fluid-mechanics model under the irrotational assumption. For this reason, it provides a possible way in deriving the explicit characteristic equation for the linearized Cauchy problem (9) and studying optimal estimates as well as asymptotic profiles without using energy methods. It would be interesting (but this goes beyond the purposes of the present paper) to investigate some qualitative properties (e.g. well-posedness with less regular assumption on initial data, optimal time-decay estimates) of solutions to the general Blackstock's model by applying Littlewood-Paley decomposition, Fourier multiplier theorem and energy method in Fourier spaces as aforementioned references.
    \end{remark}
    
    Our first purpose in the present paper is to understand the linear Blackstock's model \eqref{Linear_Blackstock_Eq} by employing the phase space analysis, which helps us to understand the underlying physical phenomena in a certain condition and plays the crucial role for investigating the corresponding nonlinear problem \eqref{Blackstock_Eq-New-Nonlinear}. In Section \ref{Sec_Decay_Asym}, by using asymptotic expansions of characteristic roots and WKB analysis, we derive optimal decay estimates with weighted $L^1$ data for high dimensions $n\geqslant 5$ as well as three asymptotic profiles of solution including heat equations and diffusion-wave equations. Nonetheless, the decay structures are lost in lower dimensions, e.g. the $L^2$ norm of the solution blows up as $t\to\infty$ for $n=3,4$, particularly, we will show the sharp logarithmic  growth for $n=4$ and the polynomial growth of order $1/2$ for $n=3$. The crucial point is to get optimal estimates for the Fourier multiplier
    \begin{align*}
    	|\xi|^{-2}\left(\mathrm{e}^{-\kappa|\xi|^2t}-\cos(|\xi|t)\mathrm{e}^{-\frac{\delta}{2}|\xi|^2t}\right)\ \ \mbox{for}\ \ |\xi|\ll 1
    \end{align*}
    in some norms, since the combined influence from the dissipative structure, dispersive effect and singularities are not clear.
    Let us formally take into account the limit case with vanishing thermal diffusivity in the Cauchy problem \eqref{Linear_Blackstock_Eq}, which turns out to be the viscoelastic damped wave equation (see \cite{Shibata2000,IkehataTodorovaYordanov2013,DabbiccoReissig2014,Ikehata2014,IkehataNatsume,IkehataOnodera2017,BarreraVolkmer2019,BarreraVolkmer2020} and references therein) whose initial data are inherited from \eqref{Linear_Blackstock_Eq} as follows:
    \begin{align}\label{Linear_Kuznet_Eq}
    	\begin{cases}
    		\psi^{(0)}_{tt}-\Delta\psi^{(0)}-b\nu\Delta\psi^{(0)}_t=0,&x\in\mb{R}^n,\ t>0,\\
    		\psi^{(0)}(0,x)=\psi_0(x),\ \psi^{(0)}_t(0,x)=\psi_1(x),&x\in\mb{R}^n.
    	\end{cases}
    \end{align}
    It means that no heat is exchanged during the compression and decompression phases of sound propagation.
    Formally, \eqref{Linear_Kuznet_Eq} seems to be the limit equation for \eqref{Linear_Blackstock_Eq} with vanishing thermal diffusivity $\kappa=0$.
    In Section \ref{Sect_singular}, under the consistency assumption $\psi_2=b\nu\Delta\psi_1+\Delta\psi_0$, we rigorously justify singular limits (limiting processes) from the linear Blackstock's model \eqref{Linear_Blackstock_Eq} to the viscoselastic damped wave equation \eqref{Linear_Kuznet_Eq} as the thermal diffusivity tends to zero, for instance, we conclude in Subsection \ref{Sub_Singular_Solution} that $\psi\to\psi^{(0)}$ in $L^{\infty}([0,\infty)\times\mb{R}^n)$ as $\kappa\downarrow0$ for $n\geqslant 2$ with the rate of convergence $\sqrt{\kappa}$. Furthermore, by applying multi-scale analysis, we derive a formal expansion of solution with respect to $\kappa$, and rigorously justify the correction for the second-order asymptotic expansion with the rate of convergence $\kappa\sqrt{\kappa}$. We believe this approach can be widely used in the acoustic waves model, e.g. singular limits with higher-order profiles for  the Moore-Gibson-Thompson equation with the small thermal relaxation or the small sound diffusivity \cite{Chen-Ikehata-2021,Kaltenbacher-Nikolic}.

    Our second purpose is to study global (in time) well-posedness of the nonlinear Blackstock's model \eqref{Blackstock_Eq-New-Nonlinear} with small data. We firstly construct suitable weighted evolution spaces with  time-dependent weights inherited from the linear problem. By applying some fractional tools from harmonic analysis (see Appendix \ref{Appendix_A}), we estimate the nonlinear term $F(\psi;\partial_t,\nabla)$ in the Riesz potential spaces $\dot{H}^s$ for some $s>0$. Utilizing Banach's fixed-point theorem, we demonstrate global (in time) existence of small data Sobolev solutions
    to the nonlinear Cauchy problem \eqref{Blackstock_Eq-New-Nonlinear}  throughout Section \ref{Sec_GESDS}. 

    Our main contributions in the present paper consists in deriving optimal estimates and asymptotic profiles of solutions for large time (see Theorems \ref{Thm_Asym} and \ref{Thm_Optimal}), for the small thermal diffusivity (see Theorems \ref{Thm_Potential_2} and \ref{Thm_Higher}) to linear Blackstock's model \eqref{Linear_Blackstock_Eq}, and in proving global well-posedness (see Theorem \ref{Thm_GESDS}) for nonlinear Blackstock's model \eqref{Blackstock_Eq-New-Nonlinear}. In particular, we emphasize that all of our results still hold for the special case with Becker's assumption ($\delta=b\gamma\nu$).

    \begin{remark} 
    	Let us introduce the partial differential operator appearing on the linear Blackstock's model \eqref{Linear_Blackstock_Eq}, namely,
    	\begin{align*}
    		\mathcal{L}:= \partial_t^3-(b\nu+\gamma\kappa)\Delta\partial_t^2-\Delta\partial_t+\gamma b\nu\kappa\Delta^2\partial_t+\kappa\Delta^2.
    	\end{align*}
    	This operator $\mathcal{L}$ is not of Kovalevskian type (cf. \cite[Section 3.1]{Ebert-Reissig-book}). Moreover, the principal part of $\mathcal{L}$ in the sense of Petrovsky is 
    	\begin{align*}
    		\partial_t^3-(b\nu+\gamma\kappa)\Delta\partial_t^2+\gamma b\nu\kappa\Delta^2\partial_t = \partial_t (\partial_t-b\nu \Delta)  (\partial_t-\gamma\kappa \Delta),
    	\end{align*} so, the principal symbol in the sense of Petrovsky admits the roots $\tau=0$, $\tau=ib\nu |\zeta|$ and $\tau=i\gamma\kappa |\zeta|$ that are not all real for $\zeta\in\mathbb{R}^n\backslash\{0\} $. Consequently, $\mathcal{L}$ is not even a $2$-evolution operator  (see \cite[Chapter 3, Definition 3.2]{Ebert-Reissig-book}), and we may not apply the general theory for such class of operators to derive a well-posedness result in $L^2$. 
    	For this reason the well-posedness results that we are going to show in the energy space requiring different kinds of additional regularity for the Cauchy data (namely, $L^{1,1}$-regularity in Section \ref{Section additional L1,1 reg} and $L^1$-regularity in Section \ref{Section additional L1 reg}) are relevant.
    \end{remark}
    
    \medskip
    \noindent\textbf{Notation:} To end this section, we give some notations to be used in this paper. Later, $c$ and $C$ denote some positive constants, which may be changed from line to line. We denote that $f\lesssim g$ if there exists a positive constant $C$ such that $f\leqslant Cg$ and, analogously, for $f\gtrsim g$.  Moreover, $\dot{H}^s_p$ with $s\geqslant0$ and $1\leqslant p<\infty$ denote Riesz potential spaces based on $L^p$. The operators $\langle D\rangle^s$ and $|D|^s$ stand for the pseudo-differential operators with symbols $\langle\xi\rangle^s$ and $|\xi|^s$, respectively, where $\langle\xi\rangle:=\sqrt{1+|\xi|^2}$.
    \section{Decay properties and asymptotic profiles}\label{Sec_Decay_Asym}
    Firstly, by denoting $\tilde{\gamma}:=\kappa\delta+\kappa(\gamma-1)(b\nu-\kappa)=\gamma b\nu\kappa$, we directly employ the partial Fourier transform with respect to spatial variables to the Cauchy problem \eqref{Linear_Blackstock_Eq} and derive
    \begin{align}\label{Linear_Fourier_Blackstock_Eq}
    	\begin{cases}
    		\widehat{\psi}_{ttt}+(\delta+\kappa)|\xi|^2\widehat{\psi}_{tt}+\left(1+\tilde{\gamma}|\xi|^2\right)|\xi|^2\widehat{\psi}_t+\kappa|\xi|^4\widehat{\psi}=0,\\
    		\widehat{\psi}(0,\xi)=\widehat{\psi}_0(\xi),\ \widehat{\psi}_t(0,\xi)=\widehat{\psi}_1(\xi),\ \widehat{\psi}_{tt}(0,\xi)=\widehat{\psi}_2(\xi),
    	\end{cases}
    \end{align}
    for $\xi\in\mb{R}^n$, $t>0$. Then, the corresponding characteristic roots $\lambda_j=\lambda_j(|\xi|)$ to the equation of \eqref{Linear_Fourier_Blackstock_Eq} satisfy the following cubic equation:
    \begin{align}\label{Character_Eq}
    	\Lambda(\lambda_j):=\lambda_j^3+(\delta+\kappa)|\xi|^2\lambda_j^2+(1+\tilde{\gamma}|\xi|^2)|\xi|^2\lambda_j+\kappa|\xi|^4=0
    \end{align}
    for $j=1,2,3$. We immediately notice that
    \begin{align}\label{Derivative_lambda}
    	\frac{\mathrm{d}}{\mathrm{d}\lambda_j}\Lambda(\lambda_j)=3\lambda_j^2+2(\delta+\kappa)|\xi|^2\lambda_j+(1+\tilde{\gamma}|\xi|^2)|\xi|^2.
    \end{align}
    Comparing the polynomials $\Lambda(\lambda_j)$ and \eqref{Derivative_lambda}, we find that there are nontrivial common division at most for value of frequencies $|\xi|$ in a zero measure set. Hence, we have only simple roots of $\Lambda(\lambda_j)=0$ outside of this zero measure set. Since there are several parameters in the $|\xi|$-dependent characteristic equation \eqref{Character_Eq}, it seems a challenging work to analyze the precise roots of it. However, for our investigation in the phase space, we just need to use asymptotic expansions of the characteristic roots and WKB analysis.

    \subsection{Asymptotic behaviors of characteristic roots}\label{Sub_sec_asymptotic_behavior}
    
    $\quad$
    
    \noindent\underline{Part I: Asymptotic expansions for $|\xi|\to0$.} We find the characteristic roots $\lambda_j$ for $j=1,2,3$ owing for $|\xi|\to 0$ the following asymptotic expansion:
    \begin{align*}
    	\lambda_j=\lambda_{j,0}+\lambda_{j,1}|\xi|+\lambda_{j,2}|\xi|^2+\cdots,
    \end{align*}
    where the constant coefficients $\lambda_{j,k}\in\mb{C}$ are independent of $|\xi|$ for all $k\in\mb{N}_0$. 
    By some direct computations, we claim that pairwise distinct characteristic roots haven been obtained 
    \begin{align*}
    	\lambda_{1/2}&=\pm i|\xi|-\frac{\delta}{2}|\xi|^2+\ml{O}(|\xi|^3)\\
    	\lambda_3&=-\kappa|\xi|^2+\ml{O}(|\xi|^3)
    \end{align*}
    for $|\xi|\to 0$. Here, the real parts of all characteristic roots are negative.
    
    \medskip
    
    \noindent\underline{Part II: Asymptotic expansions for $|\xi|\to\infty$.} 
    The characteristic roots $\lambda_j$ for $j=1,2,3$ have the asymptotic expansions for $|\xi|\to \infty$ such that
    \begin{align*}
    	\lambda_j=\bar{\lambda}_{j,0}|\xi|^2+\bar{\lambda}_{j,1}|\xi|+\bar{\lambda}_{j,2}+\cdots,
    \end{align*}
    where the constant coefficients $\bar{\lambda}_{j,k}\in\mb{C}$ are independent of $|\xi|$ for all $k\in\mb{N}_0$. 
    With some direct calculations, we assert that pairwise distinct characteristic roots haven been obtained in this case, namely,
    \begin{align*}
    	\lambda_{1/2}&=\frac{1}{2}\left(-(\delta+\kappa)\pm\sqrt{(\delta+\kappa)^2-4\gamma b\nu \kappa}\,\right)|\xi|^2+\ml{O}(|\xi|),\\
    	\lambda_3&=-\frac{1}{\gamma b\nu}+\ml{O}(|\xi|^{-1}),
    \end{align*}
    for $|\xi|\to \infty$. Here, the real parts of all characteristic roots are negative.
    
    \medskip
    
    \noindent\underline{Part III: Stabilities for other cases.} As last step in the analysis of the characteristic roots, we study the case in which $|\xi|$ neither tends to zero nor tends to infinity, which means bounded frequencies away from 0. To prove the negativity of the real parts of characteristic roots (that means, we have some stability properties), we just need to claim that no pure imaginary characteristic root exists. By fixing an index $j\in\{1,2,3\}$, let the characteristic root $\lambda_j$ be pure imaginary such that $\lambda_j=i\tilde{\lambda}_j$ with  $\tilde{\lambda}_j\in\mb{R}\backslash\{0\}$. From the cubic equation \eqref{Character_Eq}, one finds
    \begin{align*}
    	i\tilde{\lambda}_j\left(\tilde{\lambda}_j^2-(1+\tilde{\gamma}|\xi|^2)|\xi|^2\right)+|\xi|^2\left((\delta+\kappa)\tilde{\lambda}_j^2-\kappa|\xi|^2\right)=0.
    \end{align*}
    The last equation is valid only if $(1+\tilde{\gamma}|\xi|^2)|\xi|^2=\tilde{\lambda}_j^2=\frac{\kappa}{\delta+\kappa}|\xi|^2$. However, it yields a contradiction from $\frac{\kappa}{\delta+\kappa}<1<1+\tilde{\gamma}|\xi|^2$. So, the continuity of the characteristic roots with respect to $|\xi|$ implies that the real parts of all characteristic roots are negative.
    
    \medskip
    
    \noindent\underline{Part IV: Pointwise estimates of solutions in the phase space.} In the case for three distinct roots of the characteristic equation \eqref{Character_Eq}, i.e. Part I and Part II, by using a Vandermonde matrix, 
    the solution to the Cauchy problem \eqref{Linear_Fourier_Blackstock_Eq} is given by
    \begin{align*}
    	\widehat{\psi}=\widehat{K}_0\widehat{\psi}_0+\widehat{K}_1\widehat{\psi}_1+\widehat{K}_2\widehat{\psi}_2,
    \end{align*}
    where the kernels in the phase space $\widehat{K}_j=\widehat{K}_j(t,|\xi|)$ for $j=0,1,2$ admit the following representations:
    \begin{align*}
    	\widehat{K}_0&:=\sum\limits_{k=1,2,3}\frac{\exp(\lambda_kt)\prod_{j=1,2,3,\ j\neq k}\lambda_j}{\prod_{j=1,2,3,\ j\neq k}(\lambda_k-\lambda_j)},\\
    	\widehat{K}_1&:=-\sum\limits_{k=1,2,3}\frac{\exp(\lambda_kt)\sum_{j=1,2,3,\ j\neq k}\lambda_j}{\prod_{j=1,2,3,\ j\neq k}(\lambda_k-\lambda_j)},\\
    	\widehat{K}_2&:=\sum\limits_{k=1,2,3}\frac{\exp(\lambda_kt)}{\prod_{j=1,2,3,\ j\neq k}(\lambda_k-\lambda_j)}.
    \end{align*}
    
    Let us give some explicit estimates of these kernels via WKB analysis. We define the three zones for the phase space $\ml{Z}_{\intt}(\varepsilon):=\{|\xi|<\varepsilon\ll1\}$, $\ml{Z}_{\bdd}(\varepsilon,N):=\{\varepsilon\leqslant |\xi|\leqslant N\}$ and $\ml{Z}_{\extt}(N)=\{|\xi|> N\gg1\}$. Moreover, we fix three cut-off functions $\chi_{\intt}(\xi),\chi_{\text{bdd}}(\xi),\chi_{\extt}(\xi)\in \mathcal{C}^{\infty} $ having their supports in the zone $\ml{Z}_{\intt}(\varepsilon)$, $\ml{Z}_{\text{bdd}}(\varepsilon/2,2N)$ and $\ml{Z}_{\extt}(N)$, respectively, so that $\chi_{\intt}(\xi)+\chi_{\bdd}(\xi)+\chi_{\extt}(\xi)=1$ for any $\xi\in\mathbb{R}^n$. 
    
    In the consideration as $|\xi|\to 0$, we should do more precise computations since the imaginary terms will exert some oscillation influence. According to the asymptotic expansions of the pairwise distinct characteristic roots, we may write
    \begin{align*}
    	\chi_{\intt}(\xi)\widehat{K}_0&=\chi_{\intt}(\xi)(\kappa|\xi|+\ml{O}(|\xi|^2))\sin(|\xi|t)\mathrm{e}^{-\frac{\delta}{2}|\xi|^2t+\ml{O}(|\xi|^3)t}\\
    	&\quad+\chi_{\intt}(\xi)\left(1+\ml{O}(|\xi|)\right)\mathrm{e}^{-\kappa|\xi|^2t+\ml{O}(|\xi|^3)t},\\
    	\chi_{\intt}(\xi)\widehat{K}_1&=\chi_{\intt}(\xi)\left(1+\ml{O}(|\xi|)\right)\frac{\sin(|\xi|t)}{|\xi|}\mathrm{e}^{-\frac{\delta}{2}|\xi|^2t+\ml{O}(|\xi|^3)t}\\
    	&\quad+\chi_{\intt}(\xi)\left(\delta+\ml{O}(|\xi|)\right)\mathrm{e}^{-\kappa|\xi|^2t+\ml{O}(|\xi|^3)t},\\	\chi_{\intt}(\xi)\widehat{K}_2&=\chi_{\intt}(\xi)\frac{-\cos(|\xi|t)}{|\xi|^2+\ml{O}(|\xi|^3)}\mathrm{e}^{-\frac{\delta}{2}|\xi|^2t+\ml{O}(|\xi|^3)t}\\
    	&\quad+\chi_{\intt}(\xi)\frac{1}{|\xi|^2+\ml{O}(|\xi|^3)}\mathrm{e}^{-\kappa|\xi|^2t+\ml{O}(|\xi|^3)t}.
    \end{align*}
    Then, the last asymptotic representation shows
    \begin{align}\label{New_2}
    	\chi_{\intt}(\xi)|\widehat{K}_2|&\lesssim\chi_{\intt}(\xi)|\xi|^{-2}\left|\mathrm{e}^{-\kappa|\xi|^2t+\ml{O}(|\xi|^3)t}-\cos(|\xi|t)\mathrm{e}^{-\frac{\delta}{2}|\xi|^2t+\ml{O}(|\xi|^3)t}\right|\notag\\
    	&\quad+\chi_{\intt}(\xi)|\xi|^{-1}\left(|\cos(|\xi|t)|\mathrm{e}^{-\frac{\delta}{2}|\xi|^2t+\ml{O}(|\xi|^3)t}+\mathrm{e}^{-\kappa|\xi|^2t+\ml{O}(|\xi|^3)t}\right)\notag\\
    	&\lesssim\chi_{\intt}(\xi)|\xi|^{-2}\left|\mathrm{e}^{-\kappa|\xi|^2t+\ml{O}(|\xi|^3)t}-\mathrm{e}^{-\frac{\delta}{2}|\xi|^2t+\ml{O}(|\xi|^3)t}\right|\notag\\
    	&\quad+\chi_{\intt}(\xi)\left(|\xi|^{-2}\left|\sin\left(\tfrac{|\xi|}{2}t\right)\right|^2\mathrm{e}^{-\frac{\delta}{2}|\xi|^2t}+|\xi|^{-1}\mathrm{e}^{-c|\xi|^2t}\right)\notag\\
    	&\lesssim\chi_{\intt}(\xi)\left(t+|\xi|^{-2}\left|\sin\left(\tfrac{|\xi|}{2}t\right)\right|^2+|\xi|^{-1}\right)\mathrm{e}^{-c|\xi|^2t},
    \end{align}
    where we used 
    \begin{align*}
    	\frac{1}{|\xi|^2+\ml{O}(|\xi|^3)}=\frac{1}{|\xi|^2}(1+\ml{O}(|\xi|))\ \ \mbox{for} \ \ |\xi|\ll 1,\ \ \cos(|\xi|t)=1-\left|\sin\left(\tfrac{|\xi|}{2}t\right)\right|^2,
    \end{align*}
    and
    \begin{align*}
    	&\chi_{\intt}(\xi)\left|\mathrm{e}^{-\kappa|\xi|^2t+\ml{O}(|\xi|^3)t}-\mathrm{e}^{-\frac{\delta}{2}|\xi|^2t+\ml{O}(|\xi|^3)t}\right|\\
    	&\qquad=\chi_{\intt}(\xi)\mathrm{e}^{-\kappa|\xi|^2t+\ml{O}(|\xi|^3)t}\left|1-\mathrm{e}^{\frac{2\kappa-\delta}{2}|\xi|^2t+\ml{O}(|\xi|^3)t}\right|\\
    	&\qquad\lesssim \chi_{\intt}(\xi)|\xi|^2t\mathrm{e}^{-\kappa|\xi|^2t}\left|\int_0^1\mathrm{e}^{\frac{2\kappa-\delta}{2}|\xi|^2ts+\ml{O}(|\xi|^3)ts}\mathrm{d}s\right|\\
    	&\qquad\lesssim\chi_{\intt}(\xi)|\xi|^2t\mathrm{e}^{-c|\xi|^2t}.
    \end{align*}
    The previous relations lead to
    \begin{align}\label{Pointwise_small}	\chi_{\intt}(\xi)|\widehat{\psi}|&\lesssim\chi_{\intt}(\xi)\mathrm{e}^{-c|\xi|^2t}\left(|\widehat{\psi}_0|+\left(1+\tfrac{|\sin(|\xi|t)|}{|\xi|}\right)|\widehat{\psi}_1|\right.\notag\\
    	&\quad\qquad\qquad\qquad\quad\left.+\left(t+|\xi|^{-2}\left|\sin\left(\tfrac{|\xi|}{2}t\right)\right|^2+|\xi|^{-1}\right)|\widehat{\psi}_2|\right)
    \end{align}
    with a positive constant $c$.

    Concerning $|\xi|\to\infty$ the employment of the asymptotic expansions of the roots, which shows the presence of dominant terms that are pairwise distinct and real, provides the estimate
    \begin{align}\label{Pointwise_Large}
    	\chi_{\extt}(\xi)|\widehat{\psi}|\lesssim\chi_{\extt}(\xi)\mathrm{e}^{-ct}\left(|\widehat{\psi}_0|+|\xi|^{-2}|\widehat{\psi}_1|+|\xi|^{-4}|\widehat{\psi}_2|\right)
    \end{align}
    with a positive constant $c$.
    
    Finally, in the case neither $|\xi|\to0$ nor $|\xi|\to\infty$, there may happen that the cubic equation \eqref{Character_Eq} has three roots but not pairwise distinct, i.e. roots of double multiplicity. At this time, to write down the explicit formula of solution is a possible but complex work with a lot of discussions for the parameters. We only concentrate on exponential stability in this case. Because the real parts of characteristic roots are negative as we shown in Part III, it yields
    \begin{align}\label{Pointwise_Middle}
    	\chi_{\bdd}(\xi)|\widehat{\psi}|\lesssim\chi_{\bdd}(\xi)\mathrm{e}^{-ct}\left(|\widehat{\psi}_0|+|\widehat{\psi}_1|+|\widehat{\psi}_2|\right)
    \end{align} 
    with a positive constant $c$.
    
    \subsection{$L^2-L^2$ decay estimates with additional weighted $L^1$ regularity} \label{Section additional L1,1 reg}
    Let us recall the definition of the weighted $L^1$ space
    \begin{align*}
    	L^{1,1} :=\left\{f\in L^1 :\ \|f\|_{L^{1,1} }:=\int_{\mb{R}^n}(1+|x|)|f(x)|\mathrm{d}x<\infty\right\}.
    \end{align*}
    Moreover, we introduce the notation $P_f:=\int_{\mb{R}^n}f(x)\mathrm{d}x$ for the mean of a summable function $f$, which can be estimated by $|P_f|\leqslant\|f\|_{L^{1,1} }$. Before deriving some estimates of the solution, let us propose a lemma to control the $L^2$ norm of some multipliers.
    \begin{lemma}\label{Lemma_Basic_1}
    	Let $s\in[0,\infty)$ and $c>0$. Then, the following estimates hold:
    	\begin{align*}
    		\|\chi_{\intt}(\xi)|\xi|^s\mathrm{e}^{-c|\xi|^2t}\|_{L^2 }&\lesssim(1+t)^{-\frac{s}{2}-\frac{n}{4}},\\
    		\|\chi_{\intt}(\xi)|\xi|^{-1}|\sin(|\xi|t)|\mathrm{e}^{-c|\xi|^2t}\|_{L^2 }&\lesssim\ml{D}_n(1+t),\\
    		\|\chi_{\intt}(\xi)|\xi|^{-1}\mathrm{e}^{-c|\xi|^2t}\|_{L^2 }
    		&\lesssim(1+t)^{\frac{1}{2}-\frac{n}{4}}\ \ \mbox{if}\ \ n\geqslant 3,
    	\end{align*}
    	for $t\geqslant0$,  moreover, the next estimate holds:
    	\begin{align}\label{New_01}	\left\|\chi_{\intt}(\xi)|\xi|^{-2}\big|\sin\big(\tfrac{|\xi|}{2}t\big)\big|^{2}\mathrm{e}^{-c|\xi|^2t}\right\|_{L^2}\lesssim\widetilde{\ml{D}}_n(1+t),
    	\end{align}
    	where
    	\begin{align*}
    		\ml{D}_n(1+t):= \begin{cases} (1+t)^{\frac{1}{2}} & \mbox{if} \ \  n=1, \\  (\ln (\mathrm{e}+t))^{\frac{1}{2}} & \mbox{if} \ \  n=2, \\ (1+t)^{\frac{1}{2}-\frac{n}{4}} & \mbox{if} \ \ n\geqslant 3,\end{cases}
    	\end{align*}
    	\begin{align*}
    		\widetilde{\ml{D}}_n(1+t):= \begin{cases} (1+t)^{2-\frac{n}{2}} & \mbox{if} \ \  n\leqslant3, \\  (\ln (\mathrm{e}+t))^{\frac{1}{2}} & \mbox{if} \ \  n=4,\\
    			(1+t)^{1-\frac{n}{4}}& \mbox{if} \ \ n\geqslant 5.\end{cases}
    	\end{align*}
    \end{lemma}
    \begin{proof}
    	The proof of the first inequality can be easily completed with the change of variables, and the derivation of the second estimates has been shown in \cite{Ikehata2014,IkehataOnodera2017}. We give a proof for the third inequality.
    	By using $\omega=r^2t$, we get
    	\begin{align*}
    		\left\|\chi_{\intt}(\xi)|\xi|^{-1}\mathrm{e}^{-c|\xi|^2t}\right\|_{L^2 }^2\lesssim\int_0^{\varepsilon}r^{n-3}\mathrm{e}^{-cr^2t}\mathrm{d}r\lesssim t^{-\frac{n-2}{2}} \int_0^{\infty}\omega^{\frac{n-4}{2}}\mathrm{e}^{-c\omega}\mathrm{d}\omega.
    	\end{align*}
    	Then, we demand that $n\geqslant2$, in order to guarantee the summability as $\omega\downarrow0$. For the derivation of inequality \eqref{New_01}, we observe
    	\begin{align*}	\left\|\chi_{\intt}(\xi)|\xi|^{-2}\big|\sin\big(\tfrac{|\xi|}{2}t\big)\big|^{2}\mathrm{e}^{-c|\xi|^2t}\right\|_{L^2}\lesssim\left\|\chi_{\intt}(\xi)|\xi|^{-1}\big|\sin\big(\tfrac{|\xi|}{2}t\big)\big|\mathrm{e}^{-c|\xi|^2t}\right\|_{L^4}^2.
    	\end{align*} 
    	Then, by applying $s=0$ and $m=4/3$ in the proof of Theorem 2.2 in \cite{Chen-Ikehata-2021}, we can complete the proof.
    \end{proof}
    
    %
    %
    
    \begin{theorem}\label{Thm_Energy}
    	Let us assume initial data $\psi_0,\psi_1,\psi_2\in L^2\cap L^{1,1} $. Then, the solution to Blackstock's model \eqref{Linear_Blackstock_Eq} fulfills the following estimate for $n\geqslant 3$: 
    	\begin{align*}
    		&\|\psi(t,\cdot)\|_{L^2 }\\
    		&\quad\lesssim (1+t)^{-\frac{n+2}{4}}\|\psi_0\|_{L^2\cap L^{1,1} }+(1+t)^{-\frac{n}{4}}\|\psi_1\|_{L^2\cap L^{1,1} }+\ml{D}_n(1+t)\|\psi_2\|_{L^2\cap L^{1,1}}\\
    		&\quad\quad+(1+t)^{-\frac{n}{4}}|P_{\psi_0}|+\ml{D}_n(1+t)|P_{\psi_1}|+\widetilde{\ml{D}}_n(1+t)|P_{\psi_2}|.
    	\end{align*}
    \end{theorem}
    \begin{remark}\label{Rem_Trivial_psi_2}
    	The nontrivial assumption for $\psi_2$ not only influences on the decay rate of total estimates, but also gives a strong restriction on dimensions. The main reason is that when $\psi_2\not\equiv0$, there is a strong singularity for $|\xi|\to 0$ for its corresponding multiplier in the norm, i.e. $\|\chi_{\intt}(\xi)|\xi|^{-1}\mathrm{e}^{-c|\xi|^2t}\|_{L^2}$ for the lower dimensions $n=1,2$. Nevertheless, this phenomenon disappears if $\psi_2\equiv0$. One may see \eqref{Last_01} later.
    \end{remark}
    \begin{proof}
    	By using Lemma 3.1 in \cite{Ikehata2004}, the solution in the phase space can be estimated by
    	\begin{align*}
    		|\widehat{\psi}|\lesssim|\xi|\sum\limits_{j=0,1,2}|\widehat{K}_j|\,\|\psi_j\|_{L^{1,1}}+\sum\limits_{j=0,1,2}|\widehat{K}_j|\,|P_{\psi_j}|=:\widehat{I}_1(t,|\xi|)+\widehat{I}_2(t,|\xi|).
    	\end{align*}
    	According to WKB analysis, the forthcoming discussion on the estimates will be divided into three parts for different sizes of frequencies. Concerning the first case $|\xi|\to0$, from the derived estimate \eqref{Pointwise_small}, one has
    	\begin{align*}
    		&\|\chi_{\intt}(\xi)\widehat{I}_1(t,|\xi|)\|_{L^2 }\\
    		&\qquad\lesssim\|\chi_{\intt}(\xi)|\xi|\mathrm{e}^{-c|\xi|^2t}\|_{L^2 }\|\psi_0\|_{L^{1,1} }+\|\chi_{\intt}(\xi)(|\xi|+|\sin(|\xi|t)|)\mathrm{e}^{-c|\xi|^2t}\|_{L^2 }\|\psi_1\|_{L^{1,1} }\\
    		&\qquad\quad+\left\|\chi_{\intt}(\xi)\left(t|\xi|+|\xi|^{-1}\left|\sin\left(\tfrac{|\xi|}{2}t\right)\right|^2+1\right)\mathrm{e}^{-c|\xi|^2t}\right\|_{L^2 }\|\psi_2\|_{L^{1,1} }\\
    		&\qquad\lesssim (1+t)^{-\frac{1}{2}-\frac{n}{4}}\|\psi_0\|_{L^{1,1} }+(1+t)^{-\frac{n}{4}}\|\psi_1\|_{L^{1,1} }+\ml{D}_n(1+t)\|\psi_2\|_{L^{1,1} }
    	\end{align*}
    	for $n\geqslant 1$, moreover,
    	\begin{align}\label{Last_01}
    		&\|\chi_{\intt}(\xi)\widehat{I}_2(t,|\xi|)\|_{L^2 }\notag\\
    		&\qquad\lesssim\|\chi_{\intt}(\xi)\mathrm{e}^{-c|\xi|^2t}\|_{L^2 }|P_{\psi_0}|+\|\chi_{\intt}(\xi)(1+|\xi|^{-1}|\sin(|\xi|t)|)\mathrm{e}^{-c|\xi|^2t}\|_{L^2 }|P_{\psi_1}|\notag\\
    		&\qquad \quad+\left\|\chi_{\intt}(\xi)\left(t+|\xi|^{-2}\left|\sin\left(\tfrac{|\xi|}{2}t\right)\right|^2+|\xi|^{-1}\right)\mathrm{e}^{-c|\xi|^2t}\right\|_{L^2 }|P_{\psi_2}|\notag\\
    		&\qquad\lesssim (1+t)^{-\frac{n}{4}}|P_{\psi_0}|+\ml{D}_n(1+t)|P_{\psi_1}|+\widetilde{\ml{D}}_n(1+t)|P_{\psi_2}|
    	\end{align}
    	for $n\geqslant 3$, where we used Lemma \ref{Lemma_Basic_1}. Thus, the Plancherel theorem implies for $n\geqslant 3$
    	\begin{align*}
    		&\|\chi_{\intt}(D)\psi(t,\cdot)\|_{L^2 }\\
    		&\qquad\lesssim (1+t)^{-\frac{1}{2}-\frac{n}{4}}\|\psi_0\|_{L^{1,1} }+(1+t)^{-\frac{n}{4}}\|\psi_1\|_{L^{1,1} }+\ml{D}_n(1+t)\|\psi_2\|_{L^{1,1} }\\
    		&\qquad\quad+(1+t)^{-\frac{n}{4}}|P_{\psi_0}|+\ml{D}_n(1+t)|P_{\psi_1}|+\widetilde{\ml{D}}_n(1+t)|P_{\psi_2}|.
    	\end{align*}
    	Let us turn to the case $|\xi|\to\infty$, the pointwise estimate \eqref{Pointwise_Large} implies
    	\begin{align*}
    		\|\chi_{\extt}(D)\psi(t,\cdot)\|_{L^2 }\lesssim \mathrm{e}^{-ct}\left(\|\psi_0\|_{L^2 }+\|\psi_1\|_{L^2 }+\|\psi_2\|_{L^2 }\right).
    	\end{align*}
    	To end the proof, we just need to consider some estimates for bounded frequencies. From \eqref{Pointwise_Middle} we get exponential decay estimates without additional assumption on the regularity of  initial data. 
    \end{proof}
    \begin{remark}\label{Rem_new_01}
    	In the case with general $\psi_2$ for $n=1,2$, we next will apply another idea to get upper bound estimates of solution to avoid a singularity as $|\xi|\to0$. We notice that $\widehat{K}_2(t,|\xi|)=\int_0^t\partial_t\widehat{K}_2(s,|\xi|)\mathrm{d}s$, since $\widehat{K}_2(0,|\xi|)=0$ from the initial condition of third data. Thus, from the representation of third kernel, it holds
    	\begin{align*}
    		\chi_{\intt}(\xi)|\partial_t\widehat{K}_2|\lesssim\chi_{\intt}(\xi)\left(|\xi|^{-1}\left|\sin\left(\tfrac{|\xi|}{2}s\right)\right|+1\right)\mathrm{e}^{-c|\xi|^2s},
    	\end{align*}
    	leading to
    	\begin{align*}
    		&\left\|\chi_{\intt}(\xi)\widehat{K}_2(t,|\xi|)\right\|_{L^2}\\
    		&\qquad\lesssim\int_0^t\left(\int_{\mb{R}^n}|\chi_{\intt}(\xi)\partial_t\widehat{K}_2(s,|\xi|)|^2\mathrm{d}\xi\right)^{1/2}\mathrm{d}s\\
    		&\qquad\lesssim\int_0^t\left(\left\|\chi_{\intt}(\xi)|\xi|^{-1}\left|\sin\left(\tfrac{|\xi|}{2}s\right)\right|\mathrm{e}^{-c|\xi|^2s}\right\|_{L^2}+\left\|\chi_{\intt}(\xi)\mathrm{e}^{-c|\xi|^2s}\right\|_{L^2}\right)\mathrm{d}s\\
    		&\qquad\lesssim\int_0^t\left(\ml{D}_n(1+s)+(1+s)^{-\frac{n}{4}}\right)\mathrm{d}s\lesssim\begin{cases}
    			(1+t)^{\frac{3}{2}}&\mbox{if}\ \ n=1,\\
    			(1+t)(\ln(\mathrm{e}+t))^{\frac{1}{2}}&\mbox{if}\ \ n=2,
    		\end{cases}
    	\end{align*}
    	where we applied Minkowski's integral inequality and Lemma \ref{Lemma_Basic_1}. By this way, we immediately conclude
    	\begin{align*}
    		&\|\psi(t,\cdot)\|_{L^2 }\\
    		&\quad\lesssim (1+t)^{-\frac{1}{2}-\frac{n}{4}}\|\psi_0\|_{L^2\cap L^{1,1} }+(1+t)^{-\frac{n}{4}}\|\psi_1\|_{L^2\cap L^{1,1} }+\ml{D}_n(1+t)\|\psi_2\|_{L^2\cap L^{1,1}}\\
    		&\quad\quad+(1+t)^{-\frac{n}{4}}|P_{\psi_0}|+\ml{D}_n(1+t)|P_{\psi_1}|+\begin{cases}
    			(1+t)^{\frac{3}{2}}|P_{\psi_2}|&\mbox{if}\ \ n=1,\\
    			(1+t)(\ln(\mathrm{e}+t))^{\frac{1}{2}}|P_{\psi_2}|&\mbox{if}\ \ n=2.
    		\end{cases}
    	\end{align*}
    \end{remark}
    \begin{remark}
    	Let us recall $(L^2\cap L^{1,1})-L^2$ estimates in Theorem \ref{Thm_Energy} and Remark \ref{Rem_new_01}. By applying the same approach used to prove existence of solutions in the classical energy space to the Cauchy problem for free wave equation (e.g. Chapter 14 in \cite{Ebert-Reissig-book}), and the representation of kernels associated with the mean value theorem, we may conclude the existence of solutions to the Blackstock's model \eqref{Linear_Blackstock_Eq} such that $\psi\in \ml{C}([0,\infty),L^2)$, where we assume initial data belonging to $L^2\cap L^{1,1}$. The existence result can be extended to $\psi\in\ml{C}([0,\infty),H^{s+4})\cap\ml{C}^1([0,\infty),H^{s+2})\cap \ml{C}^2([0,\infty),H^{s})$ for any $s\geqslant0$ with the aid of Proposition \ref{Prop_Estimate_Cru_Lin}.
    \end{remark}

    \subsection{$L^2-L^2$ decay estimates with additional $L^1$ regularity} \label{Section additional L1 reg}
    In this subsection, we will prepare some total $(L^2\cap L^1)-L^2$ estimates and $L^2-L^2$ estimates for the solution and its derivatives to the linearized problem \eqref{Linear_Blackstock_Eq}, which will contribute to the establishment of global (in time) existence of Sobolev solution to the nonlinear problem \eqref{Blackstock_Eq-New-Nonlinear} in Section \ref{Sec_GESDS}.

    Combining the asymptotic expansions for the characteristic roots from Part I and Part II in Subsection \ref{Sub_sec_asymptotic_behavior}
    with the representations for $\widehat{\psi}_t$, $\widehat{\psi}_{tt}$ from Part IV, we can easily get the next pointwise estimates:
    \begin{align*}
    	\chi_{\intt}(\xi)|\widehat{\psi}_t|&\lesssim\chi_{\intt}(\xi)\mathrm{e}^{-c|\xi|^2t}\left(|\widehat{\psi}_0|+|\widehat{\psi}_1|+(|\xi|^{-1}|\sin(|\xi|t)| +1)|\widehat{\psi}_2|\right),\\
    	\chi_{\intt}(\xi)|\widehat{\psi}_{tt}|&\lesssim\chi_{\intt}(\xi)\mathrm{e}^{-c|\xi|^2t}\left(|\widehat{\psi}_0|+|\widehat{\psi}_1|+|\widehat{\psi}_2|\right),\\
    	\chi_{\extt}(\xi)|\partial_t^j\widehat{\psi}|&\lesssim\chi_{\extt}(\xi)\mathrm{e}^{-ct}\left(|\xi|^{2j-2}|\widehat{\psi}_0|+|\xi|^{2j-2}|\widehat{\psi}_1|+|\xi|^{2j-4}|\widehat{\psi}_2|\right),\\
    	\chi_{\bdd}(\xi)|\partial_t^j\widehat{\psi}|&\lesssim\chi_{\bdd}(\xi)\mathrm{e}^{-ct}\left(|\widehat{\psi}_0|+|\widehat{\psi}_1|+|\widehat{\psi}_2|\right),
    \end{align*}
    for $j=1,2$.
    
    Therefore, by using Lemma \ref{Lemma_Basic_1} and combining
    \eqref{Pointwise_small}, \eqref{Pointwise_Large}, \eqref{Pointwise_Middle} and the above estimates for $\widehat{\psi}_t,\widehat{\psi}_{tt}$ in the different zones of the phase space, the following propositions for the energy estimates hold.
    \begin{prop}\label{Prop_Estimate_Cru_Lin}
    	Let us assume initial data $\psi_j\in H^{s+4-2j}\cap L^{1} $ for $j=0,1,2$ and $s\in[0,\infty)$. Then, the solution  to Blackstock's model \eqref{Linear_Blackstock_Eq} for $n\geqslant5$ fulfills for $t\geqslant 0$ the following estimates:
    	\begin{align*}
    		\|\partial_t^j\psi(t,\cdot)\|_{L^2}&\lesssim(1+t)^{\frac{2-j}{2}-\frac{n}{4}}\left(\|\psi_0\|_{\dot{H}^{\max\{2j-2,0\}}\cap L^1}+\|\psi_1\|_{\dot{H}^{\max\{2j-2,0\}}\cap L^1}\right.\\
    		&\qquad\qquad\qquad\quad\ \left.+\|\psi_2\|_{L^2\cap L^1}\right),
    	\end{align*}
    	and
    	\begin{align*}
    		\|\partial_t^j\psi(t,\cdot)\|_{\dot{H}^{s+4-2j}}&\lesssim(1+t)^{-\frac{s+2-j}{2}-\frac{n}{4}}\left(\|\psi_0\|_{\dot{H}^{s+2+\max\{2-2j,0\}}\cap L^1}+\|\psi_1\|_{\dot{H}^{s+2}\cap L^1}\right.\\
    		&\qquad\qquad\qquad\qquad\ \ \left.+\|\psi_2\|_{\dot{H}^s\cap L^1}\right).
    	\end{align*}
    \end{prop}
    
    \begin{prop}\label{Prop_Estimate_Cru_Lin_2}
    	Let us assume initial data $\psi_0\equiv0\equiv\psi_1$ and $\psi_2\in \dot{H}^{s} $ for $s\in[0,\infty)$. Then, the solution to Blackstock's model \eqref{Linear_Blackstock_Eq} for $n\geqslant5$ fulfills  for $t\geqslant 0$ the following estimates:
    	\begin{align*}
    		\|\partial_t^j\psi(t,\cdot)\|_{\dot{H}^{s+4-2j}}&\lesssim(1+t)^{-\frac{2-j}{2}}\|\psi_2\|_{\dot{H}^s}.
    	\end{align*}
    \end{prop}
    
    To end this part, we underline that the homogeneous Sobolev space $\dot{H}^s$ for initial data can be replaced by the inhomogeneous Soblev space $H^s$ since $|\xi|^{2s}\leqslant \langle \xi\rangle^{2s}$ for any $\xi\in\mb{R}^n$. Moreover, the derivation of Proposition \ref{Prop_Estimate_Cru_Lin_2} contributes to the nonlinear problem in Section \ref{Sec_GESDS}, therefore, with this aim, we just need to consider nontrivial $\psi_2$. The corresponding estimates for nontrivial $\psi_0$ and $\psi_1$ also can be obtained easily.

    \subsection{Asymptotic profiles and optimal estimates}
    To understand the asymptotic profiles of solutions, we first give some approximations to the kernels in the next proposition.
    \begin{prop}\label{Prop_01}
    	The following estimates:
    	\begin{align*}
    		\left|\chi_{\intt}(\xi)(\widehat{K}_j-\widehat{J}_j)\widehat{\psi}_j\right|&\lesssim\chi_{\intt}(\xi)|\xi|^{1-j}\mathrm{e}^{-c|\xi|^2t}|\widehat{\psi}_j|
    	\end{align*}
    	hold for $j=0,1,2,$ with positive constants $c$, where $\widehat{J}_j=\widehat{J}_j(t,|\xi|)$ for $j=0,1,2$ are given by
    	\begin{align*}
    		\widehat{J}_0:=\mathrm{e}^{-\kappa|\xi|^2 t},\ \ 
    		\widehat{J}_1:=\frac{\sin(|\xi|t)}{|\xi|}\mathrm{e}^{-\frac{\delta}{2}|\xi|^2t},\ \ 
    		\widehat{J}_2:=-\frac{\cos(|\xi|t)}{|\xi|^2}\mathrm{e}^{-\frac{\delta}{2}|\xi|^2 t}+\frac{1}{|\xi|^2}\mathrm{e}^{-\kappa|\xi|^2 t}.
    	\end{align*}
    \end{prop}
    \begin{proof}
    	Let us introduce a useful estimate
    	\begin{align}
    		\chi_{\intt}(\xi)\mathrm{e}^{-c|\xi|^2t}(\mathrm{e}^{\ml{O}(|\xi|^3)t}-1)&=\chi_{\intt}(\xi)t\,\ml{O}(|\xi|^3)\mathrm{e}^{-c|\xi|^2t}\int_0^1\mathrm{e}^{\ml{O}(|\xi|^3)ts}\mathrm{d}s\notag\\
    		&\lesssim \chi_{\intt}(\xi)|\xi|\mathrm{e}^{-c|\xi|^2t}.\label{Useful_1}
    	\end{align}
    	By direct computations and \eqref{Useful_1}, one arrives at
    	\begin{align*}
    		\left|\chi_{\intt}(\xi)(\widehat{K}_0-\widehat{J}_0)\widehat{\psi}_0\right|&\lesssim \chi_{\intt}(\xi)|\xi|\mathrm{e}^{-\frac{\delta}{2}|\xi|^2t}|\widehat{\psi}_0|\\
    		&\quad+\chi_{\intt}(\xi)\mathrm{e}^{-\kappa|\xi|^2t}(\mathrm{e}^{\ml{O}(|\xi|^3)t}-1)|\widehat{\psi}_0|\\
    		&\quad+\chi_{\intt}(\xi)\ml{O}(|\xi|)\mathrm{e}^{-\kappa|\xi|^2t+\ml{O}(|\xi|^3)t}|\widehat{\psi}_0|\\
    		&\lesssim \chi_{\intt}(\xi)|\xi|\mathrm{e}^{-c|\xi|^2t}|\widehat{\psi}_0|.
    	\end{align*}
    	Similarly, we find
    	\begin{align*}
    		\left|\chi_{\intt}(\xi)(\widehat{K}_1-\widehat{J}_1)\widehat{\psi}_1\right|&\lesssim\chi_{\intt}(\xi)|\sin(|\xi|t)|\,|\xi|^{-1}\mathrm{e}^{-\frac{\delta}{2}|\xi|^2t}(\mathrm{e}^{\ml{O}(|\xi|^3)t}-1)|\widehat{\psi}_1|\\
    		&\quad+\chi_{\intt}(\xi)|\sin(|\xi|t)|\mathrm{e}^{-\frac{\delta}{2}|\xi|^2t+\ml{O}(|\xi|^3)t}|\widehat{\psi}_1|\\
    		&\quad+\chi_{\intt}(\xi)\mathrm{e}^{-\kappa|\xi|^2t}|\widehat{\psi}_1|\\
    		&\lesssim \chi_{\intt}(\xi)\mathrm{e}^{-c|\xi|^2t}|\widehat{\psi}_1|,
    	\end{align*}
    	and
    	\begin{align*}
    		\left|\chi_{\intt}(\xi)(\widehat{K}_2-\widehat{J}_2)\widehat{\psi}_2\right|&\lesssim \chi_{\intt}(\xi)\frac{\ml{O}(|\xi|^3)|\cos(|\xi|t)|}{(|\xi|^2+\ml{O}(|\xi|^3))|\xi|^2}\mathrm{e}^{-\frac{\delta}{2}|\xi|^2t+\ml{O}(|\xi|^3)t}|\widehat{\psi}_2|\\
    		&\quad+\chi_{\intt}(\xi)\frac{|\cos(|\xi|t)|}{|\xi|^2}\mathrm{e}^{-\frac{\delta}{2}|\xi|^2t}(\mathrm{e}^{\ml{O}(|\xi|^3)t}-1)|\widehat{\psi}_2|\\
    		&\quad+\chi_{\intt}(\xi)\frac{\ml{O}(|\xi|^3)}{(|\xi|^2+\ml{O}(|\xi|^3))|\xi|^2}\mathrm{e}^{-\kappa|\xi|^2t+\ml{O}(|\xi|^3)t}|\widehat{\psi}_2|\\
    		&\quad+\chi_{\intt}(\xi)\frac{1}{|\xi|^2}\mathrm{e}^{-\kappa|\xi|^2t}(\mathrm{e}^{\ml{O}(|\xi|^3)t}-1)\\
    		&\lesssim \chi_{\intt}(\xi)|\xi|^{-1}\mathrm{e}^{-c|\xi|^2t}|\widehat{\psi}_2|.
    	\end{align*}
    	Then, the proof is complete.
    \end{proof}
    
    Next, we will propose some optimal estimates for these approximations for large time $t\gg1$, which will provides us a way to demonstrate the optimality of the estimates in previous sections, by deriving lower bound estimates for the solutions later in Theorem \ref{Thm_Optimal}.
    \begin{prop}\label{Prop_02}
    	The following optimal estimates:
    	\begin{align*}
    		t^{-\frac{n}{4}}&\lesssim\|\chi_{\intt}(\xi)\widehat{J}_0(t,|\xi|)\|_{L^2 }\lesssim t^{-\frac{n}{4}},\\
    		\ml{D}_n(t)&\lesssim\|\chi_{\intt}(\xi)\widehat{J}_1(t,|\xi|)\|_{L^2 }\lesssim \ml{D}_n(t),\\
    		\widetilde{\ml{D}}_n(t)& \lesssim\|\chi_{\intt}(\xi)\widehat{J}_2(t,|\xi|)\|_{L^2 }\lesssim \widetilde{\ml{D}}_n(t),
    	\end{align*}
    	hold for any $t\gg1$ and $n\geqslant 1$. Here, the time-dependent functions $\ml{D}_n(t)$ and $\widetilde{\ml{D}}_n(t)$ were defined in Lemma \ref{Lemma_Basic_1}.
    \end{prop}
    \begin{proof}
    	The first two estimates have been derived by \cite{Ikehata2014,IkehataOnodera2017} already. Furthermore,  concerning the upper bound estimate of $\widehat{J}_2(t,|\xi|)$, with the same explanation as \eqref{New_2}, we have
    	\begin{align*}
    		\chi_{\intt}(\xi)|\widehat{J}_2|\lesssim\chi_{\intt}(\xi)\left(t+\left||\xi|^{-1}\sin\left(\tfrac{|\xi|}{2}t\right)\right|^2\right)\mathrm{e}^{-c|\xi|^2t},
    	\end{align*}
    	which can be estimated in the $L^2$ norm by using Lemma \ref{Lemma_Basic_1} again for all $n\geqslant 1$ since the profile $\widehat{J}_2(t,|\xi|)$ does not contain the term $1/|\xi|$ with strong singularity.
    	For these reasons, we just focus on the deduction of the lower bound estimates for $\widehat{J}_2(t,|\xi|)$ in small frequency zone in the rest of the proof.
    	
    	We first give a lower bound estimate for the approximation $\chi_{\intt}(\xi)\widehat{J}_2(t,|\xi|)$ in the $L^2$ norm for $n\geqslant 5$. 
    	Let us fix a positive parameter $\alpha_0$ such that
    	\begin{align}\label{Choice_alpha0}
    		\alpha_0^2>\frac{5\ln 2}{\delta-8\kappa}=\frac{5\ln 2}{b\nu+(\gamma-9)\kappa}.
    	\end{align}
    	By using $|f-g|^2\geqslant\frac{1}{2}|f|^2-|g|^2$ and the boundedness of the cosine function, one derives
    	\begin{align*}
    		&\|\chi_{\intt}(\xi)\widehat{J}_2(t,|\xi|)\|_{L^2}^2\\
    		&\qquad\geqslant \left\|-\frac{\cos(|\xi|t)}{|\xi|^2}\mathrm{e}^{-\frac{\delta}{2}|\xi|^2t}+\frac{1}{|\xi|^2}\mathrm{e}^{-\kappa|\xi|^2t}\right\|_{L^2(\alpha_0 t^{-1/2}\leqslant|\xi|\leqslant2\alpha_0 t^{-1/2})}^2\\
    		&\qquad\geqslant\frac{1}{2}\underbrace{\int_{\alpha_0 t^{-1/2}\leqslant|\xi|\leqslant2\alpha_0 t^{-1/2}}\frac{\mathrm{e}^{-2\kappa|\xi|^2t}}{|\xi|^4}\mathrm{d}\xi}_{=:\ml{K}_{1}(t)}-\underbrace{\int_{\alpha_0 t^{-1/2}\leqslant|\xi|\leqslant2\alpha_0 t^{-1/2}}\frac{\mathrm{e}^{-\delta|\xi|^2t}}{|\xi|^4}\mathrm{d}\xi}_{=:\ml{K}_{2}(t)},
    	\end{align*}
    	where we considered $t\gg1$ satisfying $2\alpha_0 t^{-1/2}\leqslant\varepsilon$, i.e. $t\geqslant 4\alpha_0^2/\varepsilon^2$ with small $\varepsilon>0$. Using polar coordinates, we arrive at
    	\begin{align*}
    		\ml{K}_{1}(t)&\geqslant\frac{\omega_n\mathrm{e}^{-8\kappa\alpha_0^2}}{16\alpha_0^4t^{-2}}\int_{\alpha_0 t^{-1/2}}^{2\alpha_0 t^{-1/2}}\sigma^{n-1}\mathrm{d}\sigma=\frac{\omega_n(2^n-1)\alpha_0^{n-4}}{16n\mathrm{e}^{8\kappa\alpha_0^2}}t^{2-\frac{n}{2}},\\
    		\ml{K}_{2}(t)&\leqslant\frac{\omega_n\mathrm{e}^{-\delta\alpha_0^2}}{\alpha_0^4 t^{-2}}\int_{\alpha_0 t^{-1/2}}^{2\alpha_0 t^{-1/2}}\sigma^{n-1}\mathrm{d}\sigma=\frac{\omega_n(2^n-1)\alpha_0^{n-4}}{n\mathrm{e}^{\delta\alpha_0^2}}t^{2-\frac{n}{2}},
    	\end{align*}
    	where $\omega_n$ denoted the $(n-1)$-dimensional measure of the unit sphere.
    	It means
    	\begin{align*}
    		\|\chi_{\intt}(\xi)\widehat{J}_2(t,|\xi|)\|_{L^2}^2\geqslant\frac{\omega_n(2^n-1)\alpha_0^{n-4}}{n\mathrm{e}^{\delta\alpha_0^2}}\left( \frac{\mathrm{e}^{(\delta-8\kappa)\alpha_0^2}}{32}-1 \right)t^{2-\frac{n}{2}}.
    	\end{align*}
    	The choice \eqref{Choice_alpha0} implies the lower bound estimate of $\chi_{\intt}(\xi)\widehat{J}_2(t,|\xi|)$ for $t\gg1$ and $n\geqslant 5$.
    	
    	In the case $n=4$, we need to carry out more delicate WKB analysis. We now consider $t\gg1$ such that $\beta_0t^{-1/2}\leqslant \varepsilon$, where $\beta_0>1$ and 
    	\begin{align*}
    		\beta_0^2>\frac{3\ln 2}{b\nu+(\gamma-3)\kappa}=\frac{3\ln 2}{\delta-2\kappa}.	
    	\end{align*}
    	Let us estimate
    	\begin{align*}
    		\|\chi_{\intt}(\xi)\widehat{J}_2(t,|\xi|)\|_{L^2}^2&\geqslant \|\widehat{J}_2(t,|\xi|)\|_{L^2(\beta_0 t^{-1/2}\leqslant|\xi|\leqslant \varepsilon)}^2\\
    		&\geqslant\frac{1}{2}\underbrace{\int_{\beta_0 t^{-1/2}\leqslant|\xi|\leqslant\varepsilon}\frac{\mathrm{e}^{-2\kappa|\xi|^2t}}{|\xi|^4}\mathrm{d}\xi}_{=:\ml{K}_{3}(t)}-\underbrace{\int_{\beta_0 t^{-1/2}\leqslant|\xi|\leqslant\varepsilon}\frac{\mathrm{e}^{-\delta|\xi|^2t}}{|\xi|^4}\mathrm{d}\xi}_{=:\ml{K}_{4}(t)}.
    	\end{align*}
    	For one thing, we change the variable and use integration by parts to treat $\ml{K}_3(t)$, namely,
    	\begin{align*}
    		\ml{K}_3(t)&=\omega_4\int_{\beta_0t^{-1/2}}^{\varepsilon}\sigma^{-1}\mathrm{e}^{-2\kappa\sigma^2 t}\mathrm{d}\sigma\\
    		&=\omega_4\mathrm{e}^{-2\kappa\varepsilon^2t}\ln\varepsilon-\omega_4\left(\ln\beta_0-\tfrac{1}{2}\ln t\right)\mathrm{e}^{-2\kappa\beta_0^2}+4\omega_4\kappa t\int_{\beta_0t^{-1/2}}^{\varepsilon}\sigma\mathrm{e}^{-2\kappa\sigma^2t}\ln\sigma\mathrm{d}\sigma.
    	\end{align*}
    	We know
    	\begin{align*}
    		\mathrm{e}^{-2\kappa\beta_0^2}\ln\beta_0+\mathrm{e}^{-2\kappa\varepsilon^2t}\ln\frac{1}{\varepsilon}\leqslant \frac{1}{16}\mathrm{e}^{-2\kappa\beta_0^2}\ln t
    	\end{align*}
    	for $t\gg1$. Additionally, concerning $t\gg1$, one finds
    	\begin{align*}
    		&\lim\limits_{t\to\infty}\frac{-2\kappa t\int_{\beta_0t^{-1/2}}^{\varepsilon}\sigma\mathrm{e}^{-2\kappa\sigma^2t}\ln\sigma\mathrm{d}\sigma}{\frac{1}{4}\mathrm{e}^{-2\kappa\beta_0^2}\ln t}\\
    		&\qquad=\lim\limits_{t\to\infty}\frac{2\kappa\mathrm{e}^{2\kappa\beta_0^2}\int_{\beta_0^2}^{\varepsilon^2t}(\ln t-\ln\sigma)\mathrm{e}^{-2\kappa\sigma}\mathrm{d}\sigma}{\ln t}\\
    		&\qquad=2\kappa\mathrm{e}^{2\kappa\beta_0^2}\lim\limits_{t\to\infty}\left(\int_{\beta_0^2}^{\varepsilon^2t}\mathrm{e}^{-2\kappa\sigma}\mathrm{d}\sigma-\frac{1}{\ln t}\int_{\beta_0^2}^{\varepsilon^2t}\mathrm{e}^{-2\kappa\sigma}\ln\sigma\mathrm{d}\sigma\right)\leqslant 1
    	\end{align*}
    	due to the fact that $\mathrm{e}^{-2\kappa\sigma}\ln\sigma>0$ for $\sigma\in[\beta_0^2,\varepsilon^2t]$. 
    	Thus, we can get
    	\begin{align*}
    		\left|2\kappa t\int_{\beta_0t^{-1/2}}^{\varepsilon}\sigma\mathrm{e}^{-2\kappa\sigma^2t}\ln\sigma\mathrm{d}\sigma\right|&\leqslant -2\kappa t\int_{\beta_0t^{-1/2}}^{\varepsilon}\sigma\mathrm{e}^{-2\kappa\sigma^2t}\ln\sigma\mathrm{d}\sigma\leqslant \frac{3}{8}\mathrm{e}^{-2\kappa\beta_0^2}\ln t
    	\end{align*}
    	for large time $t\gg1$. The summary of the last inequalities shows
    	\begin{align*}
    		\ml{K}_3(t)\geqslant \frac{\omega_4}{16}\mathrm{e}^{-2\kappa\beta_0^2}\ln t.
    	\end{align*}
    	For another, we directly obtain
    	\begin{align*}
    		\ml{K}_4(t)&\leqslant\omega_4\mathrm{e}^{-\delta\beta_0^2}\int_{\beta_0t^{-1/2}}^{\varepsilon}\sigma^{-1}\mathrm{d}\sigma \leqslant\frac{\omega_4}{2}\mathrm{e}^{-\delta\beta_0^2}\ln t
    	\end{align*}
    	for $t\gg1$. All in all,
    	\begin{align*}
    		\|\chi_{\intt}(\xi)\widehat{J}_2(t,|\xi|)\|_{L^2}^2\geqslant \frac{\omega_4}{16}\mathrm{e}^{-2\kappa\beta_0^2}\ln t-\frac{\omega_4}{2}\mathrm{e}^{-\delta\beta_0^2}\ln t\gtrsim \ln t
    	\end{align*}
    	for $t\gg1$ and $0<\kappa\ll 1$, where we used the condition for $\beta_0$.

    	Finally, we determine the lower bound estimate for  $\|\chi_{\intt}(\xi)\widehat{J}_2(t,|\xi|)\|_{L^2}$ for $n=1,2,3$. In comparison to what we did in the higher dimensional case $n\geqslant 5$, we have to combine somehow the two terms that are presented in $\widehat{J}_2(t,|\xi|)$. More precisely, from the following representation:
    	\begin{align*}
    		\widehat{J}_2(t,|\xi|) & = |\xi|^{-2}\mathrm{e}^{-\kappa|\xi|^2 t}\left(1-\mathrm{e}^{(\kappa-\frac{\delta}{2})|\xi|^2 t}\right) +|\xi|^{-2}\left|\sin \left(\tfrac{|\xi|}{2}t\right)\right|^2 \mathrm{e}^{-\frac{\delta}{2}|\xi|^2 t} \\
    		& =\left(\kappa-\tfrac{\delta}{2}\right)\mathrm{e}^{-\kappa|\xi|^2 t}\int_0^1 \mathrm{e}^{(\kappa-\frac{\delta}{2})|\xi|^2  t\omega}\mathrm{d}\omega +|\xi|^{-2}\left|\sin\left(\tfrac{|\xi|}{2}t\right)\right|^2 \mathrm{e}^{-\frac{\delta}{2}|\xi|^2 t},
    	\end{align*}
    	by using $|f-g|^2\geqslant\frac{1}{2}|f|^2-|g|^2$, we arrive at
    	\begin{align*}
    		|\widehat{J}_2(t,|\xi|) |^2 & \geqslant  \tfrac12 |\xi|^{-4}\left|\sin \left(\tfrac{|\xi|}{2}t\right)\right|^4 \mathrm{e}^{-\delta|\xi|^2 t} - \left(\kappa-\tfrac{\delta}{2}\right)^2\mathrm{e}^{-2\kappa|\xi|^2 t}\left|\int_0^1 \mathrm{e}^{(\kappa-\frac{\delta}{2})|\xi|^2  t\omega}\mathrm{d}\omega\right|^2 \\ 
    		& \geqslant  \tfrac12 |\xi|^{-4}\left|\sin \left(\tfrac{|\xi|}{2}t\right)\right|^4 \mathrm{e}^{-\delta|\xi|^2 t} - C \mathrm{e}^{-c|\xi|^2 t}.
    	\end{align*} Let us fix $\alpha_1>0$ such that $|\tau^{-1}\sin \tau|\geqslant  1/2$ for any $\tau\in [0,\alpha_1]$. Then, considering $t\gg 1$ such that $2\alpha_1 t^{-1}\leqslant \varepsilon$, that is, $t\geqslant 2\alpha_1/\varepsilon$ for a small $\varepsilon>0$, we find
    	\begin{align*}
    		&\|\chi_{\intt}(\xi)\widehat{J}_2(t,|\xi|)\|_{L^2}^2\\
    		&\qquad \geqslant \frac12\int_{\alpha_1 t^{-1}\leqslant |\xi|\leqslant 2\alpha_1 t^{-1}}   |\xi|^{-4}\left|\sin \left(\tfrac{|\xi|}{2}t\right)\right|^4 \mathrm{e}^{-\delta|\xi|^2 t} \mathrm{d}\xi- C\int_{|\xi|\leqslant\varepsilon} \mathrm{e}^{-c|\xi|^2 t} \mathrm{d}\xi \\
    		&\qquad \gtrsim t^{4}\int_{\alpha_1 t^{-1}\leqslant |\xi|\leqslant 2\alpha_1 t^{-1}}  \mathrm{e}^{-\delta|\xi|^2 t} \mathrm{d}\xi-  t^{-\frac{n}{2}}\\
    		&\qquad \gtrsim t^{4-n} \mathrm{e}^{-\delta\alpha_1^2 t^{-1}}-  t^{-\frac{n}{2}}  \gtrsim t^{4-n},
    	\end{align*} where in the second inequality we used Lemma \ref{Lemma_Basic_1}. This completes the proof.
    \end{proof}

    Let us now state the main results  of this part and their corresponding explanations.
    \begin{theorem}\label{Thm_Asym}
    	Let us assume initial data  $\psi_0,\psi_1,\psi_2\in L^2\cap L^{1,1} $. Then, the solution to Blackstock's model \eqref{Linear_Blackstock_Eq} fulfills the following refined estimate for $n\geqslant 3$:
    	\begin{align*}
    		& \|(\psi-J_0\psi_0-J_1\psi_1-J_2\psi_2)(t,\cdot)\|_{L^2 }\\
    		&\quad \ \lesssim (1+t)^{-\frac{1}{2}-\frac{n}{4}}\|\psi_0\|_{L^2\cap L^{1,1}}+(1+t)^{-\frac{n}{4}}\|\psi_1\|_{L^2\cap L^{1,1}}+(1+t)^{\frac{1}{2}-\frac{n}{4}}\|\psi_2\|_{L^2\cap L^{1,1}}.
    	\end{align*}
    	%
    	%
    \end{theorem}
    \begin{proof}
    	By using Plancherel formula, \eqref{Pointwise_Large}, \eqref{Pointwise_Middle} and Proposition \ref{Prop_01}, we obtain
    	\begin{align}
    		&\left\|\left(\psi-\sum\limits_{j=0,1,2}J_j\psi_j\right)(t,\cdot)
    		\right\|_{L^2 }\notag\\
    		&\qquad\lesssim\left\|\chi_{\intt}(\xi)\left(\psi(t,\xi)-\sum\limits_{j=0,1,2}\widehat{J}_j(t,|\xi|)\widehat{\psi}_j(\xi)\right)\right\|_{L^2 }+\mathrm{e}^{-ct}\sum\limits_{j=0,1,2}\|\psi_j\|_{L^2}\notag \\
    		&\qquad\lesssim\sum\limits_{j=0,1,2}\left\|\chi_{\intt}(\xi)\left(\widehat{K}_j(t,|\xi|)-\widehat{J}_j(t,|\xi|)\right)\widehat{\psi}_j(\xi)\right\|_{L^2 }+\mathrm{e}^{-ct}\sum\limits_{j=0,1,2}\|\psi_j\|_{L^2}\notag \\
    		&\qquad\lesssim\sum\limits_{j=0,1,2}\left\|\chi_{\intt}(\xi)|\xi|^{1-j}\mathrm{e}^{-c|\xi|^2t}\widehat{\psi}_j(\xi)\right\|_{L^2}+\mathrm{e}^{-ct}\sum\limits_{j=0,1,2}\|\psi_j\|_{L^2}. \label{psi - Jj psij}
    	\end{align}
    	Then, with the help of Lemma \ref{Lemma_Basic_1} in the present paper and the next inequality from \cite{Ikehata2004,Racke-Said-Houari}
    	\begin{align*}
    		|\hat{f}|\lesssim|\xi|\,\|f\|_{L^{1,1}}+|P_f|,
    	\end{align*}
    	we can complete the proof.
    \end{proof}
    
    We now give some explanations for three operators shown in Proposition \ref{Prop_01}, respectively,
    \begin{align*}
    	J_0=\mathrm{e}^{\kappa t\Delta},\ \ J_1=\frac{\sin(\sqrt{-\Delta}\, t)}{\sqrt{-\Delta}}\mathrm{e}^{\frac{\delta}{2}t\Delta},\ \ J_2=-\frac{\cos(\sqrt{-\Delta}\,t)}{-\Delta}\mathrm{e}^{\frac{\delta}{2}t\Delta}+\frac{1}{-\Delta}\mathrm{e}^{\kappa t\Delta}.
    \end{align*}
    From the action of these operators on the corresponding initial data in Theorem \ref{Thm_Asym} we can really show some asymptotic profiles of the solution to the Blackstock's model \eqref{Linear_Blackstock_Eq}.
    
    \begin{itemize}
    	\item Obviously, $J_0$ is a solution operator associated with the heat equation
    	\begin{align}\label{Heat}
    		u_t-\kappa\Delta u=0.
    	\end{align}
    	Furthermore, by subtracting the function $J_0\psi_0$ in Theorem \ref{Thm_Asym}, we observe an improvement of decay rate $(1+t)^{-\frac{1}{2}}$  of the time-dependent coefficient of $\psi_0$ if we comparing with that in Theorem \ref{Thm_Energy}. Therefore, we may interpret the first profile $J_0\psi_0$ by the heat equation generated by Fourier's law. It is a reasonable profile due to the fact that we have used the heat condition of the classical Fourier's law in the conservation of energy from the modeling.
    	\item Let us define $c_{\delta}=\delta/2$. Then, we do some formal computations as follows:
    	\begin{align*}
    		\partial_tJ_1&=\cos(\sqrt{-\Delta}\,t)\mathrm{e}^{c_{\delta}t\Delta}+c_{\delta}\Delta J_1,\\
    		\partial_t^2J_1&=\Delta J_1+2c_{\delta}\Delta\partial_tJ_1-c_{\delta}^2\Delta^2 J_1.
    	\end{align*}
    	By this way, we found that $J_1$ is one of solution operators (for second data) associated with the structurally damped plate equation (one also may see \cite{Ikehata-Soga-2015})
    	\begin{align*}
    		v_{tt}-\Delta v+c_{\delta}^2\Delta^2v-2c_{\delta}\Delta v_t=0,
    	\end{align*}
    	or the diffusion-wave equations (include heat structure and half-wave structure)
    	\begin{align}\label{diffusion-wave}
    		\begin{cases}
    			v_t-c_{\delta}\Delta v+i\sqrt{-\Delta}\,v=w,\\
    			w_t-c_{\delta}\Delta w-i\sqrt{-\Delta}\,w=0.
    		\end{cases}
    	\end{align}
    	Furthermore, by subtracting the function $J_1\psi_1$ in Theorem \ref{Thm_Asym}, we observe an improvement of decay rate $(1+t)^{-\frac{3}{4}}$ when $n=1$, $(\ln(\mathrm{e}+t))^{-\frac{1}{2}}(1+t)^{-\frac{1}{2}}$ when $n=2$, and $(1+t)^{-\frac{1}{2}}$ when $n\geqslant 3$  of the time-dependent coefficient of $\psi_1$ if we compare with those in Theorem \ref{Thm_Energy}. Hence, we may interpret the second profile $J_1\psi_1$ by the diffusion-wave equations. The asymptotic profile of the linearized Kuznetsov's equation also can be shown by the diffusion-wave equations (see \cite{IkehataTodorovaYordanov2013,Ikehata2014}). For this reason, we may conjecture some relations between Blackstock's model and Kuznetsov's equation. We will give some possible answers later in Remark \ref{Rem_new_02}.
    	\item Finally, we may decompose the operator by $J_2=:\Delta^{-1}(J_{2,1}-J_{2,2})$, where $J_{2,1}$ is another solution operator associated with the diffusion-wave \eqref{diffusion-wave} and $J_{2,2}$ is a solution operator associated with the heat equation \eqref{Heat}. By subtracting the function $J_2\psi_2$ in Theorem \ref{Thm_Asym}, we observe an improvement of decay rate $(1+t)^{-\frac{1}{2}}$ for $n\geqslant 5$, $(1+t)^{-\frac{1}{2}}(\ln(\mathrm{e}+t))^{-\frac{1}{2}}$ for $n=4$, $(1+t)^{-\frac{3}{4}}$ for $n=3$,  of the time-dependent coefficient of $\psi_2$ in comparison to that in Theorem \ref{Thm_Energy}. Namely, the third profile $J_2\psi_2$ can be interpreted as a combination of heat equation and diffusion-wave equations. Here, the Laplace operator with negative power can be understood as additional second-order spatial derivatives acting on the evolution equations.
    \end{itemize}

    \begin{remark}\label{Rem_new_02}
    	The large time profile for Blackstock's model \eqref{Linear_Blackstock_Eq} can be described by the linear Kuznetsov's model with the same diffusivity of sound in \eqref{Linear_Blackstock_Eq} as follows:
    	\begin{align}\label{Linear_Kuznet_Eq_2}
    		\begin{cases}
    			\varphi_{tt}-\Delta\varphi-\delta\Delta\varphi_t=0,\\
    			\varphi(0,x)=\varphi_0(x),\ \varphi_t(0,x)=\varphi_1(x),
    		\end{cases}
    	\end{align}
    	for $x\in\mb{R}^n$, $t>0$, under the following assumptions for initial data: $\varphi_0\equiv\psi_0$, $\varphi_1\equiv\psi_1$ and $\psi_2=C_0\Delta\psi_0+C_1\Delta\psi_1$ with real constants $C_0,C_1$ such that $C_0C_1\neq0$. Precisely, by assuming $\psi_0,\psi_1\in H^2\cap L^{1,1}$ the error estimate
    	\begin{align} \label{estimate varphi -psi}
    		\ml{D}^{-1}_n(t)\|(\psi-\varphi)(t,\cdot)\|_{L^2}\lesssim\ml{B}_n(t) (\|\psi_0\|_{L^2\cap L^{1,1}}+\|\psi_1\|_{L^2\cap L^{1,1}})
    	\end{align}
    	holds for $t\gg1$, where the time-dependent function $\ml{B}_n(t)$ is $t^{-\frac{3}{4}}$ if $n=1$, $(t\ln t)^{-\frac{1}{2}}$ if $n=2$, and $t^{-\frac{1}{2}}$ if $n\geqslant 3$. The proof of \eqref{estimate varphi -psi} is based on asymptotic profiles for \eqref{Linear_Blackstock_Eq} and \eqref{Linear_Kuznet_Eq_2}. The time-decay function $\ml{B}_n(t)$ for all $n\geqslant 1$ implies that the gap between two solutions tends to zero for large time.
    \end{remark}

    \begin{theorem}\label{Thm_Optimal}
    	Let us assume initial data  $\psi_0,\psi_1,\psi_2\in L^2\cap L^{1,1} $ and $|P_{\psi_2}|\neq0$. Then, the solution to Blackstock's model \eqref{Linear_Blackstock_Eq} fulfills the following optimal estimate:
    	\begin{align*}
    		\widetilde{\ml{D}}_n(t)|P_{\psi_2}|\lesssim\|\psi(t,\cdot)\|_{L^2 }\lesssim \widetilde{\ml{D}}_n(t)(\|\psi_0\|_{L^2\cap L^{1,1}}+\|\psi_1\|_{L^2\cap L^{1,1}}+\|\psi_2\|_{L^2\cap L^{1,1}});
    	\end{align*}
    	%
    	for $n\geqslant 3$ and $t\gg1$.
    \end{theorem}
    \begin{remark}
    	According to Theorem \ref{Thm_Optimal} and concerning $t\gg1$, we realize that the growth or decay rates for the estimates of $\|\psi(t,\cdot)\|_{L^2}$ from the above and the below are the same if $n\geqslant 3$. Moreover, $\psi_j\in L^{1,1}$ implies $|P_{\psi_j}|<\infty$ for any $j=0,1,2$. That is to say that the growth or decay estimates stated in Theorem \ref{Thm_Optimal} for $n\geqslant 3$ are optimal in the framework of weighted $L^1$ space.
    \end{remark}
    \begin{remark}
    	We restrict ourselves to consider lower bound estimates in the case $n\geqslant 3$. Indeed, in \eqref{psi - Jj psij} for $n=1,2$ we may not control the norm $\|\chi_{\intt}(\xi)|\xi|^{-1}\mathrm{e}^{-c|\xi|^2t}\|_{L^2}$ by means of Lemma \ref{Lemma_Basic_1} due to the not summable singularity in $|\xi|=0$. This lack will not allow us to continue the estimate in \eqref{Optimal_01} below for $n=1,2$.
    \end{remark}
    \begin{remark}
    	Considering the case $|P_{\psi_2}|\not\equiv0$ for $n=3,4$, we found a new blow-up phenomenon appearing in the linearized Blackstock's model \eqref{Linear_Blackstock_Eq}, namely, the solution in the $L^2$ norm blows up as $t\to \infty$. Precisely, the growth rate is $t^{\frac{1}{2}}$ if $n=3$ and $(\ln t)^{\frac{1}{2}}$ if $n=4$ for large time.
    \end{remark}
    \begin{proof}
    	First of all, the upper bound estimate has been given in Theorem \ref{Thm_Energy}. In order to find the lower bound estimate, we use Proposition \ref{Prop_02} as well as the Minkowski inequality to get
    	\begin{align}\label{Optimal_01}
    		\|\psi(t,\cdot)\|_{L^2}
    		&\geqslant\|\chi_{\intt}(D)\ml{F}_{\xi\to x}^{-1}(\widehat{J}_2(t,|\xi|))\|_{L^2}|P_{\psi_2}|\notag\\
    		&\quad\ -\|\chi_{\intt}(D)(\psi(t,\cdot)-\ml{F}^{-1}_{\xi\to x}(\widehat{J}_2(t,|\xi|))P_{\psi_2})\|_{L^2}\notag\\
    		&\gtrsim \widetilde{\ml{D}}_n(t)|P_{\psi_2}|-\|\chi_{\intt}(\xi)(\widehat{\psi}(t,\xi)-\widehat{J}_2(t,|\xi|)P_{\psi_2})\|_{L^2}
    	\end{align}
    	for $t\gg1$ and $n\geqslant 3$. Since the composition of initial data such that
    	\begin{align*}
    		\widehat{\psi}_2(\xi)=P_{\psi_2}+A_2(\xi)-iB_2(\xi)
    	\end{align*} with $A_2(\xi):=\int_{\mb{R}^n}\psi_2(x)(\cos(x\cdot\xi)-1)\mathrm{d}x$ and $B_2(\xi):=\int_{\mb{R}^n}\psi_2(x)\sin(x\cdot\xi)\mathrm{d}x$. By employing Lemma 2.2 in \cite{Ikehata2014}, the auxiliary functions can be estimated by
    	\begin{align*}
    		|A_2(\xi)|+|B_2(\xi)|\lesssim|\xi|\,\|\psi_2\|_{L^{1,1}}.
    	\end{align*}
    	Therefore, it yields
    	\begin{align*}
    		&\|\chi_{\intt}(\xi)(\widehat{\psi}(t,\xi)-\widehat{J}_2(t,|\xi|)P_{\psi_2})\|_{L^2}\notag\\
    		&\qquad=\|\chi_{\intt}(\xi)(\widehat{\psi}(t,\xi)-\widehat{J}_2(t,|\xi|)(\widehat{\psi}_2(\xi)-A_2(\xi)+iB_2(\xi)))\|_{L^2}\notag\\
    		&\qquad\lesssim\|\chi_{\intt}(\xi)(\widehat{\psi}(t,\xi)-\widehat{J}_2(t,|\xi|)\widehat{\psi}_2(\xi))\|_{L^2}+\|\chi_{\intt}(\xi)|\xi|\widehat{J}_2(t,|\xi|)\|_{L^2}\|\psi_2\|_{L^{1,1}}\notag\\
    		&\qquad\lesssim t^{\frac{1}{2}-\frac{n}{4}}(\|\psi_0\|_{L^2\cap L^{1,1}}+\|\psi_1\|_{L^2\cap L^{1,1}}+\|\psi_2\|_{L^2\cap L^{1,1}})
    	\end{align*}
    	for $t\gg1$ and $n\geqslant 3$, where we used a corollary of Theorem \ref{Thm_Asym}, i.e.
    	\begin{align*}
    		\|(\psi-J_2\psi_2)(t,\cdot)\|_{L^2}\lesssim t^{\frac{1}{2}-\frac{n}{4}}(\|\psi_0\|_{L^2\cap L^{1,1}}+\|\psi_1\|_{L^2\cap L^{1,1}}+\|\psi_2\|_{L^2\cap L^{1,1}}).
    	\end{align*}
    	Finally, with $t\gg1$ and $|P_{\psi_2}|\not\equiv0$, we finish the proof of this optimal estimate.
    \end{proof}

    \section{Singular limits problem for the vanishing thermal diffusivity}\label{Sect_singular}
    Let us recall that $\bar{\kappa}\downarrow0$ if and only if $\kappa\downarrow0$. We will consider the limiting process from the linear Blackstock's model \eqref{Linear_Blackstock_Eq} to the linear viscoelastic damped wave equation \eqref{Linear_Kuznet_Eq} as the (modified) thermal diffusivity tends to zero. Throughout this section, we assume that $\psi_0$ and $\psi_1$ are simultaneously nontrivial to guarantee the nontrivial solution to the Cauchy problem \eqref{Linear_Kuznet_Eq}. Note that the prescribed initial data are independent of $\kappa$.
    

    \subsection{Inhomogeneous Blackstock's model in the phase space}\label{Subsec_Inhomo}
    Actually, the study for the inhomogeneous Blackstock's model with Becker's assumption and those without Becker's assumption are quite different. In the thesis \cite{Brunnhuber-2015}, the author regarded the  Blackstock's model with Becker's assumption \eqref{Becker_Assum} as a heat equation with the variable $(\psi_{tt}-\Delta\psi-\delta \Delta\psi_t)$ or the Kuznetsov's model with the variable $(\psi_t-\kappa\Delta\psi)$. Nevertheless, this approach does not hold anymore without Becker's assumption \eqref{Becker_Assum} because the additional term $-\kappa(\delta-b\gamma\nu)\Delta^2\psi_t$ prevents such factorization of the equation. So, we will deal with the general case in this subsection by applying suitable energy methods in the phase space.
    
    Throughout this part, we will consider some estimates of the solution $\hat{u}=\hat{u}(t,\xi)$ to the following inhomogeneous Cauchy problem for Blackstock's model in the phase space:
    \begin{align}\label{Inhomo_Cauchy_Phase}
    	\begin{cases}
    		\hat{u}_{ttt}+(\delta+\kappa)|\xi|^2\hat{u}_{tt}+(1+\tilde{\gamma}|\xi|^2)|\xi|^2\hat{u}_t+\kappa|\xi|^4\hat{u}=\hat{f},\\
    		\hat{u}(0,\xi)=\hat{u}_t(0,\xi)=0,\ \hat{u}_{tt}(0,\xi)=\hat{u}_2(\xi),
    	\end{cases}
    \end{align}
    for $\xi\in\mb{R}^n$, $t>0$, with $\tilde{\gamma}=\gamma b\nu\kappa$, where $\hat{f}=\hat{f}(t,\xi)$ is a suitable function as a source term.
    \begin{prop}\label{Prop_Energy_Fourier_01} The solution $\hat{u}$ to the inhomogeneous Cauchy problem \eqref{Inhomo_Cauchy_Phase} fulfills the next estimates:
    	\begin{align*}
    		& \left|\hat{u}_{tt}+b\nu|\xi|^2\hat{u}_t+\frac{|\xi|^2}{2}\hat{u}\right|^2+\frac{|\xi|^4}{4}\left|\frac{2\kappa(\gamma+2)}{2\kappa b\nu(\gamma+2)|\xi|^2+1}\hat{u}_t +\hat{u}\right|^2\\
    		&\qquad+\frac{|\xi|^2}{4\kappa b\nu (\gamma+2)|\xi|^2+2}|\hat{u}_t|^2\leqslant |\hat{u}_2|^2+\frac{4+\gamma}{4\kappa\gamma|\xi|^2}\int_0^t|\hat{f}(s,\xi)|^2\mathrm{d}s.
    	\end{align*}
    \end{prop}
    \begin{proof}
    	We introduce the new functions $\hat{v}=\hat{v}(t,\xi)$ and $\hat{w}=\hat{w}(t,\xi)$ such that
    	\begin{align*}
    		\hat{v}=\hat{u}_t\ \ \mbox{and}\ \ \hat{w}=\hat{v}_t=\hat{u}_{tt}.
    	\end{align*}
    	At this time, we immediately rewrite the third-order differential equation in \eqref{Inhomo_Cauchy_Phase} by
    	\begin{align}\label{Eq_S}
    		\begin{cases}
    			\hat{u}_t=\hat{v},\\
    			\hat{v}_t=\hat{w},\\
    			\hat{w}_t=-(\delta+\kappa)|\xi|^2\hat{w}-(1+\tilde{\gamma}|\xi|^2)|\xi|^2\hat{v}-\kappa|\xi|^4\hat{u}+\hat{f}.
    		\end{cases}
    	\end{align}
    	To begin with the construction of energy, we say $k_1\in(0,1)$ to be a real constant independent of $\kappa$ that will be fixed later. By proceeding \eqref{Eq_S}$_3+b\nu|\xi|^2\times$\eqref{Eq_S}$_2+k_1|\xi|^2\times$\eqref{Eq_S}$_1$, we get
    	\begin{align}\label{Eq_02}
    		(\hat{w}+b\nu|\xi|^2\hat{v}+k_1|\xi|^2\hat{u})_t=-\gamma\kappa|\xi|^2\hat{w}-(1-k_1+\tilde{\gamma}|\xi|^2)|\xi|^2\hat{v}-\kappa|\xi|^4\hat{u}+\hat{f}.
    	\end{align}
    	Multiplying \eqref{Eq_02} by $\bar{\hat{w}}+b\nu|\xi|^2\bar{\hat{v}}+k_1|\xi|^2\bar{\hat{u}}$ and taking the real part of the resultant equation
    	\begin{align}\label{S_1}
    		&\frac{1}{2}\frac{\mathrm{d}}{\mathrm{d}t}\left(|\hat{w}+b\nu|\xi|^2\hat{v}+k_1|\xi|^2\hat{u}|^2\right)\notag\\
    		&\qquad=-\kappa k_1|\xi|^6|\hat{u}|^2-b\nu(1-k_1+\tilde{\gamma}|\xi|^2)|\xi|^4|\hat{v}|^2-\gamma\kappa|\xi|^2|\hat{w}|^2\notag\\
    		&\qquad\quad-\left(\kappa b\nu|\xi|^2+k_1(1-k_1+\tilde{\gamma}|\xi|^2)\right)|\xi|^4\Re(\bar{\hat{u}}\hat{v})\notag\\
    		&\qquad\quad-\kappa(1+k_1\gamma)|\xi|^4\Re(\bar{\hat{u}}\hat{w})-(1-k_1+2\tilde{\gamma}|\xi|^2)|\xi|^2\Re(\bar{\hat{v}}\hat{w})\notag\\
    		&\qquad\quad+\Re\left(\hat{f}(\bar{\hat{w}}+b\nu|\xi|^2\bar{\hat{v}}+k_1|\xi|^2\bar{\hat{u}})\right).
    	\end{align}
    	From \eqref{Eq_S}$_1$ and \eqref{Eq_S}$_2$, it holds
    	\begin{align}\label{S_2}
    		\frac{1}{2}\frac{\mathrm{d}}{\mathrm{d}t}\left(k_2|\hat{u}+k_3\hat{v}|^2\right)&=k_2\Re\left((\hat{v}+k_3\hat{w})(\bar{\hat{u}}+k_3\bar{\hat{v}})\right)\notag\\
    		&=k_2\Re(\bar{\hat{u}}\hat{v})+k_2k_3|\hat{v}|^2+k_2k_3\Re(\bar{\hat{u}}\hat{w})+k_2k_3^2\Re(\bar{\hat{v}}\hat{w}),
    	\end{align}
    	with real parameters $k_2$ and $k_3$ satisfying
    	\begin{align*}
    		k_2&=\left(\kappa b\nu|\xi|^2+k_1(1-k_1+\tilde{\gamma}|\xi|^2)\right)|\xi|^4>0,\\
    		k_3&=\frac{\kappa(1+k_1\gamma)}{\kappa b\nu|\xi|^2+k_1(1-k_1+\tilde{\gamma}|\xi|^2)}>0,
    	\end{align*}
    	so that $k_2k_3=\kappa(1+k_1\gamma)|\xi|^4$. Moreover, the equation \eqref{Eq_S}$_2$ shows
    	\begin{align*}
    		\frac{1}{2}\frac{\mathrm{d}}{\mathrm{d}t}\left[ \left((1-k_1+2\tilde{\gamma}|\xi|^2)|\xi|^2-k_2k_3^2\right)|\hat{v}|^2 \right]=\left((1-k_1+2\tilde{\gamma}|\xi|^2)|\xi|^2-k_2k_3^2\right)\Re(\bar{\hat{v}}\hat{w}).
    	\end{align*}
    	Therefore, we summarize the derived equalities to see
    	\begin{align*}
    		&\frac{1}{2}\frac{\mathrm{d}}{\mathrm{d}t}\left[|\hat{w}+b\nu|\xi|^2\hat{v}+k_1|\xi|^2\hat{u}|^2+k_2|\hat{u}+k_3\hat{v}|^2+ \left((1-k_1+2\tilde{\gamma}|\xi|^2)|\xi|^2-k_2k_3^2\right)|\hat{v}|^2   \right]\\
    		&\qquad=-\kappa k_1|\xi|^6|\hat{u}|^2-\left(b\nu(1-k_1+\tilde{\gamma}|\xi|^2)|\xi|^4-k_2k_3\right)|\hat{v}|^2-\gamma\kappa|\xi|^2|\hat{w}|^2\\
    		&\qquad\quad+\Re\left(\hat{f}(\bar{\hat{w}}+b\nu|\xi|^2\bar{\hat{v}}+k_1|\xi|^2\bar{\hat{u}})\right).
    	\end{align*}
    	Let us now employ Cauchy's inequality as follows:
    	\begin{align}
    		k_1|\xi|^2\Re(\hat{f}\bar{\hat{u}})&\leqslant \kappa k_1|\xi|^6|\hat{u}|^2+\frac{k_1}{4\kappa|\xi|^2}|\hat{f}|^2, \label{S_3}\\
    		b\nu|\xi|^2\Re(\hat{f}\bar{\hat{v}})&\leqslant b\nu\tilde{\gamma}|\xi|^6|\hat{v}|^2+\frac{1}{4\gamma\kappa|\xi|^2}|\hat{f}|^2,\label{S_4}\\
    		\Re(\hat{f}\bar{\hat{w}})&\leqslant\gamma\kappa|\xi|^2|\hat{w}|^2+\frac{1}{4\gamma\kappa|\xi|^2}|\hat{f}|^2.\label{S_5}
    	\end{align}
    	Namely,
    	\begin{align}\label{Eq_03}
    		&\frac{1}{2}\frac{\mathrm{d}}{\mathrm{d}t}\left[|\hat{w}+b\nu|\xi|^2\hat{v}+k_1|\xi|^2\hat{u}|^2+k_2|\hat{u}+k_3\hat{v}|^2+ \left((1-k_1+2\tilde{\gamma}|\xi|^2)|\xi|^2-k_2k_3^2\right)|\hat{v}|^2   \right]\notag\\
    		&\qquad\leqslant-\left(b\nu(1-k_1+\tilde{\gamma}|\xi|^2)|\xi|^4-k_2k_3-b\nu\tilde{\gamma}|\xi|^6\right)|\hat{v}|^2+\frac{2+\gamma k_1}{4\kappa\gamma |\xi|^2}|\hat{f}|^2\notag\\
    		&\qquad\leqslant \frac{2+\gamma k_1}{4\kappa\gamma |\xi|^2}|\hat{f}|^2
    	\end{align}
    	for $0<\kappa\ll 1$. We notice that
    	\begin{align*}
    		(1-k_1+2\tilde{\gamma}|\xi|^2)|\xi|^2-k_2k_3^2=\frac{k_4|\xi|^2}{(\kappa b\nu+k_1\tilde{\gamma})|\xi|^2+k_1(1-k_1)}
    	\end{align*}
    	with the positive constant
    	\begin{align*}
    		k_4&=k_1(1-k_1)^2+\left((1-k_1)(\kappa b\nu+3k_1\tilde{\gamma})-\kappa^2(1+k_1\gamma)^2\right)|\xi|^2+2\tilde{\gamma}(\kappa b\nu+k_1\tilde{\gamma})|\xi|^4\\ &\geqslant k_1(1-k_1)^2
    	\end{align*}
    	for $0<\kappa\ll 1$.
    	Here, we avoid too many losses of $\kappa$  in the convergence result. For this reason, we dropped the second and third (positive) terms in the last chain. 
    	Since $k_2\geqslant k_1(1-k_1)|\xi|^4$, we integrate \eqref{Eq_03} over $[0,t]$ to derive
    	\begin{align*}
    		&|\hat{w}+b\nu|\xi|^2\hat{v}+k_1|\xi|^2\hat{u}|^2+k_1(1-k_1)|\xi|^4|\hat{u}+k_3\hat{v}|^2\\
    		&\qquad+ \frac{k_1(1-k_1)^2|\xi|^2}{(\kappa b\nu+k_1\tilde{\gamma})|\xi|^2+k_1(1-k_1)}|\hat{v}|^2\leqslant|\hat{u}_2|^2+\frac{2+\gamma k_1}{2\kappa\gamma|\xi|^2}\int_0^t|\hat{f}(s,\xi)|^2\mathrm{d}s,
    	\end{align*}
    	where we used $\hat{u}(0,\xi)=\hat{v}(0,\xi)=0$ for all $\xi\in\mb{R}^n$. Finally, by choosing $k_1=1/2$, we complete the proof of this proposition.
    \end{proof}
    
    \begin{prop}\label{Prop_Potential_1}The solution $\hat{u}$ to the inhomogeneous Cauchy problem \eqref{Inhomo_Cauchy_Phase} fulfills for any $k_1\in(0,1)$ the next estimates:
    	\begin{align*}
    		|\hat{u}|^2\leqslant\frac{1-k_1+2\tilde{\gamma}|\xi|^2}{k_1(1-k_1)^2|\xi|^4}|\hat{u}_2|^2+\frac{(1-k_1+2\tilde{\gamma}|\xi|^2)(2+\gamma k_1)}{4\kappa k_1(1-k_1)^2\gamma|\xi|^6}\int_0^t|\hat{f}(s,\xi)|^2\mathrm{d}s.
    	\end{align*}
    \end{prop}
    \begin{proof}
    	Let us take $k_2k_3=\kappa(1+k_1\gamma)|\xi|^4$ and $k_2k_3^2=(1-k_1+2\tilde{\gamma}|\xi|^2)|\xi|^2$ in \eqref{S_1} and \eqref{S_2}, in other words,
    	\begin{align}\label{k_2,k_3}
    		k_2=\frac{\kappa^2(1+k_1\gamma)^2|\xi|^6}{1-k_1-2\tilde{\gamma}|\xi|^2}>0\ \ \mbox{as well as}\ \ k_3=\frac{1-k_1+2\tilde{\gamma}|\xi|^2}{\kappa(1+k_1\gamma)|\xi|^2}>0.
    	\end{align}
    	It leads to
    	\begin{align*}
    		&\frac{1}{2}\frac{\mathrm{d}}{\mathrm{d}t}\left(|\hat{w}+b\nu|\xi|^2\hat{v}+k_1|\xi|^2\hat{u}|^2+k_2|\hat{u}+k_3\hat{v}|^2\right)\\
    		&\qquad=-\kappa k_1|\xi|^6|\hat{u}|^2-\left(b\nu(1-k_1+\tilde{\gamma}|\xi|^2)-\kappa(1+k_1\gamma)\right)|\xi|^4|\hat{v}|^2-\gamma\kappa|\xi|^2|\hat{w}|^2\\
    		&\qquad\quad-\underbrace{\left(\kappa b\nu|\xi|^2+k_1(1-k_1+\tilde{\gamma}|\xi|^2)-\frac{\kappa^2(1+k_1\gamma)^2|\xi|^2}{1-k_1+2\tilde{\gamma}|\xi|^2}\right)}_{=:k_5}|\xi|^4\Re(\bar{\hat{u}}\hat{v})\\
    		&\qquad\quad+\Re\left(\hat{f}(\bar{\hat{w}}+b\nu|\xi|^2\bar{\hat{v}}+k_1|\xi|^2\bar{\hat{u}})\right).
    	\end{align*}
    	Additionally, we know from \eqref{Eq_S}$_1$ that
    	\begin{align*}
    		\frac{1}{2}\frac{\mathrm{d}}{\mathrm{d}t}\left(k_5|\xi|^4|\hat{u}|^2 \right)=k_5|\xi|^4\Re(\bar{\hat{u}}\hat{v}).
    	\end{align*}
    	By using the derived estimates \eqref{S_3}-\eqref{S_5}, we get
    	\begin{align}\label{*}
    		\frac{1}{2}\frac{\mathrm{d}}{\mathrm{d}t}\left(|\hat{w}+b\nu|\xi|^2\hat{v}+k_1|\xi|^2\hat{u}|^2+k_2|\hat{u}+k_3\hat{v}|^2+k_5|\xi|^4|\hat{u}|^2  \right)\leqslant\frac{2+\gamma k_1}{4\kappa\gamma|\xi|^2}|\hat{f}|^2.
    	\end{align}
    	Actually, we can compute
    	\begin{align}\label{k_5}
    		k_5&\!=\!\frac{k_1(1-k_1)^2+\kappa\left((1-k_1)b\nu(3\gamma k_1+1)-\kappa (1+k_1\gamma)^2\right)|\xi|^2+2\tilde{\gamma}(\kappa b\nu+k_1\tilde{\gamma})|\xi|^4}{1-k_1+2\tilde{\gamma}|\xi|^2}\notag\\
    		&\!\geqslant\!\frac{k_1(1-k_1)^2}{1-k_1+2\tilde{\gamma}|\xi|^2}
    	\end{align}
    	for $0<\kappa\ll 1$. As a consequence, it leads to
    	\begin{align*}
    		\frac{k_1(1-k_1)^2|\xi|^4}{1-k_1+2\tilde{\gamma}|\xi|^2}|\hat{u}|^2\leqslant|\hat{u}_2|^2+\frac{2+\gamma k_1}{4\kappa\gamma|\xi|^2}\int_0^t|\hat{f}(s,\xi)|^2\mathrm{d}s.
    	\end{align*}
    	It completes the proof.
    \end{proof}
    
    \subsection{Singular limits for some energy terms}
    Let us define the difference of solutions between Blackstock's model \eqref{Linear_Blackstock_Eq} and Kuznetsov's model \eqref{Linear_Kuznet_Eq} by $u=u(t,x)$ such that
    \begin{align}\label{Difference_01}
    	u=\psi-\psi^{(0)},
    \end{align}
    which is a solution to the inhomogeneous Blackstock's model. Then, we have the next global (in time) convergence result in $L^{\infty}([0,\infty),L^2)$ as $\kappa\downarrow 0$.
    \begin{theorem}\label{Thm_Singular_Energy}
    	Let us assume $\psi_0\in H^3\cap L^1$ and $\psi_1\in H^2\cap L^1$ such that
    	\begin{align}\label{Assumption_Singular_Energy}
    		\psi_2=b\nu\Delta\psi_1+\Delta\psi_0.
    	\end{align}
    	Then, the difference $u=u(t,x)$ defined in \eqref{Difference_01} satisfies the following estimates:
    	\begin{align*}
    		&\sup\limits_{t\in[0,\infty)}\left(\,\left\| u_{tt}(t,\cdot)-b\nu\Delta u_t(t,\cdot)-\tfrac{1}{2}\Delta u(t,\cdot) \right\|_{L^2}+\|u_t(t,\cdot)\|_{L^2}\right)\\
    		&\qquad\leqslant \sqrt{\kappa} C\left( \|\psi_0\|_{H^3\cap L^1}+\|\psi_1\|_{H^2\cap L^1} \right),
    	\end{align*}
    	where $C$ is a positive constant independent of $\kappa$.
    \end{theorem}
    \begin{remark}
    	The consistency condition on initial data \eqref{Assumption_Singular_Energy} coincides with (1.3) in the recent paper \cite{Kaltenbacher-Thalhammer-2018} for the Dirichlet boundary value problem of Blackstock's model.
    \end{remark}
    \begin{proof}
    	According to our definition for the difference, we can show $u$ is a solution to
    	\begin{align*}
    		\begin{cases}
    			(\partial_t-\kappa\Delta)(u_{tt}-\Delta u-\delta\Delta u_t)+\kappa(\gamma-1)(b\nu-\kappa)\Delta^2u_t=f,\\
    			u(0,x)=u_t(0,x)=u_{tt}(0,x)=0,
    		\end{cases}
    	\end{align*}
    	for $x\in\mb{R}^n$, $t>0$, since $\psi_2(x)-b\nu\Delta\psi_1(x)-\Delta\psi_0(x)=0$, where the source term $f=f(t,x;\psi^{(0)})$ is
    	\begin{align*}
    		f&=-(\partial_t-\kappa\Delta)\left(\psi^{(0)}_{tt}-\Delta \psi^{(0)}-(b\nu+(\gamma-1)\kappa)\Delta \psi^{(0)}_t\right)-\kappa(\gamma-1)(b\nu-\kappa)\Delta^2\psi^{(0)}_t\\
    		&=\kappa(\gamma-1)\Delta^2\psi^{(0)}.
    	\end{align*}
    	In the above deduction, we used the limit equation in \eqref{Linear_Kuznet_Eq} twice. Therefore, we can apply the theory in Subsection \ref{Subsec_Inhomo} by taking $\hat{f}=\kappa (\gamma-1)|\xi|^4\widehat{\psi}^{(0)}$ and $\hat{u}_2=0$,
    	since our assumption \eqref{Assumption_Singular_Energy}. From Proposition \ref{Prop_Energy_Fourier_01}, one arrives at
    	\begin{align}\label{Eq_04}
    		&\left|\hat{u}_{tt}+b\nu|\xi|^2\hat{u}_t+\frac{|\xi|^2}{2}\hat{u}\right|^2+\frac{|\xi|^4}{4}\left|\frac{2\kappa(\gamma+2)}{2\kappa b\nu(\gamma+2)|\xi|^2+1}\hat{u}_t +\hat{u}\right|^2\notag\\
    		&\qquad+\frac{|\xi|^2}{4\kappa b\nu (\gamma+2)|\xi|^2+2}|\hat{u}_t|^2\leqslant \kappa\frac{(4+\gamma)(\gamma-1)^2}{4\gamma}\int_0^t\left||\xi|^3\widehat{\psi}^{(0)}(s,\xi)\right|^2\mathrm{d}s.
    	\end{align}
    	Following the similar approaches as those of Theorem 14.3.3 as well as Corollary 14.3.1 in the book \cite{Ebert-Reissig-book}, we can show the solution to the Cauchy problem \eqref{Linear_Kuznet_Eq} satisfying
    	\begin{align*}
    		&\|\,|\xi|^k\widehat{\psi}^{(0)}(t,\xi)\|_{L^2}^2=\|\,|D|^k\psi^{(0)}(t,\cdot)\|_{L^2}^2\\
    		&\qquad\leqslant C(1+t)^{-k+1-\frac{n}{2}}\left( \|\psi_0\|^2_{H^k\cap L^1}+\|\psi_1\|^2_{H^{k-1}\cap L^1} \right)
    	\end{align*}
    	for $k\in[1,\infty)$. Consequently, the right-hand side of \eqref{Eq_04} implies
    	\begin{align*}
    		&\kappa\frac{(4+\gamma)(\gamma-1)^2}{4\gamma}\int_0^t\|\,|\xi|^3\widehat{\psi}^{(0)}(s,\cdot)\|^2_{L^2}\mathrm{d}s\\
    		&\qquad\leqslant\kappa C\int_0^t(1+s)^{-2-\frac{n}{2}}\mathrm{d}s\left( \|\psi_0\|^2_{H^3\cap L^1}+\|\psi_1\|^2_{H^2\cap L^1} \right)\\
    		&\qquad\leqslant\kappa C\left( \|\psi_0\|^2_{H^3\cap L^1}+\|\psi_1\|^2_{H^2\cap L^1} \right).
    	\end{align*}
    	We find with the Plancherel theorem and the last estimate that
    	\begin{align*}
    		\left\| u_{tt}(t,\cdot)-b\nu\Delta u_t(t,\cdot)-\tfrac{1}{2}\Delta u(t,\cdot) \right\|_{L^2}^2\leqslant \kappa C\left( \|\psi_0\|^2_{H^3\cap L^1}+\|\psi_1\|^2_{H^2\cap L^1} \right).
    	\end{align*}
    	Moreover, from \eqref{Eq_04} again, one observes
    	\begin{align*}
    		|\hat{u}_t|^2\leqslant\kappa^2 C\int_0^t|\,|\xi|^3\widehat{\psi}^{(0)}(s,\xi)|^2\mathrm{d}s+\kappa C\int_0^t|\,|\xi|^2\widehat{\psi}^{(0)}(s,\xi)|^2\mathrm{d}s.
    	\end{align*}
    	It leads to
    	\begin{align*}
    		\|u_t(t,\cdot)\|_{L^2}^2&\leqslant \kappa^2 C\int_0^t(1+s)^{-2-\frac{n}{2}}\mathrm{d}s\left( \|\psi_0\|^2_{H^3\cap L^1}+\|\psi_1\|^2_{H^2\cap L^1} \right)\notag\\
    		&\quad+\kappa C\int_0^t(1+s)^{-1-\frac{n}{2}}\mathrm{d}s\left( \|\psi_0\|^2_{H^2\cap L^1}+\|\psi_1\|^2_{H^1\cap L^1} \right)\notag\\
    		&\leqslant \kappa C\left( \|\psi_0\|^2_{H^3\cap L^1}+\|\psi_1\|^2_{H^2\cap L^1} \right).
    	\end{align*}
    	The proof is complete.
    \end{proof}

    \begin{coro}\label{Coro_energy}
    	Let us assume $\langle D\rangle^{s_0+1}\psi_0\in L^1$ and $\langle D\rangle^{s_0}\psi_1\in L^1$ with $s_0>n+2$ such that \eqref{Assumption_Singular_Energy} holds.
    	Then, the difference $u=u(t,x)$ defined in \eqref{Difference_01} satisfies the following estimates:
    	\begin{align*}
    		&\sup_{t\in[0,\infty)} 
    		\left(\,\left\| u_{tt}(t,\cdot)-b\nu\Delta u_t(t,\cdot)-\tfrac{1}{2}\Delta u(t,\cdot)\right\|_{L^\infty} +\|u_t(t,\cdot)\|_{L^\infty}\right)\\
    		&\qquad\leqslant \sqrt{\kappa} C\left(\|\langle D\rangle^{s_0+1}\psi_0\|_{L^1}+\|\langle D\rangle^{s_0}\psi_1\|_{L^1}\right),
    	\end{align*}
    	where $C$ is a positive constant independent of $\kappa$.
    \end{coro}
    
    \begin{remark}
    	Comparing the previous corollary with Theorem \ref{Thm_Singular_Energy}, we see that, under suitable assumptions for the data, we have showed
    	\begin{align}\label{Convergence_Dervative}
    		\psi_t\to \psi^{(0)}_t\ \ \mbox{in}\  L^\infty([0,\infty),L^2)  \ \mbox{and in}  \ L^{\infty}([0,\infty), L^\infty )\ \ \mbox{as}\  \kappa\downarrow0
    	\end{align}
    	with the rate of convergence $\sqrt{\kappa}$.
    \end{remark}
    
    
    \begin{proof}
    	The combination of \eqref{Eq_04} and  the Hausdorff-Young inequality shows
    	\begin{align*}
    		&\left\| u_{tt}(t,\cdot)-b\nu\Delta u_t(t,\cdot)-\tfrac{1}{2}\Delta u(t,\cdot) \right\|_{L^{\infty}}\\
    		&\qquad\leqslant C\left\| \hat{u}_{tt}(t,\xi)+b\nu|\xi|^2 \hat{u}_t(t,\xi)+\tfrac{1}{2}|\xi|^2 \hat{u}(t,\xi) \right\|_{L^1}\\
    		&\qquad\leqslant \sqrt{\kappa} C\underbrace{\int_{\mb{R}^n}\left(\int_0^t\left|\,|\xi|^3\widehat{\psi}^{(0)}(s,\xi)\right|^2\mathrm{d}s\right)^{1/2}\mathrm{d}\xi}_{=:\ml{J}_0(t)}.
    	\end{align*}
    	Recalling from Lemma 2.4 of \cite{IkehataNatsume}, the solution $\widehat{\psi}^{(0)}$ has the pointwise estimate
    	\begin{align*}
    		|\xi|^3|\widehat{\psi}^{(0)}|\leqslant C\exp\left(-c\tfrac{|\xi|^2}{|\xi|^2+1}t\right)\left(|\xi|^3|\widehat{\psi}_0|+|\xi|^2|\widehat{\psi}_1|\right).
    	\end{align*}
    	It immediately leads to
    	\begin{align*}
    		\ml{J}_0(t)\leqslant C\int_{\mb{R}^3}\left(\int_0^t\exp\left(-c\tfrac{|\xi|^2}{|\xi|^2+1}s\right)\mathrm{d}s\right)^{1/2}\left(|\xi|^3|\widehat{\psi}_0|+|\xi|^2|\widehat{\psi}_1|\right)\mathrm{d}\xi.
    	\end{align*}
    	Next, for $|\xi|\leqslant 1$, we can see
    	\begin{align*}
    		\ml{J}_0(t;|\xi|\leqslant 1)&\leqslant C\int_{|\xi|\leqslant 1}\left(1-\mathrm{e}^{-c|\xi|^2t}\right)^{1/2}\left(|\xi|^2|\widehat{\psi}_0|+|\xi|\,|\widehat{\psi}_1|\right)\mathrm{d}\xi\\
    		& \leqslant C\left(\|\psi_0\|_{L^1}+\|\psi_1\|_{L^1}\right),
    	\end{align*}
    	and for $|\xi|\geqslant 1$, it holds
    	\begin{align*}
    		\ml{J}_0(t;|\xi|\geqslant 1)&\leqslant C\int_{|\xi|\geqslant 1}\left(\int_0^t\mathrm{e}^{-cs}\mathrm{d}s\right)^{1/2}\left(|\xi|^3|\widehat{\psi}_0|+|\xi|^2|\widehat{\psi}_1|\right)\mathrm{d}\xi\\
    		&\leqslant C\int_{|\xi|\geqslant 1}\langle \xi\rangle^{2-s_0}\mathrm{d}\xi\left(\|\langle\xi\rangle^{s_0+1}\widehat{\psi}_0\|_{L^{\infty}}+\|\langle\xi\rangle^{s_0}\widehat{\psi}_1\|_{L^{\infty}}\right)\\
    		&\leqslant C\int_{1}^{\infty}\langle r\rangle^{n+1-s_0}\mathrm{d}r\left(\|\langle D\rangle^{s_0+1}\psi_0\|_{L^1}+\|\langle D\rangle^{s_0}\psi_1\|_{L^1}\right)\\
    		&\leqslant C\left(\|\langle D\rangle^{s_0+1}\psi_0\|_{L^1}+\|\langle D\rangle^{s_0}\psi_1\|_{L^1}\right),
    	\end{align*}
    	since $n+1-s_0<-1$ for $s_0>n+2$. Namely, $\ml{J}_0(t)$ is bounded for any $t>0$
    	with $s_0>n+2$, which completes the proof of the estimate for the first component.
    	
    	Following the same approach as before, we are able to get
    	\begin{align*}
    		\|u_t(t,\cdot)\|_{L^{\infty}}&\leqslant \sqrt{\kappa}C\int_{\mb{R}^n}\left(\int_0^t\left||\xi|^2\widehat{\psi}^{(0)}(s,\xi)\right|^2\mathrm{d}s\right)^{1/2}\mathrm{d}\xi\\
    		&\quad+\kappa C\int_{\mb{R}^n}\left(\int_0^t\left||\xi|^3\widehat{\psi}^{(0)}(s,\xi)\right|^2\mathrm{d}s\right)^{1/2}\mathrm{d}\xi\\
    		&\leqslant \sqrt{\kappa}C\left(\|\langle D\rangle^{s_0+1}\psi_0\|_{L^1}+\|\langle D\rangle^{s_0}\psi_1\|_{L^1}\right)
    	\end{align*}
    	with $s_0>n+2$. The desired estimates are completed.
    \end{proof}
    \subsection{Singular limits for the acoustic velocity potential}\label{Sub_Singular_Solution}
    Let us turn to the singular limits for the solution itself, or the acoustic velocity potential.

    \begin{theorem}\label{Thm_Potential_2}
    	Let us assume $\psi_0\in H^2\cap L^1$ and $\psi_1\in H^1\cap L^1$ such that \eqref{Assumption_Singular_Energy} holds.  Then, the difference $u=u(t,x)$ defined in \eqref{Difference_01} satisfies the following estimate:
    	\begin{align*}
    		\sup\limits_{t\in[0,\infty)}\left((\ml{G}_1(1+t))^{-1}\|u(t,\cdot)\|_{L^2}\right)\leqslant\sqrt{\kappa} C\left( \|\psi_0\|_{H^2\cap L^1}+\|\psi_1\|_{H^1\cap L^1} \right),
    	\end{align*}
    	where $C$ is a positive constant independent of $\kappa$, and $\ml{G}_1(1+t):=(1+t)^{\frac{1}{4}}$ if $n=1$, $\ml{G}_1(1+t):=(\ln(\mathrm{e}+t))^{\frac{1}{2}}$ if $n=2$, $\ml{G}_1(1+t):=1$ if $n\geqslant 3$.
    \end{theorem}
    \begin{proof}
    	Following a similar procedure to that in Theorem \ref{Thm_Singular_Energy}, we use Proposition \ref{Prop_Potential_1} to get
    	\begin{align*}
    		\|u(t,\cdot)\|_{L^2}^2&\leqslant\kappa^2 C\int_0^t\|\,|\xi|^2\widehat{\psi}^{(0)}(s,\xi)\|_{L^2}^2\mathrm{d}s+\kappa C\int_0^t\|\,|\xi|\widehat{\psi}^{(0)}(s,\xi)\|_{L^2}^2\mathrm{d}s\\
    		&\leqslant \kappa^2 C\int_0^t(1+s)^{-1-\frac{n}{2}}\mathrm{d}s\left(\|\psi_0\|_{H^2\cap L^1}^2+\|\psi_1\|_{H^1\cap L^1}^2\right)\\
    		&\quad+\kappa C\int_0^t(1+s)^{-\frac{n}{2}}\mathrm{d}s\left(\|\psi_0\|_{H^1\cap L^1}^2+\|\psi_1\|_{L^2\cap L^1}^2\right)\\
    		&\leqslant\kappa C(\ml{G}_1(1+t))^{2}\left(\|\psi_0\|_{H^2\cap L^1}^2+\|\psi_1\|_{H^1\cap L^1}^2\right),
    	\end{align*}
    	which completes our proof.
    \end{proof}

    The next result can be proved by following Corollary \ref{Coro_energy} and using Proposition \ref{Prop_Potential_1}. Particularly, in the proof, we need to apply
    \begin{align*}
    	\sqrt{\kappa}\int_{|\xi|\leqslant 1}|\xi|^{-1}\left(1-\mathrm{e}^{-c|\xi|^2t}\right)^{1/2}\mathrm{d}\xi\leqslant \sqrt{\kappa}C\int_0^1 r^{n-2}\left(1-\mathrm{e}^{-cr^2t}\right)^{1/2}\mathrm{d}r\leqslant \sqrt{\kappa}C
    \end{align*}
    for any $n\geqslant 2$, and $\int_{|\xi|\geqslant 1}\langle\xi\rangle^{1-s_1}\mathrm{d}\xi\leqslant C$ for $s_1>n+1$.
    \begin{coro}\label{Coro_Solution}
    	Assume $n\geqslant 2$. Let us assume $\langle D\rangle^{s_1+1}\psi_0\in L^1$ and $\langle D\rangle^{s_1}\psi_1\in L^1$ with $s_1>n+1$ such that \eqref{Assumption_Singular_Energy} holds.
    	Then, the difference $u=u(t,x)$ defined in \eqref{Difference_01} satisfies the following estimates:
    	\begin{align*}
    		\sup_{t\in[0,\infty)} \|u(t,\cdot)\|_{L^\infty} 
    		\leqslant \sqrt{\kappa} C\left(\|\langle D\rangle^{s_1+1}\psi_0\|_{L^1}+\|\langle D\rangle^{s_1}\psi_1\|_{L^1}\right),
    	\end{align*}
    	where $C$ is a positive constant independent of $\kappa$.
    \end{coro}
    \begin{remark}
    	The last corollary indicates that
    	\begin{align}\label{Convergence_Potential}
    		\psi\to \psi^{(0)}\ \ \mbox{in}\ \ L^{\infty}([0,\infty)\times\mb{R}^n)\ \ \mbox{as}\ \ \kappa\downarrow0
    	\end{align}
    	with the rate of convergence $\sqrt{\kappa}$ for $n\geqslant 2$. Comparing \eqref{Convergence_Dervative} with \eqref{Convergence_Potential}, we observe that the rates of convergence between the acoustic velocity potential and its time-derivative are the same, but we lose the convergence result \eqref{Convergence_Potential} for $n=1$.
    \end{remark}
    
    \section{Higher-order profiles with the small thermal diffusivity}\label{Sec_Initial_Layer}
    As a continuation of the singular limit argument, which also can be understood by first-order asymptotic expansion with respect to $\kappa$, one may be interested in asymptotic behaviors of solution for small modified thermal diffusivity $\kappa$. The Blackstock's model \eqref{Linear_Blackstock_Eq}, indeed, can be re-organized as
    \begin{align}\label{A1}
    	\ml{L}\psi:=\left(\partial_t^3-(b\nu+\gamma\kappa)\Delta\partial_t^2-\Delta\partial_t+\gamma b\nu\kappa\Delta^2\partial_t+\kappa\Delta^2\right)\psi=0.
    \end{align}
    Later, we will use formal asymptotic analysis to find the equations of higher-order profiles with small $\kappa>0$ in Subsection \ref{Subsection.Formal.Expan}, and rigorously justify the second-order profiles in Subsection \ref{Subsection.Second-Order}.
    \subsection{Formal derivations of higher-order profiles}\label{Subsection.Formal.Expan} 
    In this subsection, we are devoted to using multi-scale analysis and formal asymptotic expansions to find the equations for the higher-order profiles. Based on multi-scale analysis, the solution of \eqref{A1} with $\kappa>0$ has the following expansions:
    \begin{align}\label{A2}
    	\psi(t,x)=\sum\limits_{j\geqslant0}\kappa^{j}\left(\psi^{I,j}(t,x)+\psi^{L,j}(z,x)\right)
    \end{align}
    with the variable defined as $z=\kappa t$, where each term in \eqref{A2} is assumed to be smooth, and the profiles $\psi^{L,j}(z,x)$ is decaying (or zero) when $z\to0$. Moreover, \eqref{A2} should satisfy three initial conditions as those in \eqref{Linear_Blackstock_Eq}.

    Plugging \eqref{A2} into \eqref{A1}, it follows
    \begin{align*}
    	0=&\sum\limits_{j\geqslant0}\kappa^{j}\left(\partial_t^3\psi^{I,j}+\kappa^{3}\partial_z^3\psi^{L,j}\right)-b\nu\sum\limits_{j\geqslant0}\kappa^{j}\left(\Delta\partial_t^2\psi^{I,j}+\kappa^{2}\Delta\partial_z^2\psi^{L,j}\right)\notag\\
    	&-\gamma\sum\limits_{j\geqslant0}\kappa^{j+1}\left(\Delta\partial_t^2\psi^{I,j}+\kappa^{2}\Delta\partial_z^2\psi^{L,j}\right)-\sum\limits_{j\geqslant0}\kappa^{j}\left(\Delta\partial_t\psi^{I,j}+\kappa\Delta\partial_z\psi^{L,j}\right)\notag\\
    	&+\gamma b\nu\sum\limits_{j\geqslant0}\kappa^{j+1}\left(\Delta^2\partial_t\psi^{I,j}+\kappa\Delta^2\partial_z\psi^{L,j}\right)+\sum\limits_{j\geqslant0}\kappa^{j+1}\left(\Delta^2\psi^{I,j}+\Delta^2\psi^{L,j}\right),
    \end{align*}
    where $\psi^{I,j}=\psi^{I,j}(t,x)$ and $\psi^{L,j}=\psi^{L,j}(z,x)$.
    
    \noindent \underline{The outer solution:} Firstly, we collect the terms carrying $\ml{O}(\kappa^{0})$, that is,
    \begin{align*}
    	\partial_t\left(\partial_t^2\psi^{I,0}-\Delta\psi^{I,0}-b\nu\Delta\partial_t\psi^{I,0}\right)=0.
    \end{align*}
    Let us take the consistency assumption
    \begin{align}\label{A4}
    	\partial_t^2\psi^{I,0}(0,x)=b\nu\Delta\partial_t\psi^{I,0}(0,x)+\Delta\psi^{I,0}(0,x)
    \end{align}
    to get
    \begin{align}\label{A5}
    	\partial_t^2\psi^{I,0}-\Delta\psi^{I,0}-b\nu\Delta\partial_t\psi^{I,0}=0.
    \end{align}
    That is to say the assumption \eqref{A4} always holds.\\
    \noindent \underline{The higher-order profiles of solution:} Next, we collect the terms carrying $\ml{O}(\kappa)$, that is,
    \begin{align*}
    	\partial_t\left(\partial_t^2\psi^{I,1}-\Delta\psi^{I,1}-b\nu\Delta\partial_t\psi^{I,1}\right)-\gamma\Delta\left(\partial_t^2\psi^{I,0}-\tfrac{1}{\gamma}\Delta\psi^{I,0}-b\nu\Delta\partial_t\psi^{I,0}\right)&\\
    	-\Delta\left(\partial_z\psi^{L,0}-\Delta\psi^{L,0}\right)&=0.
    \end{align*}
    From \eqref{A5}, the second term in the above is equal to $(1-\gamma)\Delta\partial_t(\partial_t\psi^{I,0}-b\nu\Delta\psi^{I,0})$. Now, we take $\psi^{I,1}$ satisfying 
    \begin{align}\label{A7}
    	\partial_t^2\psi^{I,1}-\Delta\psi^{I,1}-b\nu\Delta\partial_t\psi^{I,1}=(\gamma-1)\Delta\left(\partial_t\psi^{I,0}-b\nu\Delta\psi^{I,0}\right)
    \end{align}
    with the correction for $t=0$ also, we notice $\partial_z\psi^{L,0}-\Delta\psi^{L,0}=0$. Similarly, considering the term carrying $\ml{O}(\kappa^2)$, one may observe
    \begin{align*}
    	0&=\partial_t\left(\partial_t^2\psi^{I,2}-\Delta\psi^{I,2}-b\nu\Delta\partial_t\psi^{I,2}\right)-\gamma\Delta\left(\partial_t^2\psi^{I,1}-\tfrac{1}{\gamma}\Delta\psi^{I,1}-b\nu\Delta\partial_t\psi^{I,1}\right)\\
    	&\quad-b\nu\Delta\partial_z\left(\partial_z\psi^{L,0}-\gamma\Delta\psi^{L,0}\right)-\Delta\left(\partial_z\psi^{L,1}-\Delta\psi^{L,1}\right).
    \end{align*}
    By taking $\psi^{I,2}$ fulfilling
    \begin{align*}
    	\partial_t^2\psi^{I,2}-\Delta\psi^{I,2}-b\nu\Delta\partial_t\psi^{I,2}=\gamma\Delta\left(\partial_t\psi^{I,1}-b\nu\Delta\psi^{I,1}\right)-\int_0^t\Delta^2\psi^{I,1}(s,x)\mathrm{d}s
    \end{align*}
    for $t\geqslant0$, it leads to
    \begin{align*}
    	\partial_z\psi^{L,1}-\Delta\psi^{L,1}=(\gamma-1)b\nu\Delta\partial_z\psi^{L,0}.
    \end{align*}
    Finally, to determine the formal expansion, we should consider $\ml{O}(\kappa^{j+3})$ for any $j\geqslant0$ such that
    \begin{align*}
    	0&=\partial_t\left(\partial_t^2\psi^{I,j+3}-\Delta\psi^{I,j+3}-b\nu\Delta\partial_t\psi^{I,j+3}\right)\\
    	&\quad-\gamma\Delta\left(\partial_t^2\psi^{I,j+2}-\tfrac{1}{\gamma}\Delta\psi^{I,j+2}-b\nu\Delta\partial_t\psi^{I,j+2}\right)\\
    	&\quad+\partial_z^2\left(\partial_z\psi^{L,j}-\gamma\Delta\psi^{L,j}\right)-b\nu\Delta\partial_z\left(\partial_z\psi^{L,j+1}-\gamma\Delta\psi^{L,j+1}\right)\\
    	&\quad-\Delta\left(\partial_z\psi^{L,j+2}-\Delta\psi^{L,j+2}\right).
    \end{align*}
    Setting
    \begin{align*}
    	&\partial_t^2\psi^{I,j+3}-\Delta\psi^{I,j+3}-b\nu\Delta\partial_t\psi^{I,j+3}\\
    	&\qquad=\gamma\Delta\left(\partial_t\psi^{I,j+2}-b\nu\Delta\psi^{I,j+2}\right)-\int_0^t\Delta^2\psi^{I,j+2}(s,x)\mathrm{d}s
    \end{align*}
    for $t\geqslant0$, one gets
    \begin{align*}
    	\Delta\left(\partial_z\psi^{L,j+2}-\Delta\psi^{L,j+2}\right)&=\partial_z^2\left(\partial_z\psi^{L,j}-\gamma\Delta\psi^{L,j}\right)\\
    	&\quad-b\nu\Delta\partial_z\left(\partial_z\psi^{L,j+1}-\gamma\Delta\psi^{L,j+1}\right),
    \end{align*}
    especially,
    \begin{align*}
    	\partial_z\psi^{L,2}-\Delta\psi^{L,2}=(1-\gamma)\partial_z^2\psi^{L,0}-b\nu\partial_z\left(\partial_z\psi^{L,1}-\gamma\Delta\psi^{L,1}\right).
    \end{align*}
    
    \noindent \underline{The matching conditions:} Concerning initial conditions, since
    \begin{align*}
    	\psi_0(x)&=\psi^{I,0}(0,x)+\psi^{L,0}(0,x)+\sum\limits_{j\geqslant 1}\kappa^j\left(\psi^{I,j}(0,x)+\psi^{L,j}(0,x)\right),\\
    	\psi_1(x)&=\partial_t\psi^{I,0}(0,x)+\sum\limits_{j\geqslant1}\kappa^j\left(\partial_t\psi^{I,j}(0,x)+\partial_z\psi^{L,j-1}(0,x)\right),\\
    	\psi_2(x)&=\partial_t^2\psi^{I,0}(0,x)+\kappa\partial_t^2\psi^{I,1}(0,x)+\sum\limits_{j\geqslant2}\kappa^j\left(\partial_t^2\psi^{I,j}(0,x)+\partial_z^2\psi^{L,j-2}(0,x)\right),
    \end{align*}
    letting $\kappa\downarrow0$, it holds that
    \begin{itemize}
    	\item $\psi^{I,0}(0,x)=\psi_0(x)$, $\psi^{L,0}(0,x)=0$, and $\psi^{I,j}(0,x)=-\psi^{L,j}(0,x)$ for $j\geqslant 1$;
    	\item  $\partial_t\psi^{I,0}(0,x)=\psi_1(x)$, and $\partial_t\psi^{I,j}(0,x)=-\partial_z\psi^{L,j-1}(0,x)$ for $j\geqslant 1$;
    	\item $\partial_t^2\psi^{I,0}(0,x)=\psi_2(x)$, $\partial_t^2\psi^{I,1}(0,x)=0$, and $\partial_t^2\psi^{I,j}(0,x)=-\partial_z^2\psi^{L,j-2}(0,x)$ for $j\geqslant 2$.
    \end{itemize}
    Finally, we show some special relations among the profiles for the initial data. According to $\psi^{L,0}(0,x)=0$ and uniqueness of solution to heat equations, it yields the trivial solution $\psi^{L,0}(z,x)\equiv0$, $\partial_t\psi^{I,1}(0,x)=0$ and $\partial_t^2\psi^{I,2}(0,x)=0$. Furthermore, the dominant profile $\psi^{I,0}\equiv\psi^{(0)}$, thanks to \eqref{Assumption_Singular_Energy}, \eqref{A5} and the previous relations for initial data. Taking $t=0$ in \eqref{A7}, we get
    \begin{align*}
    	\Delta\psi^{I,1}(0,x)+b\nu\Delta\partial_t\psi^{I,1}(0,x)+(\gamma-1)\Delta(\psi_1(x)-b\nu\Delta\psi_0(x))=\partial_t^2\psi^{I,1}(0,x).
    \end{align*}
    With the aid of $\partial_t\psi^{I,1}(0,x)=\partial_t^2\psi^{I,1}(0,x)=0$, it implies
    \begin{align*}
    	\psi^{I,1}(0,x)=-\psi^{L,1}(0,x)=(1-\gamma)(\psi_1(x)-b\nu\Delta\psi_0(x)).
    \end{align*}
    
    Summarizing the last statements, we conclude the next result for formal higher-order profiles of solution to \eqref{Linear_Blackstock_Eq}, and the second-order profiles will be rigorously proved in Subsection \ref{Subsection.Second-Order}.
    \begin{prop}\label{Prop_A}
    	The solution $\psi=\psi(t,x)$ to the Cauchy problem for Blackstock's model \eqref{Linear_Blackstock_Eq}  with the consistency condition \eqref{Assumption_Singular_Energy} as well as small $\kappa$ formally have the following asymptotic expansions:
    	\begin{align*}
    		\psi(t,x)=\psi^{(0)}(t,x)+\sum\limits_{j\geqslant0}\kappa^{j+1}\left(\psi^{I,j+1}(t,x)+\psi^{L,j+1}(\kappa t,x)\right),
    	\end{align*}
    	where
    	\begin{itemize}
    		\item $\psi^{(0)}=\psi^{(0)}(t,x)$ is the solution to the viscoelastic damped waves \eqref{Linear_Kuznet_Eq};
    		\item $\psi^{I,j+1}=\psi^{I,j+1}(t,x)$ for $j\geqslant0$ is the solution to the inhomogeneous viscoelastic damped waves
    		\begin{align*}
    			&\partial_t^2\psi^{I,j+1}-\Delta\psi^{I,j+1}-b\nu\Delta\partial_t\psi^{I,j+1}\\
    			&\qquad=\gamma\Delta\left(\partial_t\psi^{I,j}-b\nu\Delta\psi^{I,j}\right)-\int_0^t\Delta^2\psi^{I,j}(s,x)\mathrm{d}s
    		\end{align*}
    		for $x\in\mb{R}^n$, $t\geqslant0$, with $\psi^{I,0}(t,x)\equiv\psi^{(0)}(t,x)$, carrying its initial data $\psi^{I,j+1}(0,x)=-\psi^{L,j+1}(0,x)$ and $\partial_t\psi^{I,j+1}(0,x)=-\partial_z\psi^{L,j}(0,x)$ such that $\partial_t^2\psi^{I,j+2}(0,x)=-\partial_z^2\psi^{L,j}(0,x)$, particularly, $\psi^{L,0}\equiv0$;
    		\item $\psi^{L,1}=\psi^{L,1}(z,x)$ is the solution to the heat equation
    		\begin{align}\label{A8}
    			\begin{cases}
    				\partial_z\psi^{L,1}-\Delta\psi^{L,1}=0,\\
    				\psi^{L,1}(0,x)=(\gamma-1)(\psi_1(x)-b\nu\Delta\psi_0(x)),
    			\end{cases}
    		\end{align}
    		for $x\in\mb{R}^n$, $z>0$;
    		\item $\psi^{L,2}=\psi^{L,2}(z,x)$ is the solution to the inhomogeneous heat equation 
    		\begin{align*}
    			\begin{cases}
    				\partial_z\psi^{L,2}-\Delta\psi^{L,2}=(\gamma-1)b\nu\Delta\partial_z\psi^{L,1},\\
    				\psi^{L,2}(0,x)=\psi_0^{L,2}(x),
    			\end{cases}
    		\end{align*}
    		for $x\in\mb{R}^n$, $t>0$, where $\psi_0^{L,2}(x)$ can be expressed in terms of $\psi_0(x)$ and $\psi_1(x)$;
    		\item $\psi^{L,j+2}=\psi^{L,j+2}(z,x)$ for $j\geqslant 1$ is the solution to the inhomogeneous heat equation 
    		\begin{align*}
    			\begin{cases}
    				\partial_z\psi^{L,j+2}-\Delta\psi^{L,j+2}=F^{L,j+1},\\
    				\psi^{L,j+2}(0,x)=\psi_0^{L,j+2}(x),
    			\end{cases}
    		\end{align*}
    		for $x\in\mb{R}^n$, $t>0$, where $\psi_0^{L,j+2}(x)$ can be expressed in terms of $\psi_0(x)$ and $\psi_1(x)$. Here, we denote the source term
    		\begin{align*}
    			F^{L,j+1}:=\Delta^{-1}\partial_z^2\left(\partial_z\psi^{L,j}-\gamma\Delta\psi^{L,j}\right)-b\nu\partial_z\left(\partial_z\psi^{L,j+1}-\gamma\Delta\psi^{L,j+1}\right).
    		\end{align*}
    	\end{itemize}
    \end{prop}
    
    \subsection{Rigorous justification of second-order profiles}\label{Subsection.Second-Order}
    Although it seems challenging to rigorously justify the expansion stated in Proposition \ref{Prop_A} for all $j\geqslant 0$, we still can construct
    \begin{align}\label{K_2}
    	\psi(t,x)=\psi^{\kappa}(t,x)+\kappa^2R^{\kappa}(t,x),
    \end{align}
    where
    \begin{align*}
    	\psi^{\kappa}(t,x):=\psi^{(0)}(t,x)+\kappa\psi^{I,1}(t,x)+\kappa\psi^{L,1}(\kappa t,x).
    \end{align*}
    In the above, the function $\psi^{I,1}$ is the solution to
    \begin{align}\label{Eq_V}
    	\begin{cases}
    		\partial_t^2\psi^{I,1}-\Delta\psi^{I,1}-b\nu\Delta\partial_t\psi^{I,1}=(\gamma-1)\Delta\left(\partial_t\psi^{(0)}-b\nu\Delta\psi^{(0)}\right),\\
    		(\psi^{I,1},\partial_t\psi^{I,1})(0,x)=((1-\gamma)(\psi_1-b\nu\Delta\psi_0),0)(x),
    	\end{cases}
    \end{align}
    for $x\in\mb{R}^n$, $t>0$,
    and $\psi^{L,1}(z,x)$ with $z=\kappa t$ is the solution to \eqref{A8}, which can be directly solved as $\psi^{L,\kappa}(t,x)=\psi^{L,1}(z,x)$ such that
    \begin{align*}
    	\psi^{L,\kappa}(t,x)=\frac{\gamma-1}{(4\pi\kappa t)^{\frac{n}{2}}}\int_{\mb{R}^n}\exp\left(-\tfrac{|x-y|^2}{4\kappa t}\right)(\psi_1(y)-b\nu\Delta\psi_0(y))\mathrm{d}y.
    \end{align*}
    It allows us to rewrite
    \begin{align}
    	\kappa^2R^{\kappa}(t,x)=\psi(t,x)-\psi^{(0)}(t,x)-\kappa\psi^{I,1}(t,x)-\kappa \psi^{L,\kappa}(t,x). \label{def R^kappa}
    \end{align}
    
    Our next purpose is to prove 
    that the term $\kappa^2 R^\kappa(t,x)$ in \eqref{K_2} is an error term (in the sense that we are going to explain later on in Remark \ref{Rem improvement rate kappa}). For this reason we need some estimates for $R^{\kappa}=R^{\kappa}(t,x)$. Direct computations show
    \begin{align*}
    	\ml{L}\psi^{(0)}&=\kappa(1-\gamma)\Delta^2\psi^{(0)},\\
    	\ml{L}\psi^{I,1}&=-(1-\gamma)\Delta^2\psi^{(0)}-\kappa(\gamma-1)\Delta^2\psi^{I,1}-\kappa\gamma(\gamma-1)\Delta^2\left(\partial_t\psi^{(0)}-b\nu\Delta\psi^{(0)}\right),\\
    	\ml{L}\psi^{L,1}&=\kappa^2(1-\gamma)(\kappa-b\nu)\Delta^3\psi^{L,1}.
    \end{align*}
    Then, $R^\kappa$ satisfies
    \begin{align}\label{Eq_Rk}
    	\begin{cases}
    		\ml{L}R^{\kappa}=-(\gamma-1)\Delta^2\left(-\gamma\partial_t\psi^{(0)}+\gamma b\nu\Delta\psi^{(0)}-\psi^{I,1}+\kappa(b\nu-\kappa)\Delta\psi^{L,1}\right),\\
    		(R^{\kappa},\partial_tR^{\kappa},\partial_t^2R^{\kappa})(0,x)=(0,(1-\gamma)\Delta\widetilde{I},\kappa(1-\gamma)\Delta^2\widetilde{I})(x),
    	\end{cases}
    \end{align}
    for $x\in\mb{R}^n$, $t>0$, where $\widetilde{I}(x):=\psi_1(x)-b\nu\Delta\psi_0(x)$. Let us take the partial Fourier transform with respect to $x$ in \eqref{Eq_Rk} to get
    \begin{align*}
    	\begin{cases}
    		\left(\partial_t^3+(b\nu+\gamma\kappa)|\xi|^2\partial_t^2+|\xi|^2\partial_t+\gamma b\nu\kappa|\xi|^4\partial_t+\kappa|\xi|^4\right)\widehat{R}^{\kappa}=\widehat{G}(t,\xi),\\
    		(\widehat{R}^{\kappa},\partial_t\widehat{R}^{\kappa},\partial_t^2\widehat{R}^{\kappa})(0,\xi)=(0,(\gamma-1)|\xi|^2\widetilde{J},\kappa(1-\gamma)|\xi|^4\widetilde{J})(\xi),
    	\end{cases}
    \end{align*}
    for $\xi\in\mb{R}^n$, $t>0$, where we denote
    \begin{align*}
    	\widehat{G}(t,\xi)&:=(\gamma-1)|\xi|^4\left(\gamma\partial_t\widehat{\psi}^{(0)}+\gamma b \nu |\xi|^2 \widehat{\psi}^{(0)}+\widehat{\psi}^{I,1}+\kappa(b\nu-\kappa)|\xi|^2\widehat{\psi}^{L,1}\right),\\
    	\widetilde{J}(\xi)&:=\widehat{\psi}_1(\xi)+b\nu|\xi|^2\widehat{\psi}_0(\xi).
    \end{align*}
    
    First of all, we notice that the differential operator in \eqref{Eq_Rk} is exactly the one for Blackstock's model \eqref{Linear_Blackstock_Eq}. Consequently, the approach in Subsection \ref{Subsec_Inhomo} still works. Clearly from the derived estimate \eqref{*}, the next differential inequality holds:
    \begin{align*}
    	\frac{\mathrm{d}}{\mathrm{d}t}\left(|\widehat{R}^{\kappa}_{tt}+b\nu|\xi|^2\widehat{R}^{\kappa}_t+k_1|\xi|^2\widehat{R}^{\kappa}|^2+k_2|\widehat{R}^{\kappa}+k_3\widehat{R}^{\kappa}_t|^2+k_5|\xi|^4|\widehat{R}^{\kappa}|^2\right)\leqslant\frac{2+\gamma k_1}{2\kappa \gamma|\xi|^2}|\widehat{G}|^2,
    \end{align*}
    where $k_2,k_3$ were chosen in \eqref{k_2,k_3}, and $k_5$ fulfills the estimate \eqref{k_5}. Let us integrate the previous inequality over $[0,t]$ to show
    \begin{align}\label{K_1}
    	|\widehat{R}^{\kappa}|^2
    	&\leqslant C|\xi|^2(1+|\xi|^2)(1+\kappa|\xi|^2)|\widetilde{J}(\xi)|^2+\tfrac{C(1+\kappa|\xi|^2)}{\kappa|\xi|^6}\int_0^t|\widehat{G}(s,\xi)|^2\mathrm{d}s\notag\\
    	&\leqslant C|\xi|^6(1+|\xi|^2)(1+\kappa|\xi|^2)|\widehat{\psi}_0|^2+C|\xi|^2(1+|\xi|^2)(1+\kappa|\xi|^2)|\widehat{\psi}_1|^2\notag\\
    	&\quad+\tfrac{C}{\kappa}|\xi|^2(1+\kappa|\xi|^2)\int_0^t|\partial_t\widehat{\psi}^{(0)}(s,\xi)|^2\mathrm{d}s+\tfrac{C}{\kappa}|\xi|^6(1+\kappa|\xi|^2)\int_0^t|\widehat{\psi}^{(0)}(s,\xi)|^2 \mathrm{d}s\notag\\
    	&\quad+\tfrac{C}{\kappa}|\xi|^2(1+\kappa|\xi|^2)\int_0^t|\widehat{\psi}^{I,1}(s,\xi)|^2\mathrm{d}s+C\kappa|\xi|^6(1+\kappa|\xi|^2)\int_0^t|\widehat{\psi}^{L,1}(\kappa s,\xi)|^2\mathrm{d}s.
    \end{align}
    
    Let us state the pointwise estimate for $\widehat{\psi}^{I,1}(t,\xi)$ to the inhomogeneous viscoelastic damped waves \eqref{Eq_V}, which is the most complex part in the estimate \eqref{K_1}.
    \begin{prop}\label{Pointwise_Psi_I}
    	The Fourier image for the solution to \eqref{Eq_V} fulfills the next pointwise estimate:
    	\begin{align*}
    		|\widehat{\psi}^{I,1}(t,\xi)|\leqslant\begin{cases}
    			C\mathrm{e}^{-c|\xi|^2t}\left(|\widehat{\psi}_0|+(1+t)^{\frac{1}{2}}|\widehat{\psi}_1|\right)&\mbox{if} \ \ 0<|\xi|\leqslant \tfrac{1}{b\nu},\\
    			C\mathrm{e}^{-ct}\left(|\xi|^2|\widehat{\psi}_0|+|\widehat{\psi}_1|\right)&\mbox{if} \ \ |\xi|\geqslant \tfrac{1}{b\nu},
    		\end{cases}
    	\end{align*}
    	with positive constant $c$ depending on the parameters $b,\nu$.
    \end{prop}
    \begin{proof}
    	In order to deal with the solution $\psi^{I,1}$ to \eqref{Eq_V}, we need an exact representation for the solution. Let us begin by recalling the representation formula for the Fourier transform with respect to the space variables of the solution to the linear wave equation with viscoelastic damping $-b\nu\Delta v_t$ and source term $f=f(t,x)$.
    	For the inhomogeneous Cauchy problem associated to
    	\begin{align}\label{CP transf v}
    		\begin{cases}
    			\widehat{v}_{tt}+b\nu |\xi|^2\widehat{v}_t+|\xi|^2\widehat{v}=\widehat{f}(t,\xi), & \xi\in\mathbb{R}^n, \, t>0 \\
    			(\widehat{v},\widehat{v}_t)(0,\xi)= (\widehat{v}_0,\widehat{v}_1)(\xi), & \xi\in\mathbb{R}^n, 
    		\end{cases} 
    	\end{align}    we find the characteristic roots 
    	\begin{align}\label{roots}
    		\rho_{\pm}(|\xi|) := \begin{cases} -\frac{b\nu}{2}|\xi|^2 \pm \frac{|\xi|}{2} \sqrt{(b\nu)^2|\xi|^2-4} & \mbox{if} \ |\xi|>\frac{2}{b\nu}, \\
    			-\frac{b\nu}{2}|\xi|^2 \pm \frac{i|\xi|}{2} \sqrt{4-(b\nu)^2|\xi|^2} & \mbox{if} \ 0< |\xi|<\frac{2}{b\nu}.\end{cases}
    	\end{align}
    	By straightforward computations we find the representation
    	\begin{align*}
    		\widehat{v}(t,\xi)= \widehat{G}_0(t,|\xi|)\widehat{v}_0(\xi)+ \widehat{G}_1(t,|\xi|)\widehat{v}_1(\xi)+\int_0^t \widehat{G}_1(t-s,|\xi|)\widehat{f}(s,\xi) \mathrm{d}s,
    	\end{align*} where
    	\begin{align}\label{def G0 hat, G1 hat}
    		\widehat{G}_0(t,|\xi|) &:= \frac{\rho_{+}(|\xi|)\mathrm{e}^{\rho_{-}(|\xi|) t}-\rho_{-}(|\xi|)\mathrm{e}^{\rho_{+}(|\xi|) t}}{\rho_{+}(|\xi|) -\rho_{-}(|\xi|)},\notag\\   \widehat{G}_1(t,|\xi|) &:= \frac{\mathrm{e}^{\rho_{+}(|\xi|) t}-\mathrm{e}^{\rho_{-}(|\xi|) t}}{\rho_{+}(|\xi|) -\rho_{-}(|\xi|)}.
    	\end{align} In particular, we may express the Fourier transform of the kernel functions $G_0,G_1$ as follows:
    	\begin{align*}
    		\widehat{G}_0(t,|\xi|) & =\begin{cases} \left(\cosh \left(\tfrac{|\xi|}{2}\sqrt{(b\nu)^2|\xi|^2-4} t\right)+b\nu |\xi| \, \tfrac{\sinh\left(\tfrac{|\xi|}{2}\sqrt{(b\nu)^2|\xi|^2-4} t\right)}{\sqrt{(b\nu)^2|\xi|^2-4} }\right) \mathrm{e}^{-\frac{b\nu}{2}|\xi|^2 t}\\
    			\qquad\qquad\qquad\qquad\qquad \qquad\qquad\qquad\qquad\qquad\qquad\qquad\ \ \mbox{if} \ \ |\xi|>\frac{2}{b\nu}, \\
    			\left(\cos \left(\tfrac{|\xi|}{2}\sqrt{4-(b\nu)^2|\xi|^2} t\right)+b\nu |\xi| \, \tfrac{\sin\left(\tfrac{|\xi|}{2}\sqrt{4-(b\nu)^2|\xi|^2} t\right)}{\sqrt{4-(b\nu)^2|\xi|^2} }\right) \mathrm{e}^{-\frac{b\nu}{2}|\xi|^2 t}\\
    			\qquad\qquad\qquad\qquad\qquad \qquad\qquad\qquad\qquad\qquad\qquad\quad \mbox{if} \ \ 0<|\xi|<\frac{2}{b\nu},\end{cases} \\
    		\widehat{G}_1(t,|\xi|) & =\begin{cases} \vphantom{\Bigg(}\tfrac{2\sinh\left(\tfrac{|\xi|}{2}\sqrt{(b\nu)^2|\xi|^2-4} t\right)}{|\xi|\sqrt{(b\nu)^2|\xi|^2-4} } \, \mathrm{e}^{-\frac{b\nu}{2}|\xi|^2 t}  & \mbox{if} \ \ |\xi|>\frac{2}{b\nu}, \\
    			\vphantom{\Bigg(} \tfrac{2\sin\left(\tfrac{|\xi|}{2}\sqrt{4-(b\nu)^2|\xi|^2} t\right)}{|\xi|\sqrt{4-(b\nu)^2|\xi|^2} } \ \mathrm{e}^{-\frac{b\nu}{2}|\xi|^2 t}  & \mbox{if} \ \ 0<|\xi|<\frac{2}{b\nu}.\end{cases}
    	\end{align*}
    	Having in mind the previous representation formula for the solution to \eqref{CP transf v}, we obtain
    	\begin{align*}
    		\widehat{\psi}^{I,1}(t,\xi) & = (1-\gamma)\widehat{G}_0(t,|\xi|)\left(\widehat{\psi}_1(\xi)+b\nu|\xi|^2\widehat{\psi}_0(\xi)\right)\\ & \quad +(1-\gamma)\int_0^t |\xi|^2\widehat{G}_1(t-s,|\xi|)\left(\partial_t\widehat{\psi}^{(0)}(s,\xi)+b\nu|\xi|^2\widehat{\psi}^{(0)}(s,\xi)\right)\mathrm{d}s
    	\end{align*} and
    	\begin{align*}
    		\widehat{\psi}^{(0)}(s,\xi) &= \widehat{G}_0(s,|\xi|)\widehat{\psi}_0(\xi)+ \widehat{G}_1(s,|\xi|)\widehat{\psi}_1(\xi), \\
    		\partial_t \widehat{\psi}^{(0)}(s,\xi) &= \partial_t\widehat{G}_0(s,|\xi|)\widehat{\psi}_0(\xi)+ \partial_t \widehat{G}_1(s,|\xi|)\widehat{\psi}_1(\xi).
    	\end{align*}
    	Hence, combining the previous identities, we arrive at
    	\begin{align}
    		&\widehat{\psi}^{I,1}(t,\xi)=\widehat{\psi}^{I,1}_{\mathrm{hom}}(t,\xi) +\widehat{\psi}^{I,1}_{\mathrm{Duh}}(t,\xi),\label{fourier transf psi I,1}
    	\end{align} 	
    	where
    	\begin{align*}
    		\widehat{\psi}^{I,1}_{\mathrm{hom}}&:=(1-\gamma)\widehat{G}_0(t,|\xi|)\left(\widehat{\psi}_1(\xi)+b\nu|\xi|^2\widehat{\psi}_0(\xi)\right),\\
    		\widehat{\psi}^{I,1}_{\mathrm{Duh}}&:=(1-\gamma)\int_0^t |\xi|^2\widehat{G}_1(t-s,|\xi|) \sum_{j=0,1}\left(\partial_t\widehat{G}_j(s,|\xi|)+b\nu|\xi|^2\widehat{G}_j(s,|\xi|)\right)\widehat{\psi}_j(\xi) \mathrm{d}s.
    	\end{align*}
    	From the explicit representation for $\widehat{G}_0$, we find immediately the following estimates:
    	\begin{equation}\label{est psi I,1 hom}
    		\begin{split}
    			& |\widehat{\psi}^{I,1}_{\mathrm{hom}}(t,\xi)|\lesssim \mathrm{e}^{-c t} \big( |\xi|^2|\widehat{\psi}_0(\xi)|+|\widehat{\psi}_1(\xi)|\big) \qquad \,  \mbox{if} \ \ |\xi|>\tfrac{1}{b\nu}, \\
    			& |\widehat{\psi}^{I,1}_{\mathrm{hom}}(t,\xi)|\lesssim \mathrm{e}^{-\frac{b\nu }{2}|\xi|^2 t} \big(|\widehat{\psi}_0(\xi)|+|\widehat{\psi}_1(\xi)|\big) \qquad \mbox{if} \ \ 0<|\xi|\leqslant\tfrac{1}{b\nu}. 
    		\end{split}
    	\end{equation}
    	Next, we estimate the term $\widehat{\psi}^{I,1}_{\mathrm{Duh}}$.
    	Consequently, for this purpose, it is necessary to derive some representations of the time derivative of $\widehat{G}_0(t,|\xi|),\widehat{G}_1(t,|\xi|)$. From \eqref{def G0 hat, G1 hat}, we have
    	\begin{align*}
    		\partial_t \widehat{G}_0(t,|\xi|)&=\rho_+(|\xi|)\rho_-(|\xi|)  \frac{\mathrm{e}^{\rho_{-}(|\xi|) t}-\mathrm{e}^{\rho_{+}(|\xi|) t}}{\rho_{+}(|\xi|) -\rho_{-}(|\xi|)} \\
    		&=- \rho_+(|\xi|)\rho_-(|\xi|) \widehat{G}_1(t,|\xi|)=-|\xi|^2 \widehat{G}_1(t,|\xi|), \\
    		\partial_t \widehat{G}_1(t,|\xi|)&= \frac{\rho_{+}(|\xi|)\mathrm{e}^{\rho_{+}(|\xi|) t}-\rho_{-}(|\xi|)\mathrm{e}^{\rho_{-}(|\xi|) t}}{\rho_{+}(|\xi|) -\rho_{-}(|\xi|)}.
    	\end{align*} Moreover, by \eqref{roots} we get
    	\begin{align*}
    		\partial_t\widehat{G}_1(t,|\xi|) & =\begin{cases} \left(\cosh \left(\tfrac{|\xi|}{2}\sqrt{(b\nu)^2|\xi|^2-4} t\right)-b\nu |\xi| \, \tfrac{\sinh\left(\tfrac{|\xi|}{2}\sqrt{(b\nu)^2|\xi|^2-4} t\right)}{\sqrt{(b\nu)^2|\xi|^2-4} }\right) \mathrm{e}^{-\frac{b\nu}{2}|\xi|^2 t}\\
    			\qquad\qquad\qquad\qquad\qquad\qquad\qquad\qquad\qquad\qquad\qquad\qquad\mbox{if} \ \ |\xi|>\frac{2}{b\nu}, \\
    			\left(\cos \left(\tfrac{|\xi|}{2}\sqrt{4-(b\nu)^2|\xi|^2} t\right)-b\nu |\xi| \, \tfrac{\sin\left(\tfrac{|\xi|}{2}\sqrt{4-(b\nu)^2|\xi|^2} t\right)}{\sqrt{4-(b\nu)^2|\xi|^2} }\right) \mathrm{e}^{-\frac{b\nu}{2}|\xi|^2 t}\\
    			\qquad\qquad\qquad\qquad\qquad\qquad\qquad\qquad\qquad\qquad\qquad \mbox{if} \ \ 0<|\xi|<\frac{2}{b\nu}.\end{cases}
    	\end{align*} Hence, we try to express in a simpler form the terms
    \begin{align*}
    	\widehat{G}_1(t-s,|\xi|) \big(\partial_t\widehat{G}_j(s,|\xi|)+b\nu|\xi|^2\widehat{G}_j(s,|\xi|)\big) \ \ \mbox{for}\ \ j=0,1.
    \end{align*} By direct computations, we find
    	\begin{align}
    		&\widehat{G}_1(t-s,|\xi|) \left(\partial_t\widehat{G}_1(s,|\xi|)+b\nu|\xi|^2\widehat{G}_1(s,|\xi|)\right) \notag \\ 
    		& \qquad  = \frac{\mathrm{e}^{\rho_{+}(|\xi|) (t-s)}-\mathrm{e}^{\rho_{-}(|\xi|) (t-s)}}{(\rho_{+}(|\xi|) -\rho_{-}(|\xi|))^2}\notag\\ &\qquad
    		\quad\ \times\left((\rho_{+}(|\xi|)+b\nu|\xi|^2)\mathrm{e}^{\rho_{+}(|\xi|) s}-(\rho_{-}(|\xi|)+b\nu|\xi|^2)\mathrm{e}^{\rho_{-}(|\xi|) s} \right) \notag \\
    		& \qquad  = \frac{\mathrm{e}^{\rho_{+}(|\xi|) (t-s)}-\mathrm{e}^{\rho_{-}(|\xi|) (t-s)}}{(\rho_{+}(|\xi|) -\rho_{-}(|\xi|))^2} \left(-\rho_{-}(|\xi|)\mathrm{e}^{\rho_{+}(|\xi|) s}+\rho_{+}(|\xi|)\mathrm{e}^{\rho_{-}(|\xi|) s} \right)\notag\\
    		&\qquad= \widehat{G}_1(t-s,|\xi|) \widehat{G}_0(s,|\xi|), \label{G1(t-s) (dt+b nu xi^2 )G1(s)}
    	\end{align} where we used the relation $\rho_\pm(|\xi|)+b\nu|\xi|^2=-\rho_\mp(|\xi|)$ in the second last equality. Analogously,
    	\begin{align*}
    		& \widehat{G}_1(t-s,|\xi|) \left(\partial_t\widehat{G}_0(s,|\xi|)+b\nu|\xi|^2\widehat{G}_0(s,|\xi|)\right) \\ 
    		&  \qquad = \frac{\mathrm{e}^{\rho_{+}(|\xi|) (t-s)}-\mathrm{e}^{\rho_{-}(|\xi|) (t-s)}}{(\rho_{+}(|\xi|) -\rho_{-}(|\xi|))^2}\\
    		&\qquad\quad\times \left(-\rho_{-}(|\xi|)(\rho_{+}(|\xi|)+b\nu|\xi|^2)\mathrm{e}^{\rho_{+}(|\xi|) s}+\rho_{+}(|\xi|)(\rho_{-}(|\xi|)+b\nu|\xi|^2)\mathrm{e}^{\rho_{-}(|\xi|) s} \right) \\
    		&   \qquad= \frac{\mathrm{e}^{\rho_{+}(|\xi|) (t-s)}-\mathrm{e}^{\rho_{-}(|\xi|) (t-s)}}{(\rho_{+}(|\xi|) -\rho_{-}(|\xi|))^2} \left(\rho_{-}(|\xi|)^2\mathrm{e}^{\rho_{+}(|\xi|) s}-\rho_{+}(|\xi|)^2\mathrm{e}^{\rho_{-}(|\xi|) s} \right).
    	\end{align*} Since $\rho_\pm(|\xi|)^2= \tfrac{(b\nu)^2}{2}|\xi|^4-|\xi|^2\mp\tfrac{b\nu}{2}|\xi|^2 (\rho_+(|\xi|)-\rho_-(|\xi|))$, then
    	\begin{align*}
    		&\frac{\rho_{-}(|\xi|)^2\mathrm{e}^{\rho_{+}(|\xi|) s}-\rho_{+}(|\xi|)^2\mathrm{e}^{\rho_{-}(|\xi|) s}}{\rho_{+}(|\xi|) -\rho_{-}(|\xi|)}\\
    		&\qquad = \frac{b\nu}{2}|\xi|^2 \left(\mathrm{e}^{\rho_{+}(|\xi|) s}+\mathrm{e}^{\rho_{-}(|\xi|) s}\right)+\left(\frac{(b\nu)^2}{2}|\xi|^2-1\right)|\xi|^2 \widehat{G}_1(s,|\xi|).
    	\end{align*} So, we have
    	\begin{align}
    		& \widehat{G}_1(t-s,|\xi|) \left(\partial_t\widehat{G}_0(s,|\xi|)+b\nu|\xi|^2\widehat{G}_0(s,|\xi|)\right) \notag \\ 
    		& \ \   =  \widehat{G}_1(t-s,|\xi|) \left(\frac{b\nu}{2}|\xi|^2 \left(\mathrm{e}^{\rho_{+}(|\xi|) s}+\mathrm{e}^{\rho_{-}(|\xi|) s}\right)+\left(\frac{(b\nu)^2}{2}|\xi|^2-1\right)|\xi|^2 \widehat{G}_1(s,|\xi|)\right). \label{G1(t-s) (dt+b nu xi^2 )G0(s)}
    	\end{align} Later, it will be useful also the following representation:
    	\begin{align*}
    		&\frac{b\nu}{2}|\xi|^2 \left(\mathrm{e}^{\rho_{+}(|\xi|) s}+\mathrm{e}^{\rho_{-}(|\xi|) s}\right)\\
    		&\qquad = \begin{cases} b\nu |\xi|^2 \cosh \left(\tfrac{|\xi|}{2}\sqrt{(b\nu)^2|\xi|^2-4} s\right) \mathrm{e}^{-\frac{b\nu}{2}|\xi|^2 s}  & \mbox{if} \ \ |\xi|>\frac{2}{b\nu}, \\ b\nu |\xi|^2 \cos \left(\tfrac{|\xi|}{2}\sqrt{4-(b\nu)^2|\xi|^2} s\right) \mathrm{e}^{-\frac{b\nu}{2}|\xi|^2 s}  & \mbox{if} \ \ 0<|\xi|<\frac{2}{b\nu}. \end{cases}
    	\end{align*}
    	Since our goal is now to provide pointwise estimates for $\widehat{\psi}^{I,1}_{\mathrm{Duh}}$ in the right-hand side of \eqref{fourier transf psi I,1}, we are going to write explicitly the terms in \eqref{G1(t-s) (dt+b nu xi^2 )G1(s)} and \eqref{G1(t-s) (dt+b nu xi^2 )G0(s)}. For $|\xi|>\frac{2}{b\nu}$, we get
    	\begin{align*}
    		& \widehat{G}_1(t-s,|\xi|) \left(\partial_t\widehat{G}_1(s,|\xi|)+b\nu|\xi|^2\widehat{G}_1(s,|\xi|)\right) \\
    		& \qquad = \tfrac{2\sinh\left( \tfrac{|\xi|}{2}\sqrt{(b\nu)^2|\xi|^2-4}(t-s)\right)}{|\xi|\sqrt{(b\nu)^2|\xi|^2-4}} \, \cosh\left( \tfrac{|\xi|}{2}\sqrt{(b\nu)^2|\xi|^2-4}s\right) \mathrm{e}^{-\frac{b\nu}{2}|\xi|^2 t}\\
    		& \quad\qquad +  \tfrac{2b\nu \sinh\left( \tfrac{|\xi|}{2}\sqrt{(b\nu)^2|\xi|^2-4}(t-s)\right)}{\sqrt{(b\nu)^2|\xi|^2-4}} \,  \tfrac{\sinh\left( \tfrac{|\xi|}{2}\sqrt{(b\nu)^2|\xi|^2-4}s\right)}{\sqrt{(b\nu)^2|\xi|^2-4}}\, \mathrm{e}^{-\frac{b\nu}{2}|\xi|^2 t}, 
    	\end{align*} 
    	and
    	\begin{align*}
    		& \widehat{G}_1(t-s,|\xi|) \left(\partial_t\widehat{G}_0(s,|\xi|)+b\nu|\xi|^2\widehat{G}_0(s,|\xi|)\right) \\
    		& \qquad = 2b\nu |\xi|\tfrac{\sinh\left( \tfrac{|\xi|}{2}\sqrt{(b\nu)^2|\xi|^2-4}(t-s)\right)}{\sqrt{(b\nu)^2|\xi|^2-4}} \, \cosh\left( \tfrac{|\xi|}{2}\sqrt{(b\nu)^2|\xi|^2-4}s\right) \mathrm{e}^{-\frac{b\nu}{2}|\xi|^2 t}\\
    		& \qquad\quad + 2\big((b\nu)^2|\xi|^2-2\big) \tfrac{ \sinh\left( \tfrac{|\xi|}{2}\sqrt{(b\nu)^2|\xi|^2-4}(t-s)\right)}{\sqrt{(b\nu)^2|\xi|^2-4}} \,  \tfrac{\sinh\left( \tfrac{|\xi|}{2}\sqrt{(b\nu)^2|\xi|^2-4}s\right)}{\sqrt{(b\nu)^2|\xi|^2-4}}\, \mathrm{e}^{-\frac{b\nu}{2}|\xi|^2 t}.
    	\end{align*}
    	On the other hand, $0<|\xi|<\frac{2}{b\nu}$, we have
    	\begin{align*}
    		& \widehat{G}_1(t-s,|\xi|) \left(\partial_t\widehat{G}_1(s,|\xi|)+b\nu|\xi|^2\widehat{G}_1(s,|\xi|)\right) \\
    		& \qquad = \tfrac{2\sin\left( \tfrac{|\xi|}{2}\sqrt{4-(b\nu)^2|\xi|^2}(t-s)\right)}{|\xi|\sqrt{4-(b\nu)^2|\xi|^2}} \, \cos\left( \tfrac{|\xi|}{2}\sqrt{4-(b\nu)^2|\xi|^2}s\right) \mathrm{e}^{-\frac{b\nu}{2}|\xi|^2 t}\\
    		& \quad\qquad +  \tfrac{2b\nu \sin\left( \tfrac{|\xi|}{2}\sqrt{4-(b\nu)^2|\xi|^2}(t-s)\right)}{\sqrt{4-(b\nu)^2|\xi|^2}} \,  \tfrac{\sin\left( \tfrac{|\xi|}{2}\sqrt{4-(b\nu)^2|\xi|^2}s\right)}{\sqrt{4-(b\nu)^2|\xi|^2}}\, \mathrm{e}^{-\frac{b\nu}{2}|\xi|^2 t}, 
    	\end{align*} 
    	and
    	\begin{align*}
    		& \widehat{G}_1(t-s,|\xi|) \left(\partial_t\widehat{G}_0(s,|\xi|)+b\nu|\xi|^2\widehat{G}_0(s,|\xi|)\right) \\
    		& \qquad = 2b\nu |\xi|\tfrac{\sin\left( \tfrac{|\xi|}{2}\sqrt{4-(b\nu)^2|\xi|^2}(t-s)\right)}{\sqrt{4-(b\nu)^2|\xi|^2}} \, \cos\left( \tfrac{|\xi|}{2}\sqrt{4-(b\nu)^2|\xi|^2}s\right) \mathrm{e}^{-\frac{b\nu}{2}|\xi|^2 t}\\
    		& \quad\qquad + 2\big((b\nu)^2|\xi|^2-2\big) \tfrac{ \sin\left( \tfrac{|\xi|}{2}\sqrt{4-(b\nu)^2|\xi|^2}(t-s)\right)}{\sqrt{4-(b\nu)^2|\xi|^2}} \,  \tfrac{\sin\left( \tfrac{|\xi|}{2}\sqrt{4-(b\nu)^2|\xi|^2}s\right)}{\sqrt{4-(b\nu)^2|\xi|^2}}\, \mathrm{e}^{-\frac{b\nu}{2}|\xi|^2 t}.
    	\end{align*} Since for large frequencies we may estimate
    	\begin{align*}
    		\sinh\left( \tfrac{|\xi|}{2}\sqrt{(b\nu)^2|\xi|^2-4}(t-s)\right) &\leqslant \cosh\left( \tfrac{|\xi|}{2}\sqrt{(b\nu)^2|\xi|^2-4}(t-s)\right)\\
    		&\leqslant \exp\left(\tfrac{|\xi|}{2}\sqrt{(b\nu)^2|\xi|^2-4}\,(t-s)\right), 
    	\end{align*}
    	\begin{align*}
    		\sinh\left( \tfrac{|\xi|}{2}\sqrt{(b\nu)^2|\xi|^2-4}s\right) &\leqslant \cosh\left( \tfrac{|\xi|}{2}\sqrt{(b\nu)^2|\xi|^2-4}s\right)\\
    		&\leqslant \exp\left(\tfrac{|\xi|}{2}\sqrt{(b\nu)^2|\xi|^2-4}\, s\right),
    	\end{align*} and 
    	\begin{align*}
    		\exp\left(\left(-\tfrac{b\nu}{2}|\xi|^2+\tfrac{|\xi|}{2}\sqrt{(b\nu)^2|\xi|^2-4}\right)t\right)\leqslant \exp\left(-\tfrac{t}{b\nu}\right),
    	\end{align*} 
    	we find the pointwise estimates
    	\begin{equation}\label{pointwise estimates}
    		\begin{split}
    			\left|\widehat{G}_1(t-s,|\xi|) \left(\partial_t\widehat{G}_j(s,|\xi|)+b\nu|\xi|^2\widehat{G}_j(s,|\xi|)\right)\right| & \lesssim |\xi|^{-2j} \mathrm{e}^{-ct} \qquad \quad \ \mbox{if} \ |\xi|>\tfrac{1}{b\nu}, \\
    			\left|\widehat{G}_1(t-s,|\xi|) \left(\partial_t\widehat{G}_j(s,|\xi|)+b\nu|\xi|^2\widehat{G}_j(s,|\xi|)\right)\right| & \lesssim |\xi|^{-j} \mathrm{e}^{-\frac{b\nu}{2} |\xi|^2 t} \qquad \mbox{if} \ 0<|\xi|\leqslant \tfrac{1}{b\nu},
    		\end{split}
    	\end{equation} for $j=0,1$ and for a suitable constant $c=c(b,\nu)>0$.
    	By using \eqref{pointwise estimates}, we may now derive the pointwise estimates for $\widehat{\psi}^{I,1}_{\mathrm{Duh}}$. For $|\xi|>\frac{1}{b\nu}$, we have
    	\begin{align}
    		|\widehat{\psi}^{I,1}_{\mathrm{Duh}}(t,\xi)| &\lesssim \int_0^t |\xi|^2\sum_{j=0,1} \left| \widehat{G}_1(t-s,|\xi|) \left(\partial_t\widehat{G}_j(s,|\xi|)+b\nu|\xi|^2\widehat{G}_j(s,|\xi|)\right)\right| |\widehat{\psi}_j(\xi)|\, \mathrm{d}s \notag \\
    		&\lesssim  \mathrm{e}^{-c t} \int_0^t \left( |\xi|^2 |\widehat{\psi}_0(\xi)| +  |\widehat{\psi}_1(\xi)| \right)\, \mathrm{d}s \lesssim \mathrm{e}^{-c t} \left( |\xi|^2 |\widehat{\psi}_0(\xi)| +  |\widehat{\psi}_1(\xi)| \right), \label{est psi I,1 Duh large xi}
    	\end{align}
    	while for $0<|\xi|\leqslant\frac{1}{b\nu}$, we get
    	\begin{align}
    		|\widehat{\psi}^{I,1}_{\mathrm{Duh}}(t,\xi)| 
    		&\lesssim  \mathrm{e}^{-\frac{b\nu}{2} |\xi|^2 t}  \int_0^t \left( |\xi|^2 |\widehat{\psi}_0(\xi)| + |\xi| |\widehat{\psi}_1(\xi)| \right)\, \mathrm{d}s\notag\\
    		& \lesssim \mathrm{e}^{-\frac{b\nu}{2} |\xi|^2 t}  \left(  |\xi|^2 t |\widehat{\psi}_0(\xi)| +  |\xi| t  |\widehat{\psi}_1(\xi)| \right) \notag \\
    		& \lesssim \mathrm{e}^{-c|\xi|^2 t}  \left( |\widehat{\psi}_0(\xi)| + t^{\frac{1}{2}}  |\widehat{\psi}_1(\xi)| \right). \label{est psi I,1 Duh small xi}
    	\end{align}
    	Combining \eqref{est psi I,1 hom}, \eqref{est psi I,1 Duh large xi} and \eqref{est psi I,1 Duh small xi}, we obtain the desired pointwise estimate for $\widehat{\psi}^{I,1}$.
    \end{proof}
    
    Then, we can announce our main theorem in this part.
    \begin{theorem}\label{Thm_Higher}
    	Assume $n\geqslant 3$. Let us assume $\psi_0\in H^5\cap L^1$ and $\psi_1\in H^3\cap L^1$ such that the consistency condition \eqref{Assumption_Singular_Energy} holds.  Then, the following refined estimate holds:
    	\begin{align*}
    		&\sup\limits_{t\in[0,\infty)}\left\|\psi(t,\cdot)-\left(\psi^{(0)}(t,\cdot)+\kappa\psi^{I,1}(t,\cdot)+\kappa \psi^{L,\kappa}(t,\cdot)\right)\right\|_{L^2}\\
    		&\qquad\leqslant C\kappa\sqrt{\kappa} \left(\|\psi_0\|_{H^5\cap L^1}+\|\psi_1\|_{H^3\cap L^1}\right),
    	\end{align*}
    	where $C$ is a positive constant independent of $\kappa$.
    \end{theorem}
    \begin{remark} \label{Rem improvement rate kappa}
    	Let us compare the obtained result Theorem \ref{Thm_Potential_2} with Theorem \ref{Thm_Higher} for $n\geqslant 3$. By subtracting the additional error terms $\kappa\psi^{I,1}(t,\cdot)+\kappa \psi^{L,\kappa}(t,\cdot)$ in the $L^2$ norm, the rate of convergence has been improved by a factor $\kappa$.
    \end{remark}
    \begin{remark}
    	From Theorem \ref{Thm_Higher}, we claim the validity of the formal expansion stated in Proposition \ref{Prop_A} at least for $j=0$, which means the second-order asymptotic expansion. Precisely, concerning the Blackstock's model without Becker's assumption \eqref{Linear_Blackstock_Eq}, the solution has in the asymptotic expansion as first-order term the solution to a viscoelastic damped wave equation, and as second-order terms the sum of a solution to an inhomogeneous viscoelastic damped wave equation and a solution to the heat equation, respectively.
    \end{remark}
    \begin{proof}
    	Recalling the definition of the error term $\kappa^2 R^{\kappa}(t,x)$, we just need to derive some estimates for $R^{\kappa}(t,\cdot)$ in the $L^2$ norm to complete the proof. First of all, from the regular assumptions for initial data, we claim
    	\begin{align*}
    		&\int_{\mb{R}^n}|\xi|^6(1+|\xi|^2)(1+\kappa|\xi|^2)|\widehat{\psi}_0|^2\mathrm{d}\xi+\int_{\mb{R}^n}|\xi|^2(1+|\xi|^2)(1+\kappa|\xi|^2)|\widehat{\psi}_1|^2\mathrm{d}\xi\\
    		&\qquad\leqslant C\left(\|\psi_0\|_{H^5}^2+\|\psi_1\|_{H^3}^2\right).
    	\end{align*}
    	Namely, from the pointwise estimate \eqref{K_1}, we just need to estimate the remaining four terms. Before doing this, let us state some estimates for the Fourier images of the functions appearing in the integrals on the right-hand side of \eqref{K_1}. From Lemma 2.4 in \cite{IkehataNatsume}, we can state the pointwise estimates
    	\begin{align} \label{pointwise est psi^(0)}
    		|\partial_t\widehat{\psi}^{(0)}(s,\xi)|^2+|\xi|^2|\widehat{\psi}^{(0)}(s,\xi)|^2\leqslant C\exp\left(-\tfrac{c|\xi|^2}{1+|\xi|^2}s\right)\left(|\xi|^2|\widehat{\psi}_0|^2+|\widehat{\psi}_1|^2\right),
    	\end{align}
    	and from Proposition \ref{Pointwise_Psi_I}, we see
    	\begin{align} \label{pointwise est psi^I,1}
    		|\widehat{\psi}^{I,1}(s,\xi)|^2\leqslant\begin{cases}
    			C\mathrm{e}^{-c|\xi|^2s}\left(|\widehat{\psi}_0|^2+(1+s)|\widehat{\psi}_1|^2\right)&\mbox{if} \ \ 0<|\xi|\leqslant \tfrac{1}{b\nu},\\
    			C\mathrm{e}^{-cs}\left(\langle\xi\rangle^4|\widehat{\psi}_0|^2+|\widehat{\psi}_1|^2\right)&\mbox{if} \ \ |\xi|\geqslant \tfrac{1}{b\nu}.
    		\end{cases}
    	\end{align}
    	Finally, the Fourier transform for \eqref{A8} implies
    	\begin{align*}
    		|\widehat{\psi}^{L,1}(\kappa s,\xi)|^2\leqslant C\mathrm{e}^{-c\kappa|\xi|^2s}\left(|\xi|^4|\widehat{\psi}_0|^2+|\widehat{\psi}_1|^2\right).
    	\end{align*}
    	
    	To deal with those terms, we apply WKB analysis. Precisely, by employing the Hausdorff-Young inequality, one has
    	\begin{align*}
    		&\int_{\mb{R}^n}|\xi|^2(1+\kappa|\xi|^2)\int_0^t|\partial_t\widehat{\psi}^{(0)}(s,\xi)|^2\mathrm{d}s\mathrm{d}\xi\\
    		&\quad \ \leqslant C\int_{|\xi|\leqslant 1/(b\nu)}\int_0^t\mathrm{e}^{-c|\xi|^2s}\left(|\xi|^4|\widehat{\psi}_0|^2+|\xi|^2|\widehat{\psi}_1|^2\right)\mathrm{d}s\mathrm{d}\xi\\
    		&\quad \ \quad+C\int_{|\xi|\geqslant 1/(b\nu)}\int_0^t\mathrm{e}^{-cs}\left(\langle\xi\rangle^6|\widehat{\psi}_0|^2+\langle\xi\rangle^4|\widehat{\psi}_1|^2\right)\mathrm{d}s\mathrm{d}\xi\\
    		&\quad \ \leqslant C\int_0^t\left(\|\mathrm{e}^{-c|\xi|^2s}|\xi|^2\|_{L^2(|\xi|\leqslant 1/(b\nu))}^2\|\psi_0\|_{L^1}^2+\|\mathrm{e}^{-c|\xi|^2s}|\xi|\,\|_{L^2(|\xi|\leqslant 1/(b\nu))}^2\|\psi_1\|_{L^1}^2\right)\mathrm{d}s\\
    		&\quad\quad \ +C\int_0^t\mathrm{e}^{-cs}\mathrm{d}s\left(\|\psi_0\|_{H^3}^2+\|\psi_1\|_{H^2}^2\right).
    	\end{align*}
    	With the aid of Lemma \ref{Lemma_Basic_1}, we conclude
    	\begin{align*}
    		\int_{\mb{R}^n}|\xi|^2(1+\kappa|\xi|^2)\int_0^t|\partial_t\widehat{\psi}^{(0)}(s,\xi)|^2\mathrm{d}s\mathrm{d}\xi\leqslant C\left(\|\psi_0\|_{H^3\cap L^1}^2+\|\psi_1\|_{H^2\cap L^1}^2\right)
    	\end{align*}
    	for any $n\geqslant 1$. By repeating the same steps as the above, 
    	we get
    	\begin{align*}
    		\int_{\mb{R}^n}|\xi|^6(1+\kappa|\xi|^2)\int_0^t|\widehat{\psi}^{(0)}(s,\xi)|^2\mathrm{d}s\mathrm{d}\xi\leqslant C\left(\|\psi_0\|_{H^4\cap L^1}^2+\|\psi_1\|_{H^3\cap L^1}^2\right)
    	\end{align*}
    	for any $n\geqslant 1$. Similarly, due to $\int_0^t(1+s)^{-\frac{n}{2}}\mathrm{d}s\leqslant C$ for any $t>0$ and $n\geqslant 3$ 
    	it holds
    	\begin{align*}
    		\int_{\mb{R}^n}|\xi|^2(1+\kappa|\xi|^2)\int_0^t|\widehat{\psi}^{I,1}(s,\xi)|^2\mathrm{d}s\mathrm{d}\xi\leqslant C\left(\|\psi_0\|_{H^4\cap L^1}^2+\|\psi_1\|_{H^2\cap L^1}^2\right)	
    	\end{align*}
    	for any $n\geqslant 3$. Finally, let us treat $\widehat{\psi}^{L,1}$ in the next way:
    	\begin{align*}
    		&\kappa\int_{\mb{R}^n}|\xi|^6(1+\kappa|\xi|^2)\int_0^t|\widehat{\psi}^{L,1}(\kappa s,\xi)|^2\mathrm{d}s\mathrm{d}\xi\\
    		&\qquad\leqslant \kappa C\int_{\mb{R}^n}|\xi|^6(1+\kappa|\xi|^2)\left(|\xi|^4|\widehat{\psi}_0|^2+|\widehat{\psi}_1|^2\right)\int_0^t\mathrm{e}^{-c\kappa|\xi|^2s}\mathrm{d}s\mathrm{d}\xi\\
    		&\qquad\leqslant C\left(\|\psi_0\|_{H^5}^2+\|\psi_1\|_{H^3}^2\right)
    	\end{align*}
    	for any $n\geqslant 1$. Summarizing the last estimates and \eqref{K_1}, we arrive at
    	\begin{align*}
    		\|R^\kappa(t,\cdot)\|_{L^2}^2\leqslant \tfrac{C}{\kappa}\left(\|\psi_0\|^2_{H^5\cap L^1}+\|\psi_1\|^2_{H^3\cap L^1}\right).
    	\end{align*}
    	Thus, our proof is complete.
    \end{proof}
    
    \begin{coro}\label{Coro_Higher}
    	Let $n\geqslant 2$. We assume $\langle D\rangle^{s_0+2} \psi_0 \in L^1$ and $\langle D\rangle^{s_0} \psi_1 \in L^1$ with $s_0>n+3$ such that  the consistency condition \eqref{Assumption_Singular_Energy} holds.  Then, the following refined estimate holds:
    	\begin{align} 
    		& \sup_{t\in [0,\infty)} 
    		\left\|\psi(t,\cdot)-\left(\psi^{(0)}(t,\cdot)+\kappa\psi^{I,1}(t,\cdot)+\kappa \psi^{L,1}(\kappa t,\cdot)\right)\right\|_{L^\infty} \notag \\ & \qquad \leqslant C\kappa\sqrt{\kappa} \left(\|\langle D\rangle^{s_0+2}\psi_0\|_{L^1}+\|\langle D\rangle^{s_0}\psi_1\|_{L^1}\right), \label{est err ord 2}
    	\end{align}
    	where $C$ is a positive constant independent of $\kappa$.
    \end{coro}
    \begin{proof}
    	In order to prove \eqref{est err ord 2}, we have to control the norm $\|\widehat{R}^\kappa(t,\cdot)\|_{L^1}$ for all $t\geqslant 0$, where $\widehat{R}^\kappa$ denotes the Fourier transform of $R^\kappa$, whose definition is given in  \eqref{def R^kappa}. 
    	By using pointwise estimates \eqref{pointwise est psi^(0)} and \eqref{pointwise est psi^I,1}, and following the similar procedure to the proof of Corollary \ref{Coro_energy}, 
    	we may conclude the validity of \eqref{est err ord 2}.
    \end{proof}

    \begin{remark}
    	In Theorem \ref{Thm_Higher} and Corollary \ref{Coro_Higher}, we rigorously justified the second-order profiles to Blackstock's model \eqref{Linear_Blackstock_Eq} in the spaces $L^{\infty}([0,\infty),L^2)$ and $L^{\infty}([0,\infty),L^{\infty})$, namely, the correction for $j=0$ in Proposition \ref{Prop_A}. We conjecture that the higher-order profiles, i.e. $j\geqslant 1$ in Proposition \ref{Prop_A}, also hold, which may be proved by following the approach in Subsection \ref{Subsection.Second-Order}.
    \end{remark}

    \section{Global well-posedness of the nonlinear Blackstock's model}\label{Sec_GESDS}
    In the beginning of this section, let us firstly state the main result for the global (in time) existence and decay estimates of solutions to the nonlinear Blackstock's model \eqref{Blackstock_Eq-New-Nonlinear}.
    \begin{theorem}\label{Thm_GESDS}
    	Let us consider $s\geqslant n/2-2$ for $n\geqslant 5$. Then, there exists a constant $\epsilon>0$ such that for all $(\psi_0,\psi_1,\psi_2)\in\ml{A}_s:= (H^{s+4}\cap L^1)\times(H^{s+2}\cap L^1)\times(H^s\cap L^1)$ with $$\|(\psi_0,\psi_1,\psi_2)\|_{\ml{A}_s}:= \sum_{\ell=0}^2 \left(\|\psi_\ell\|_{H^{s+4-2\ell}} +\|\psi_\ell\|_{L^1}\right) \leqslant \epsilon,$$ there is a uniquely determined Sobolev solution
    	\begin{align*}
    		\psi\in\ml{C}([0,\infty),H^{s+4})\cap\ml{C}^1([0,\infty),H^{s+2})\cap \ml{C}^2([0,\infty),H^s)
    	\end{align*}
    	to the nonlinear Blackstock's model \eqref{Blackstock_Eq-New-Nonlinear}. Furthermore, the following estimates:
    	\begin{align*}
    		\|\partial_t^{\ell}\psi(t,\cdot)\|_{L^2}&\lesssim (1+t)^{\frac{2-\ell}{2}-\frac{n}{4}}\|(\psi_0,\psi_1,\psi_2)\|_{\ml{A}_s},\\
    		\|\partial_{t}^{\ell}\psi(t,\cdot)\|_{\dot{H}^{s+4-2\ell}}&\lesssim (1+t)^{-\frac{s+2-\ell}{2}-\frac{n}{4}}\|(\psi_0,\psi_1,\psi_2)\|_{\ml{A}_s},
    	\end{align*}
    	hold for all $\ell=0,1,2$.
    \end{theorem}
    \begin{remark}
    	The derived estimates of solutions to the nonlinear Blackstock's model coincide with those in Propositions \ref{Prop_Estimate_Cru_Lin}, \ref{Prop_Estimate_Cru_Lin_2} for the linearized Blackstock's model. That is to say we observe the effect of \emph{no loss of decay} to the corresponding linear problem.
    \end{remark}
    \begin{remark}
    	By following the same approach as the proof of Theorem \ref{Thm_GESDS}, we can demonstrate global (in time) existence of small data Sobolev solution to the nonlinear Blackstock's model for $n=3,4$, where the estimate for the $L^2$-norm of the solution itself is
    	\begin{align*}
    		\|\psi(t,\cdot)\|_{L^2}\lesssim \widetilde{\ml{D}}_n(1+t)\|(\psi_0,\psi_1,\psi_2)\|_{\ml{A}_s},
    	\end{align*}
    	for $n=3,4$, and the derivatives of solution satisfy the corresponding estimates in Theorem \ref{Thm_GESDS}.
    	%
    \end{remark}
    \subsection{Philosophy of the proof}
    For $T>0$, we introduce the evolution space
    \begin{align*}
    	X_s(T):=\ml{C}([0,T],H^{s+4})\cap\ml{C}^1([0,T],H^{s+2})\cap \ml{C}^2([0,T],H^s),
    \end{align*}
    carrying the corresponding norm
    \begin{align*}
    	\|\psi\|_{X_s(T)}:=&\sup\limits_{t\in[0,T]}\bigg(\sum\limits_{\ell=0,1,2}\Big((1+t)^{-\frac{2-\ell}{2}+\frac{n}{4}}\|\partial_t^{\ell}\psi(t,\cdot)\|_{L^2}\\
    	&\qquad\qquad\qquad\quad\ \ +(1+t)^{\frac{s+2-\ell}{2}+\frac{n}{4}}\|\,|D|^{s+4-2\ell}\partial_t^{\ell}\psi(t,\cdot)\|_{L^2}\Big)\bigg)
    \end{align*}
    for $s\in(0,\infty)$. Then, we define the operator $N$ such that
    \begin{align*}
    	N:\ \psi(t,x)\in X_s(T)\to N\psi(t,x):=\psi^{\lin}(t,x)+\psi^{\non}(t,x),
    \end{align*}
    where $\psi^{\lin}=\psi^{\lin}(t,x)$ is the solution to the linearized Cauchy problem \eqref{Linear_Blackstock_Eq}, and $\psi^{\mathrm{non}}$ denotes the integral operator
    \begin{align*}
    	\psi^{\non}(t,x):=\int_0^tK_2(t-\sigma,x)\ast_{(x)}F(\psi(\sigma,x);\partial_t,\nabla)\mathrm{d}\sigma,
    \end{align*}
    where $F$ is the nonlinear term in \eqref{Blackstock_Eq-New-Nonlinear} and $K_2(t,x)$ is the kernel for the third data of the linearized problem \eqref{Linear_Blackstock_Eq}, and its definition is motivated by Duhamel's principle. Namely,  $K_2$ is the distribution solution to the Cauchy problem \eqref{Linear_Blackstock_Eq} with initial data $(0,0,\delta_0)$, where $\delta_0$ denotes the Dirac distribution at $x=0$ with respect to spatial variables. Later, we will simply use the notation $F(\psi(\sigma,x))$ instead of $F(\psi(\sigma,x);\partial_t,\nabla)$.
    
    In the next parts, we will demonstrate the global (in time) existence and uniqueness of a Sobolev solution to the nonlinear Blackstock’s model for given small data by proving the existence of a uniquely determined fixed-point in $X_s(T)$ for the operator $N$. Indeed, the following two crucial inequalities:
    \begin{align}
    	\|N\psi\|_{X_s(T)}&\lesssim\|(\psi_0,\psi_1,\psi_2)\|_{\ml{A}_{s} }+\|\psi\|_{X_s(T)}^2,\label{Cruc-01}\\
    	\|N\psi-N\tilde{\psi}\|_{X_s(T)}&\lesssim\|\psi-\tilde{\psi}\|_{X_s(T)}\left(\|\psi\|_{X_s(T)}+\|\tilde{\psi}\|_{X_s(T)}\right),\label{Cruc-02}
    \end{align}
    will be proved, with the unexpressed multiplicative constants on the right-hand sides that are independent of $T$. In our aim inequality \eqref{Cruc-02}, $\psi$ and $\tilde{\psi}$ are two solutions to the nonlinear Blackstock's model \eqref{Blackstock_Eq-New-Nonlinear}. If we take $\|(\psi_0,\psi_1,\psi_2)\|_{\ml{A}_{s} }=\epsilon$ as a sufficiently small and positive constant, then we combine \eqref{Cruc-01} with \eqref{Cruc-02} to establish the global (in time) existence of a unique   small data Sobolev solution $\psi^{*}=\psi^{*}(t,x)\in X_s(T)$ by using Banach's fixed-point theorem.
    
    \subsection{Estimates of nonlinear terms}\label{Subsec_Est_Nonlinear_Term}
    Our purpose in this part is to give some estimates to the nonlinear terms in three norms as follows:
    \begin{align*}
    	\|F(\psi(\sigma,\cdot))\|_{L^1},\ \ \|F(\psi(\sigma,\cdot))\|_{L^2}, \ \ \|F(\psi(\sigma,\cdot))\|_{\dot{H}^s}
    \end{align*}
    for any $s\in(0,\infty)$, which will be controlled by $\|\psi\|_{X_s(T)}^2$ with suitable time-dependent coefficients for any $\sigma\in[0,T]$.
    \medskip
    
    \noindent\underline{Part I: Estimates in the $L^1$ and $L^2$ norms.}\\
    Here, let us fix $m=1,2$. Applying H\"older's inequality, we arrive at
    \begin{align*}
    	\|F(\psi(\sigma,\cdot))\|_{L^m}&\lesssim\|\nabla\psi_t(\sigma,\cdot)\|_{L^{2m}}^2+\|\nabla\psi(\sigma,\cdot)\cdot\nabla\Delta\psi(\sigma,\cdot)\|_{L^m}+\|\Delta\psi(\sigma,\cdot)\|_{L^{2m}}^2\\
    	&\quad+\|\psi_t(\sigma,\cdot)\Delta\psi_t(\sigma,\cdot)\|_{L^m}\\
    	&\lesssim\|\,|D|\psi_t(\sigma,\cdot)\|_{L^{2m}}^2+\|\,|D|\psi(\sigma,\cdot)\|_{L^{p_1}}\|\,|D|^3\psi(\sigma,\cdot)\|_{L^{p_2}}\\
    	&\quad+\|\,|D|^2\psi(\sigma,\cdot)\|_{L^{2m}}^2+\|\psi_t(\sigma,\cdot)\|_{L^{q_1}}\|\,|D|^2\psi_t(\sigma,\cdot)\|_{L^{q_2}},
    \end{align*}
    where $\frac{1}{p_1}+\frac{1}{p_2}=\frac{1}{q_1}+\frac{1}{q_2}=\frac{1}{m}$ with $p_1,p_2,q_1,q_2\in(1,\infty)$. Then, the fractional Gagliardo-Nirenberg inequality (see Proposition \ref{fractionalgagliardonirenbergineq}) shows
    \begin{align*}
    	\|\,|D|\psi_t(\sigma,\cdot)\|_{L^{2m}}^2&\lesssim\|\psi_t(\sigma,\cdot)\|_{L^2}^{2-\frac{2n}{s+2}(\frac{1}{2}-\frac{1}{2m}+\frac{1}{n})}\|\psi_t(\sigma,\cdot)\|_{\dot{H}^{s+2}}^{\frac{2n}{s+2}(\frac{1}{2}-\frac{1}{2m}+\frac{1}{n})}\\
    	&\lesssim(1+\sigma)^{-n+\frac{n}{2m}}\|\psi\|_{X_s(\sigma)}^{2},\\	\|\,|D|^2\psi(\sigma,\cdot)\|_{L^{2m}}^2&\lesssim\|\psi(\sigma,\cdot)\|_{L^2}^{2-\frac{2n}{s+4}(\frac{1}{2}-\frac{1}{2m}+\frac{2}{n})}\|\psi(\sigma,\cdot)\|_{\dot{H}^{s+4}}^{\frac{2n}{s+4}(\frac{1}{2}-\frac{1}{2m}+\frac{2}{n})}\\
    	&\lesssim(1+\sigma)^{-n+\frac{n}{2m}}\|\psi\|_{X_s(\sigma)}^{2},
    \end{align*}
    where we need $n\leqslant 4(s+1)$ if $m=2$. Moreover, one derives
    \begin{align*}
    	\|\,|D|\psi(\sigma,\cdot)\|_{L^{p_1}}\|\,|D|^3\psi(\sigma,\cdot)\|_{L^{p_2}}&\lesssim\|\psi(\sigma,\cdot)\|_{L^2}^{2-(\beta_1+\beta_2)}\|\psi(\sigma,\cdot)\|_{\dot{H}^{s+4}}^{\beta_1+\beta_2}\\
    	&\lesssim(1+\sigma)^{-n+\frac{n}{2m}}\|\psi\|_{X_s(\sigma)}^{2},\\
    	\|\psi_t(\sigma,\cdot)\|_{L^{q_1}}\|\,|D|^2\psi_t(\sigma,\cdot)\|_{L^{q_2}}&\lesssim\|\psi_t(\sigma,\cdot)\|_{L^2}^{2-(\beta_3+\beta_4)}\|\psi_t(\sigma,\cdot)\|_{\dot{H}^{s+2}}^{\beta_3+\beta_4}\\&
    	\lesssim(1+\sigma)^{-n+\frac{n}{2m}}\|\psi\|_{X_s(\sigma)}^{2},
    \end{align*}
    where $n\leqslant 4(s+1)$ if $m=2$, and
    \begin{align*}
    	\beta_1&:=\tfrac{n}{s+4}\left(\tfrac{1}{2}-\tfrac{1}{p_1}+\tfrac{1}{n}\right)\in\left[\tfrac{1}{s+4},1\right],\ \ \beta_2:=\tfrac{n}{s+4}\left(\tfrac{1}{2}-\tfrac{1}{p_2}+\tfrac{3}{n}\right)\in\left[\tfrac{3}{s+4},1\right],\\
    	\beta_3&:=\tfrac{n}{s+2}\left(\tfrac{1}{2}-\tfrac{1}{q_1}\right)\in\left[0,1\right],\  \quad\qquad \  \beta_4:=\tfrac{n}{s+2}\left(\tfrac{1}{2}-\tfrac{1}{q_2}+\tfrac{2}{n}\right)\in\left[\tfrac{2}{s+2},1\right].
    \end{align*}
    The restriction on $\beta_1,\dots,\beta_4$ will not bring any additional condition on $n$ and $s$. We can precisely compute the time-dependent coefficients for $\|\psi\|_{X_s(\sigma)}^2$ leading to
    \begin{align*}
    	\|F(\psi(\sigma,\cdot))\|_{L^m}\lesssim\begin{cases}
    		(1+\sigma)^{-\frac{n}{2}}\|\psi\|_{X_s(T)}^{2}&\mbox{if}\ \ m=1,\\
    		(1+\sigma)^{-\frac{3n}{4}}\|\psi\|_{X_s(T)}^{2}&\mbox{if}\ \ m=2,
    	\end{cases} 
    \end{align*}
    with $n\leqslant 4(s+1)$ if $m=2$. In the above, we employed $\|\psi\|_{X_s(\sigma)}\lesssim\|\psi\|_{X_s(T)}$ for any $\sigma\in[0,T]$.
    
    \medskip
    
    \noindent\underline{Part II: Estimates in the $\dot{H}^s$ norms.}\\
    Basically, we can divide the nonlinear term into four parts as follows:
    \begin{align*}
    	\|F(\psi(\sigma,\cdot))\|_{\dot{H}^s}&\lesssim\|\,|\nabla\psi_t(\sigma,\cdot)|^2\|_{\dot{H}^s}+\|\nabla\psi(\sigma,\cdot)\cdot\nabla\Delta\psi(\sigma,\cdot)\|_{\dot{H}^s}+\|\,|\Delta\psi(\sigma,\cdot)|^2\|_{\dot{H}^s}\\
    	&\quad+\|\psi_t(\sigma,\cdot)\Delta\psi_t(\sigma,\cdot)\|_{\dot{H}^s}\\
    	&=:\ml{I}_1(\sigma)+\ml{I}_2(\sigma)+\ml{I}_3(\sigma)+\ml{I}_4(\sigma).
    \end{align*}
    Now, to estimate the nonlinear terms in the Riesz potential spaces, we will employ the fractional Gagliardo-Nirenberg inequality (Proposition \ref{fractionalgagliardonirenbergineq}) and the fractional Leibniz rule (Proposition \ref{fractionleibnizrule}) instead of H\"older's inequality. To be specific, one sees
    \begin{align*}
    	\ml{I}_1(\sigma)&\lesssim\|\,|D|\psi_t(\sigma,\cdot)\|_{\dot{H}^{s}_{p_3}}\|\,|D|\psi_t(\sigma,\cdot)\|_{L^{p_4}}\lesssim\|\psi_t(\sigma,\cdot)\|_{L^2}^{2-(\beta_5+\beta_6)}\|\psi_t(\sigma,\cdot)\|_{\dot{H}^{s+2}}^{\beta_5+\beta_6}\\
    	&\lesssim(1+\sigma)^{-\frac{3n}{4}-\frac{s}{2}}\|\psi\|^2_{X_s(T)},
    \end{align*}
    with $\frac{1}{p_3}+\frac{1}{p_4}=\frac{1}{2}$, where the parameters
    \begin{align}\label{beta 5,6}
    	\beta_5&:=\tfrac{n}{s+2}\left(\tfrac{1}{2}-\tfrac{1}{p_3}+\tfrac{s+1}{n}\right)\in\left[ \tfrac{s+1}{s+2},1 \right],\notag\\
    	\beta_6&:=\tfrac{n}{s+2}\left(\tfrac{1}{2}-\tfrac{1}{p_4}+\tfrac{1}{n}\right)\in\left[\tfrac{1}{s+2},1\right],
    \end{align}
    immediately conclude
    \begin{align*}
    	\beta_5+\beta_6=\tfrac{n}{2(s+2)}+1\in[1,2]\ \ \Rightarrow\ \ n\leqslant 2s+4.
    \end{align*}
    \begin{remark}
    	In the previous step we showed that the assumption $s\geqslant \frac{n}{2}-2$ is necessary for the existence of suitable $\beta_5,\beta_6$. Actually, this condition is sufficient as well. Indeed, the existence of $\beta_5,\beta_6$ as in \eqref{beta 5,6} is equivalent to require that $p_3,p_4$ satisfy
    	\begin{align*}
    		0\leqslant \tfrac{1}{2}-\tfrac{1}{p_3}\leqslant \tfrac{1}{n}, \quad 0\leqslant \tfrac{1}{2}-\tfrac{1}{p_4}\leqslant \tfrac{s+1}{n}. 
    	\end{align*} Nevertheless, $p_3,p_4$ are linked through the relation $\frac{1}{p_3}+\frac{1}{p_4}=\frac{1}{2}$. Therefore, the previous inequalities are fulfilled provided that $p_3$ guarantees
    	\begin{align*}
    		\max\left\{0,\tfrac{1}{2}-\tfrac{1}{n}\right\}\leqslant \tfrac{1}{p_3}\leqslant \min\left\{\tfrac{1}{2},\tfrac{s+1}{n}\right\}.
    	\end{align*} By a straightforward computation, we see that this range for $1/p_3$ is not empty provided that $s\geqslant \frac{n}{2}-2$. Analogously, this lower bound for $s$ ensures the existence of the other $\beta_j$ in the proof.
    \end{remark}
    Similarly, it holds 
    \begin{align*}
    	\ml{I}_2(\sigma)&\lesssim\|\,|D|\psi(\sigma,\cdot)\|_{\dot{H}^s_{p_5}}\|\,|D|^3\psi(\sigma,\cdot)\|_{L^{p_6}}+\|\,|D|\psi(\sigma,\cdot)\|_{L^{q_5}}\|\,|D|^3\psi(\sigma,\cdot)\|_{\dot{H}^s_{q_6}}\\
    	&\lesssim \|\psi(\sigma,\cdot)\|_{L^2}^{2-(\beta_7+\beta_8)}\|\psi(\sigma,\cdot)\|_{\dot{H}^{s+4}}^{\beta_7+\beta_8}+\|\psi(\sigma,\cdot)\|_{L^2}^{2-(\beta_9+\beta_{10})}\|\psi(\sigma,\cdot)\|_{\dot{H}^{s+4}}^{\beta_9+\beta_{10}}\\
    	&\lesssim (1+\sigma)^{-\frac{3n}{4}-\frac{s}{2}}\|\psi\|^2_{X_s(T)},
    \end{align*}
    with $\frac{1}{p_5}+\frac{1}{p_6}=\frac{1}{q_5}+\frac{1}{q_6}=\frac{1}{2}$, where the parameters
    \begin{align*}
    	\beta_7&:=\tfrac{n}{s+4}\left(\tfrac{1}{2}-\tfrac{1}{p_5}+\tfrac{s+1}{n}\right)\in\left[ \tfrac{s+1}{s+4},1 \right],\ \ \beta_8:=\tfrac{n}{s+4}\left(\tfrac{1}{2}-\tfrac{1}{p_6}+\tfrac{3}{n}\right)\in\left[\tfrac{3}{s+4},1\right],\\
    	\beta_9&:=\tfrac{n}{s+4}\left(\tfrac{1}{2}-\tfrac{1}{q_5}+\tfrac{1}{n}\right)\in\left[ \tfrac{1}{s+4},1 \right], \  \ \ \ \beta_{10}:=\tfrac{n}{s+4}\left(\tfrac{1}{2}-\tfrac{1}{q_6}+\tfrac{s+3}{n}\right)\in\left[\tfrac{s+3}{s+4},1\right],
    \end{align*}
    will not bring any additional condition. Analogously to the above discussions, we also can get
    \begin{align*}
    	\ml{I}_3(\sigma)\lesssim (1+\sigma)^{-\frac{3n}{4}-\frac{s}{2}}\|\psi\|^2_{X_s(T)}.
    \end{align*}
    Finally, we estimate the later term
    \begin{align*}
    	\ml{I}_4(\sigma)&\lesssim \|\psi_t(\sigma,\cdot)\|_{\dot{H}^s_{p_7}}\|\,|D|^2\psi_t(\sigma,\cdot)\|_{L^{p_8}}+\|\psi_t(\sigma,\cdot)\|_{L^{\infty}}\|\,|D|^2\psi_t(\sigma,\cdot)\|_{\dot{H}^s},\\
    	&\lesssim \|\psi_t(\sigma,\cdot)\|_{L^2}^{2-(\beta_{11}+\beta_{12})}\|\psi_t(\sigma,\cdot)\|_{\dot{H}^{s+2}}^{\beta_{11}+\beta_{12}}+\ml{I}_5(\sigma)\\
    	&\lesssim (1+\sigma)^{-\frac{3n}{4}-\frac{s}{2}}\|\psi\|^2_{X_s(T)}+\ml{I}_5(\sigma),
    \end{align*}
    with $\frac{1}{p_7}+\frac{1}{p_8}=\frac{1}{2}$, where the parameters
    \begin{align*}
    	\beta_{11}:=\tfrac{n}{s+2}\left(\tfrac{1}{2}-\tfrac{1}{p_7}+\tfrac{s}{n}\right)\in\left[ \tfrac{s}{s+2},1 \right],\ \ \beta_{12}:=\tfrac{n}{s+2}\left(\tfrac{1}{2}-\tfrac{1}{p_8}+\tfrac{2}{n}\right)\in\left[\tfrac{2}{s+2},1\right].
    \end{align*}
    \begin{remark}
    	We just know the estimates for $\psi_t(t,\cdot)$ in the $\dot{H}^{s+2}$ from the evolution space $X_s(T)$. It is better to control the previous norm with $\ml{I}_5(\sigma)$ rather than $\|\psi_t(\sigma,\cdot)\|_{L^{q_7}}\|\psi_t(\sigma,\cdot)\|_{\dot{H}^{s+2}_{q_8}}$ for $\frac{1}{q_7}+\frac{1}{q_8}=\frac{1}{2}$ and $q_7\neq\infty$, because we have to ``improve'' the regularity to estimate the norm $\|\psi_t(\sigma,\cdot)\|_{\dot{H}^{s+2}_{q_8}}$ by the application of the fractional Gagliardo-Nirenberg inequality. However, we cannot control $\|\psi_t(\sigma,\cdot)\|_{\dot{H}^{s+2+s_1}}$ for any $s_1>0$.
    \end{remark}
    \noindent To deal with $\ml{I}_5(\sigma)$, we apply the fractional Sobolev embedding (Proposition \ref{fractionembedd}) to get
    \begin{align*}
    	\ml{I}_5(\sigma)&\lesssim\left(\|\psi_t(\sigma,\cdot)\|_{\dot{H}^{s^*}}+\|\psi_t(\sigma,\cdot)\|_{\dot{H}^s}\right)\|\psi_t(\sigma,\cdot)\|_{\dot{H}^{s+2}}\\
    	&\lesssim\left(\|\psi_t(\sigma,\cdot)\|_{L^2}^{1-\frac{s^*}{s+2}}\|\psi_t(\sigma,\cdot)\|_{\dot{H}^{s+2}}^{\frac{s^*}{s+2}}+\|\psi_t(\sigma,\cdot)\|_{L^2}^{1-\frac{s}{s+2}}\|\psi_t(\sigma,\cdot)\|_{\dot{H}^{s+2}}^{\frac{s}{s+2}} \right)\\
    	&\quad\ \times\|\psi_t(\sigma,\cdot)\|_{\dot{H}^{s+2}}\\
    	&\lesssim\left((1+\sigma)^{-\frac{n}{2}-\frac{s^*+s}{2}}+(1+\sigma)^{-\frac{n}{2}-s}\right)\|\psi\|^2_{X_s(T)}\\
    	&\lesssim (1+\sigma)^{-\frac{n}{2}-\frac{s^*+s}{2}}\|\psi\|^2_{X_s(T)},
    \end{align*}
    where $0<2s^*<n<2s$. It is noteworthy that the choice of $s^*$ can enhance the decay rate in $\ml{I}_5(\sigma)$, and it gives some freedoms to the restriction between $s$ and $n$ later. Let $\epsilon_0>0$ to be a sufficient small constant. We choose $s^*=n/2-2\epsilon_0$ to determine
    \begin{align*}
    	\ml{I}_5(\sigma)\lesssim (1+\sigma)^{-\frac{3n}{4}-\frac{s}{2}+\epsilon_0}\|\psi\|^2_{X_s(T)}.
    \end{align*}
    By such Sobolev embedding with $s^*$, we just lose $(1+\sigma)^{\epsilon_0}$ as decay rate, with sufficiently small $\epsilon_0>0$. Thus, the summary of the previous estimates is
    \begin{align*}
    	\|F(\psi(\sigma,\cdot))\|_{\dot{H}^s}\lesssim (1+\sigma)^{-\frac{3n}{4}-\frac{s}{2}+\epsilon_0}\|\psi\|^2_{X_s(T)}.
    \end{align*}
    \subsection{Existence of unique global (in time) solution}
    First of all, according to the derived estimates in Proposition \ref{Prop_Estimate_Cru_Lin}, we claim that
    \begin{align*}
    	\|\psi^{\lin}\|_{X_s(T)}\lesssim\|(\psi_0,\psi_1,\psi_2)\|_{\ml{A}_s}.
    \end{align*}
    In other words, concerning \eqref{Cruc-01}, we just need to estimate $\|\psi^{\non}\|_{X_s(T)}$ in the rest of this subsection.
    
    By using the $(L^2\cap L^1)-L^2$ estimates derived in Proposition \ref{Prop_Estimate_Cru_Lin} for $[0,t]$, we may see
    \begin{align*}
    	\|\partial_t^{\ell}\psi^{\non}(t,\cdot)\|_{L^2}&\lesssim\int_0^t(1+t-\sigma)^{\frac{2-\ell}{2}-\frac{n}{4}}\|F(\psi(\sigma,\cdot))\|_{L^2\cap L^1}\mathrm{d}\sigma,\\
    	&\lesssim\int_0^t(1+t-\sigma)^{\frac{2-\ell}{2}-\frac{n}{4}}(1+\sigma)^{-\frac{n}{2}}\mathrm{d}\sigma\,\|\psi\|_{X_s(T)}^2\\
    	&\lesssim(1+t)^{-\min\{\frac{\ell-2}{2}+\frac{n}{4},\frac{n}{2}\}}\|\psi\|_{X_s(T)}^2\\
    	&=(1+t)^{\frac{2-\ell}{2}-\frac{n}{4}}\|\psi\|_{X_s(T)}^2,
    \end{align*}
    since $\ell<2+n/2$ for all $\ell=0,1,2$, where $n\leqslant 4(s+1)$ and we used 
    \begin{align}\label{control int alpha 1,2}
    	\int_0^t(1+t-\sigma)^{-\alpha_1}(1+\sigma)^{-\alpha_2}\mathrm{d}\sigma\lesssim(1+t)^{-\min\{\alpha_1,\alpha_2\}}\ \ \mbox{if}\ \ \max\{\alpha_1,\alpha_2\}>1,
    \end{align}
    for positive constants $\alpha_1$ and $\alpha_2$. For the proof of \eqref{control int alpha 1,2} we refer to \cite{Aassila-2003}. Applying the $(L^2\cap L^1)-L^2$ estimates derived in Proposition \ref{Prop_Estimate_Cru_Lin} again in $[0,t/2]$ and the $L^2-L^2$ estimates obtained in Proposition \ref{Prop_Estimate_Cru_Lin_2} in $[t/2,t]$, they yield
    \begin{align*}
    	\|\partial_t^{\ell}\psi^{\non}(t,\cdot)\|_{\dot{H}^{s+4-2\ell}}&\lesssim\int_0^{t/2}(1+t-\sigma)^{-\frac{s+2-\ell}{2}-\frac{n}{4}}\|F(\psi(\sigma,\cdot))\|_{\dot{H}^s\cap L^1}\mathrm{d}\sigma\\
    	&\quad+\int_{t/2}^t(1+t-\sigma)^{-\frac{2-\ell}{2}}\|F(\psi(\sigma,\cdot))\|_{\dot{H}^s}\mathrm{d}\sigma\\
    	&\lesssim(1+t)^{-\frac{s+2-\ell}{2}-\frac{n}{4}}\int_0^{t/2}(1+\sigma)^{-\frac{n}{2}}\mathrm{d}\sigma\,\|\psi\|_{X_s(T)}^2\\
    	&\quad+(1+t)^{-\frac{3n}{4}-\frac{s}{2}+\epsilon_0}\int_{t/2}^t(1+t-\sigma)^{-\frac{2-\ell}{2}}\mathrm{d}\sigma\,\|\psi\|_{X_s(T)}^2\\
    	&\lesssim(1+t)^{-\frac{s+2-\ell}{2}-\frac{n}{4}}\|\psi\|_{X_s(T)}^2,
    \end{align*}
    where $n\leqslant 2s+4$, due to the facts that
    \begin{align*}
    	&(1+\sigma)^{-\frac{n}{2}}\in L^1([0,t/2]),\\
    	& (1+t)^{-\frac{3n}{4}-\frac{s}{2}+\epsilon_0}\int_{t/2}^t(1+t-\sigma)^{-\frac{2-\ell}{2}}\mathrm{d}\sigma\lesssim (1+t)^{-\frac{s+2-\ell}{2}-\frac{n}{4}},
    \end{align*}
    for all $\ell=0,1,2$.  All in all, summarizing the above derived estimates, we directly conclude
    \begin{align*}
    	\|\psi^{\non}\|_{X_s(T)}\lesssim\|\psi\|_{X_s(T)}^2,
    \end{align*}
    which says that the desired estimate \eqref{Cruc-01} holds if $s\in[n/2-2,\infty)$ for all $n\geqslant 5$.
    
    Let us sketch the proof of the Lipschitz condition by remarking
    \begin{align*}
    	&\|\,|D|^{s+4-2\ell}\partial_t^{\ell}[ N\psi(t,\cdot)-N\tilde{\psi}(t,\cdot)]\|_{L^2}\\
    	&\qquad=\left\|\,|D|^{s+4-2\ell}\partial_t^{\ell}\int_0^tK_2(t-\sigma,\cdot)\ast_{(x)}\left(F(\psi(\sigma,\cdot))-F(\tilde{\psi}(\sigma,\cdot))\right)\mathrm{d}\sigma\right\|_{L^2}
    \end{align*}
    for $s\in\{2\ell-4\}\cup[0,\infty)$ and $\ell=0,1,2$. By using some estimates in Propositions \ref{Prop_Estimate_Cru_Lin} and \ref{Prop_Estimate_Cru_Lin_2} again, we just need to estimate
    \begin{align*}
    	&\|F(\psi(\sigma,\cdot))-F(\tilde{\psi}(\sigma,\cdot))\|_{L^1},\ \ \|F(\psi(\sigma,\cdot))-F(\tilde{\psi}(\sigma,\cdot))\|_{L^2},\\
    	&\qquad\qquad\qquad\quad\|F(\psi(\sigma,\cdot))-F(\tilde{\psi}(\sigma,\cdot))\|_{\dot{H}^s}.
    \end{align*}
    Therefore, by repeating the same steps as those in Subsection \ref{Subsec_Est_Nonlinear_Term}, concerning $m=1,2$, we obtain
    \begin{align*}
    	\|F(\psi(\sigma,\cdot))-F(\tilde{\psi}(\sigma,\cdot))\|_{L^m}\lesssim(1+\sigma)^{-n+\frac{n}{2m}}\|\psi-\tilde{\psi}\|_{X_s(T)}\left(\|\psi\|_{X_s(T)}+\|\tilde{\psi}\|_{X_s(T)}\right).
    \end{align*}
    Furthermore, the next estimates hold:
    \begin{align*}
    	\|F(\psi(\sigma,\cdot))-F(\tilde{\psi}(\sigma,\cdot))\|_{\dot{H}^s}&\lesssim(1+\sigma)^{-\frac{3n}{4}-\frac{s}{2}+\epsilon_0}\|\psi-\tilde{\psi}\|_{X_s(T)}\\
    	&\quad\ \times\left(\|\psi\|_{X_s(T)}+\|\tilde{\psi}\|_{X_s(T)}\right)
    \end{align*}
    for $s\in[n/2-2,\infty)$, $n\geqslant 5$ and with sufficiently small $\epsilon_0>0$. 
    More specifically, in order to get the previous estimates for the difference between the nonlinear terms $F(\psi)-F(\tilde{\psi})$, it is convenient to employ the following integral representations:
    \begin{align*}
    	\nabla \psi \cdot \nabla \Delta \psi - \nabla \tilde{\psi}\cdot \nabla \Delta \tilde{\psi}  & =  \int_0^1   \nabla (\psi -\tilde{\psi}) \cdot \nabla \Delta (\omega \psi +(1-\omega)  \tilde{\psi})  \mathrm{d}\omega \notag  \\ & \quad + \int_0^1 \nabla \Delta (\psi -\tilde{\psi}) \cdot \nabla  (\omega \psi +(1-\omega) \tilde{\psi})  \mathrm{d}\omega,
    \end{align*}
    and
    \begin{align*}
    	|\nabla \psi_t|^2-|\nabla \tilde{\psi}_t|^2 & = 2 \int_0^1  \nabla (\psi_t - \tilde{\psi}_t) \cdot \nabla (\omega \psi_t +(1-\omega)  \tilde{\psi}_t)  \mathrm{d}\omega, \\
    	|\Delta \psi|^2-|\Delta \tilde{\psi}|^2 & = 2 \int_0^1  \Delta (\psi - \tilde{\psi}) \Delta (\omega \psi +(1-\omega)  \tilde{\psi})  \mathrm{d}\omega ,\\
    	\psi_t  \Delta \psi_t - \tilde{\psi}_t \Delta \tilde{\psi} _t & =  \int_0^1  \left(  (\psi_t -\tilde{\psi}_t) \Delta (\omega \psi_t +(1-\omega)  \tilde{\psi}_t)\right.\\
    	&\qquad\quad\ \ \ \left. +\Delta (\psi_t -\tilde{\psi}_t)  (\omega \psi_t +(1-\omega)  \tilde{\psi}_t) \right)  \mathrm{d}\omega.
    \end{align*} 
    Then, employing H\"older's inequality (for the estimates in $L^m$, $m=1,2$), the fractional Gagliardo-Nirenberg inequality, the fractional Leibniz rule and the fractional Sobolev embedding  (for the estimate in $\dot{H}^s$), we may repeat similar steps to those employed in the proof of \eqref{Cruc-01}. In particular, the $\omega$-dependent terms in the above integral terms provide a $\| \omega \psi+(1-\omega)\tilde{\psi}\|_{X_s(T)}$ factor (times suitable $\sigma$-dependents weights), which can be uniformly estimate by $\| \psi\|_{X_s(T)}+ \| \tilde{\psi}\|_{X_s(T)}$ in the $\omega$-integral, while the factors in the integral representations that are independent of $\omega$ provide the $\|  \psi-\tilde{\psi}\|_{X_s(T)}$ term (times suitable $\sigma$-dependents weights). Finally, we emphasize that in the proof of \eqref{Cruc-02} the convergence of the $\sigma$-integrals is ensured requiring exactly the same assumptions on $n$ and $s$ as in the proof of \eqref{Cruc-01}. The proof of Theorem \ref{Thm_GESDS} is complete.
    
    \section{Concluding remarks: Inviscid case for Blackstock's model}
    Throughout this paper, we concentrated on the Blackstock's model, which can be explained by the approximated equation modeled by small perturbation in the Navier-Stokes-Fourier coupled system \cite{Brunnhuber-2015,Brunnhuber-2016}. To end this paper, we will consider the approximated equation modeled by small perturbation in the Euler-Fourier coupled system by taking vanishing viscous coefficients in \eqref{Linear_Blackstock_Eq}.
    
    By ignoring the viscous effect in the conservations of momentum, i.e. the application of Euler equations instead of Navier-Stokes equations, we naturally consider acoustic waves in the inviscid and irrotational flow. Let us now denote the viscous coefficient by $\bar{\epsilon}:=b\nu$. In this case, it implies $\bar{\epsilon}=0$ and $\delta=(\gamma-1)\kappa$ in linear Blackstock's model \eqref{Linear_Blackstock_Eq}, namely,
    \begin{align}\label{Linear_Inviscid_Blackstock}
    	\begin{cases}
    		(\partial_t-\kappa\Delta)(\phi_{tt}-\Delta\phi-(\gamma-1)\kappa\Delta\phi_t)-\kappa^2(\gamma-1)\Delta^2\phi_t=0,\\
    		\phi(0,x)=\phi_0(x),\ \phi_t(0,x)=\phi_1(x),\ \phi_{tt}(0,x)=\phi_2(x),
    	\end{cases}
    \end{align}
    for $x\in\mb{R}^n$, $t>0$. Clearly, with the help of partial Fourier transform, the characteristic roots satisfy
    \begin{align*}
    	\mu_j^3+\gamma\kappa|\xi|^2\mu_j^2+|\xi|^2\mu_j+\kappa|\xi|^4=0
    \end{align*}
    for $j=1,2,3$. By following the similar approaches to those in Subsection \ref{Sub_sec_asymptotic_behavior}, we gain
    \begin{itemize}
    	\item $\mu_{1/2}=\pm i|\xi|-\frac{(\gamma-1)\kappa}{2}|\xi|^2+\ml{O}(|\xi|^3)$, $\mu_3=-\kappa|\xi|^2+\ml{O}(|\xi|^3)$ for $|\xi|\to0$;
    	\item $\mu_{1/2}=\pm i\sqrt{\frac{1}{\gamma}}|\xi|-\frac{\gamma-1}{2\gamma^2\kappa}+\ml{O}(|\xi|^{-1})$, $\mu_3=-\gamma\kappa|\xi|^2+\ml{O}(|\xi|)$ for $|\xi|\to\infty$;
    	\item $\Re \mu_j<0$ with $j=1,2,3$ for $|\xi|\not\to0$ and $|\xi|\not\to\infty$.
    \end{itemize}
    Here, we observe that the characteristic roots in the inviscid case for $|\xi|\to0$ correspond to those in the viscous case (see Part I of Subsection \ref{Sub_sec_asymptotic_behavior}). However, comparing with Part II of Subsection \ref{Sub_sec_asymptotic_behavior}, the asymptotic behavior of the characteristic roots for $|\xi|\to\infty$ have changed. Precisely, the dominant terms with $-\frac{1}{2}\left((\delta+\kappa)\pm\sqrt{(\delta+\kappa)^2-4\gamma b\nu\kappa}\right)|\xi|^2$ in the viscous case have been changed into $-\frac{\gamma-1}{2\gamma^2\kappa}$ in the inviscid case, which provides an increment in the regularity assumptions for the last two data. Then, the following pointwise estimates are valid:
    \begin{align*}
    	\chi_{\intt}(\xi)|\widehat{\phi}|& \lesssim\chi_{\intt}(\xi)\mathrm{e}^{-c|\xi|^2t}\left(|\widehat{\phi}_0|+\left(1+|\xi|^{-1}|\sin(|\xi|t)|\right)|\widehat{\phi}_1|+|\xi|^{-2}|\widehat{\phi}_2|\right),\\
    	\chi_{\extt}(\xi)|\widehat{\phi}|& \lesssim\chi_{\extt}(\xi)\mathrm{e}^{-ct}\left(|\widehat{\phi}_0|+|\xi|^{-1}|\widehat{\phi}_1|+|\xi|^{-3}|\widehat{\phi}_2|\right),\\
    	\chi_{\bdd}(\xi)|\widehat{\phi}|&\lesssim\chi_{\bdd}(\xi)\mathrm{e}^{-ct}\left(|\widehat{\phi}_0|+|\widehat{\phi}_1|+|\widehat{\phi}_2|\right),
    \end{align*}
    with positive constants $c$.
    
    We are able to derive exactly the same optimal $L^2$ estimate for the solution itself $\phi(t,\cdot)$ as those in Theorem \ref{Thm_Optimal} at least for $n\geqslant 5$. The main reason is that the dominant influence from Fourier's law of heat conduction exerts thermal effect in the sense of damping. Nonetheless, the requirement of regularity for the viscous case and the inviscid case are different. Due to the fact that the regularities of initial data are mainly determined by estimates for large frequencies, we carry out the next analyses for any $s\in[0,\infty)$:
    \begin{itemize}
    	\item concerning viscous Blackstock's model \eqref{Linear_Blackstock_Eq}, we have obtained
    	\begin{align*}
    		\|\chi_{\extt}(D)\psi(t,\cdot)\|_{\dot{H}^s}\lesssim\mathrm{e}^{-ct}\left(\|\psi_0\|_{H^s}+\|\psi_1\|_{H^{\max\{s-2,0\}}}+\|\psi_2\|_{H^{\max\{s-4,0\}}}\right);
    	\end{align*}
    \item concerning inviscid Blackstock's model \eqref{Linear_Inviscid_Blackstock}, we can obtain
    \begin{align*}
    	\|\chi_{\extt}(D)\phi(t,\cdot)\|_{\dot{H}^s}\lesssim\mathrm{e}^{-ct}\left(\|\phi_0\|_{H^s}+\|\phi_1\|_{H^{\max\{s-1,0\}}}+\|\phi_2\|_{H^{\max\{s-3,0\}}}\right).
    \end{align*}
    \end{itemize}
    Summarizing the previous two estimates, one may easily notice the influence of  the vanishing viscosity  that we will lose some regularities for the second and third initial data. In particular, one order of regularity will be lost for $s\geqslant 4$.
    
    Let us now turn to the viscous limit as the viscosity coefficient $\bar{\epsilon}$ tending to zero. Actually, the difference between the solutions of the viscous case \eqref{Linear_Blackstock_Eq} and of the inviscid case \eqref{Linear_Inviscid_Blackstock} that is $\tilde{u}:=\psi-\phi$ fulfills the inhomogeneous Blackstock's model
    \begin{align*}
    	\begin{cases}
    		(\partial_t-\kappa\Delta)(\tilde{u}_{tt}-\Delta \tilde{u}-\delta\Delta \tilde{u}_t)+\kappa(\gamma-1)(b\nu-\kappa)\Delta^2\tilde{u}_t=\bar{\epsilon}(\Delta\phi_{tt}-\kappa\gamma\Delta^2\phi_t),\\
    		\tilde{u}(0,x)=0,\ \tilde{u}_t(0,x)=0,\ \tilde{u}_{tt}(0,x)=0,
    	\end{cases}
    \end{align*}
    for $x\in\mb{R}^n$, $t>0$, if we take consistency assumption that $\psi_j\equiv\phi_j$ for any $j=0,1,2$. The Fourier image of $\tilde{u}$ can be estimated by
    \begin{align*}
    	|\widehat{\tilde{u}}_t|&\leqslant \bar{\epsilon}\, C\Big(\int_0^t\left(|\xi|^4|\widehat{\phi}_t(s,\xi)|^2+|\widehat{\phi}_{tt}(s,\xi)|^2\right)\mathrm{d}s\Big)^{1/2},\\
    	|\widehat{\tilde{u}}|&\leqslant \bar{\epsilon}\, C\big(\int_0^t\left(|\xi|^2\,|\widehat{\phi}_t(s,\xi)|^2+|\xi|^{-2}\,|\widehat{\phi}_{tt}(s,\xi)|^2\right)\mathrm{d}s\Big)^{1/2},
    \end{align*}
    where we used Propositions \ref{Prop_Energy_Fourier_01} and \ref{Prop_Potential_1} with $\hat{f}=\bar{\epsilon}(\kappa\gamma|\xi|^4\hat{\phi}_t+|\xi|^2\hat{\phi}_{tt})$ as well as $\hat{u}_2\equiv0$. Therefore, by using some decay estimates for the time-derivative of solutions, we conjecture 
    \begin{align*}
    	\psi_t\to\phi_t\ \ \mbox{and}\ \ \psi\to\phi\ \ \mbox{in}\ \ L^{\infty}([0,\infty)\times\mb{R}^n)\ \ \mbox{as}\ \ \bar{\epsilon}\downarrow0,
    \end{align*}
    with the rate of convergence $\bar{\epsilon}$ for some dimensions. In other words, as the bulk and kinematic viscosity coefficients tend to zero, the solution to the Blackstock's model for the viscous case \eqref{Linear_Blackstock_Eq} converges to one for the inviscid case \eqref{Linear_Inviscid_Blackstock} with the rate $\frac{4}{3}\mu_{\mathrm{V}}+\mu_{\mathrm{B}}$.
    
    \appendix
    \section{Tools from harmonic analysis}\label{Appendix_A}
    In this section, we will introduce interpolation theorem from harmonic analysis, which have been used in Section \ref{Sec_GESDS} to deal with the nonlinear terms.
    
    \begin{prop}\label{fractionalgagliardonirenbergineq} (Fractional Gagliardo-Nirenberg inequality)
    	Let $p,p_0,p_1\in(1,\infty)$ and $\kappa\in[0,s)$ with $s\in(0,\infty)$. Then, for all $f\in L^{p_0}\cap \dot{H}^{s}_{p_1}$ the following inequality holds:
    	\begin{equation*}
    		\|f\|_{\dot{H}^{\kappa}_{p}}\lesssim\|f\|_{L^{p_0}}^{1-\beta}\|f\|^{\beta}_{\dot{H}^{s}_{p_1}},
    	\end{equation*}
    	where 
    	\begin{align*}	\beta=\left(\tfrac{1}{p_0}-\tfrac{1}{p}+\tfrac{\kappa}{n}\right)/\left(\tfrac{1}{p_0}-\tfrac{1}{p_1}+\tfrac{s}{n}\right)\ \ \mbox{and}	\ \ \beta\in\left[\tfrac{\kappa}{s},1\right].
    	\end{align*}
    \end{prop}
    The proof of the fractional Gagliardo-Nirenberg inequality can be found in \cite{Hajaiej-Molinet-Ozawa-Wang-2011}.
    
    \begin{prop}\label{fractionleibnizrule} (Fractional Leibniz rule)
    	Let $s\in(0,\infty)$, $r\in[1,\infty)$, $p_1,p_2,q_1,q_2\in(1,\infty]$ satisfy the relation
    	\begin{equation*}
    		\tfrac{1}{r}=\tfrac{1}{p_1}+\tfrac{1}{p_2}=\tfrac{1}{q_1}+\tfrac{1}{q_2}.
    	\end{equation*}
    	Then, for all $f\in\dot{H}^{s}_{p_1}\cap L^{q_1}$ and $g\in\dot{H}^{s}_{q_2}\cap L^{q_2}$
    	the following inequality holds:
    	\begin{equation*}
    		\|fg\|_{\dot{H}^{s}_{r}}\lesssim \|f\|_{\dot{H}^{s}_{p_1}}\|g\|_{L^{p_2}}+\|f\|_{L^{q_1}}\|g\|_{\dot{H}^{s}_{q_2}}.
    	\end{equation*}
    \end{prop}
    The proof of the last inequality can be found in \cite{Grafakos-Oh-2014}.
    
    \begin{prop}\label{fractionembedd} (Fractional Sobolev embedding) Let $0<2s^*<n<2s$. Then, for any function $f\in\dot{H}^{s^*}\cap\dot{H}^s$ one has
    	\begin{equation*}
    		\|f\|_{L^{\infty}}\lesssim\|f\|_{\dot{H}^{s^*}}+\|f\|_{\dot{H}^s}.
    	\end{equation*}
    \end{prop}
    The proof of this embedding result was given in \cite{Dabbicco-Ebert-Lucente-2017}.
    \section*{Acknowledgment}
    The first author was supported by the China Postdoctoral Science Foundation (Grant No. 2021T140450 and No. 2021M692084). The second author was supported in part by Grant-in-Aid for scientific Research (C) 20K03682 of JSPS. The authors thank Yaguang Wang (Shanghai Jiao Tong University), Xiang Wang (Shanghai Jiao Tong University) for the suggestions in Section \ref{Sec_Initial_Layer}. The authors thank the anonymous referees for carefully reading the paper and giving some useful suggestions.

    \bibliographystyle{alpha}
\bibliography{mybib}

\end{document}